\documentclass[11pt,a4paper,reqno]{amsart}
\fontsize{6.5pt}{6.5pt}\selectfont 
\makeatletter
\def\@evenhead{{\fontsize{6.5pt}{6.5pt}\selectfont \hfil \leftmark\hfil\thepage}}
\def\@oddhead{{\fontsize{6.5pt}{6.5pt}\selectfont\hfil\rightmark \hfil\thepage}}
\makeatother
\usepackage{amsfonts}
\usepackage{booktabs}
\usepackage{amsmath,amssymb,amsthm,amsxtra}
\usepackage{float}
\usepackage{amssymb}
\usepackage{mathabx}
\usepackage{bbm}
\usepackage[usenames, dvipsnames]{xcolor}
\usepackage[colorlinks, linkcolor=blue,anchorcolor=Periwinkle,
citecolor=Red,urlcolor=Emerald]{hyperref}
\usepackage{enumitem}
\usepackage{geometry,array} 
\usepackage{graphicx}
\usepackage{subfigure}
\usepackage{bookmark}
\usepackage{tikz}\usetikzlibrary{matrix,intersections, calc, arrows.meta, math}
\usepackage{url}
\usepackage{dsfont}
\usepackage[colorinlistoftodos]{todonotes}
\usepackage{tablists} \restorelistitem
\usepackage{bm}

\usetikzlibrary{decorations.markings}
\tikzset{->-/.style={decoration={  markings,  mark=at position #1 with
    {\arrow{>}}},postaction={decorate}}}
\tikzset{-<-/.style={decoration={  markings,  mark=at position #1 with
    {\arrow{<}}},postaction={decorate}}}
\usepackage{extarrows}
\usepackage[all]{xy}
\usepackage{setspace}
\usepackage{thmtools}
\usepackage{thm-restate}
\usepackage{hyperref}
\usepackage{cleveref}
\usepackage{multirow}
\usepackage{ifthen}
\usepackage{tikz-3dplot}
\usepackage{comment}
\usepackage{mathrsfs}
\usepackage{pdfpages}
\usepackage{listings}
\newcommand{\mbR}{\mathbb{R}}

\newcommand{\mbT}{\mathbb{T}}

\newcommand{\mbZ}{\mathbb{Z}}

\theoremstyle{plain}
\newtheorem{theorem}{Theorem}[section]

\newtheorem{lemma}[theorem]{Lemma}
\newtheorem{corollary}[theorem]{Corollary}
\newtheorem{proposition}[theorem]{Proposition}
\newtheorem{conjecture}[theorem]{Conjecture}
\theoremstyle{definition}
\newtheorem{definition}[theorem]{Definition}

\newtheorem{example}[theorem]{Example}

\newtheorem{remark}[theorem]{Remark}

\newtheorem{question}[theorem]{Question}

\numberwithin{equation}{section}
\newtheorem{definition-proposition}[theorem]{Definition-Proposition}

\newtheorem{problem}[theorem]{Problem}

\title{Real $C$-, $G$-structures and sign-coherence of cluster algebras}
\date{\today}
\author{Ryota Akagi}
\address{Graduate School of Mathematics\\ nagoya University\\Chikusa-ku\\ Nagoya\\464-0813\\ Japan.}
\email{ryota.akagi.e6@math.nagoya-u.ac.jp}

\author{Zhichao Chen}
\address{School of Mathematical Sciences\\ University of Science and Technology of China \\ Hefei, Anhui 230026, P. R. China.}
\email{czc98@mail.ustc.edu.cn}

\begin{document}

\maketitle
\begin{abstract}
We generalize the theory of integer $C$-, $G$-matrices in cluster algebras to the real case. By a skew-symmetrizing method, we can reduce the problem of skew-symmetrizable patterns to skew-symmetric patterns. In this sense, the sign-coherence of a more general real class called of quasi-integer type can be inherited directly from that of integer $C$-, $G$-matrices proved by Gross-Hacking-Keel-Kontsevich. However, the sign-coherence of real $C$-, $G$-matrices does not always hold in general. For this purpose, we classify all the rank $2$ case and the finite type case via the Coxeter diagrams. We also give two conjectures about the real exchange matrices and $C$-, $G$-matrices. Under these conjectures, the dual mutation, $G$-fan structure and synchronicity property hold. As an application, the isomorphism of several kinds of exchange graphs is studied.
\end{abstract}
\tableofcontents
\section{Introduction}
\subsection{Background}
Cluster algebras were introduced by \cite{FZ02} in the study of total positivity of Lie groups and canonical bases of quantum groups. The main object is the {\em seed} $\Sigma=({\bf x}=(x_1,x_2,\dots,x_n),B)$, where $x_1,\dots,x_n$ are called {\em cluster variables}, the integer skew-symmetrizable matrix $B$ is called an {\em exchange matrix}, and its transformation is called a {\em mutation}. By applying mutations repeatedly, we can obtain a collection of seeds ${\bf \Sigma}=\{\Sigma_t=({\bf x}_t,B_{t})\}_{t \in \mathbb{T}_{n}}$, which is called a {\em cluster pattern}. (The index set $\mathbb{T}_{n}$ is the $n$-regular tree.) The collection of exchange matrices ${\bf B}=\{B_{t}\}_{t \in \mathbb{T}_{n}}$ is called a {\em $B$-pattern}.
\par
One fundamental result is the {\em Laurent phenomenon} \cite{FZ02}, which states that any cluster variable can always be expressed as a Laurent polynomial in terms of the initial ones, despite being defined through rational mutations. This property ensures that cluster variables remain tractable in principle. However, after repeated mutations, their expressions quickly become complicated. To address this problem, $c$-vectors, $g$-vectors, and $F$-polynomials were introduced in \cite{FZ07}. These objects which are defined from specific features of cluster variables can surprisingly recover them via the {\em separation formula}. Moreover, they have simple and self-contained recursions. Thus, by focusing on these three objects instead of cluster variables, many problems become easier to handle.
\par
In this paper, we focus on the $c$-, $g$-vectors. The matrices whose row vectors are $c$-vectors (resp. $g$-vectors) are called {\em $C$-matrices} (resp. {\em $G$-matrices}). This matrix notation and the recursion (see Definition~\ref{def: C-, G-matrices}) were introduced by \cite{NZ12}. We call their collections ${\bf C}(B)=\{C_{t}\}_{t \in \mathbb{T}_{n}}$ and ${\bf G}(B)=\{G_{t}\}_{t \in \mathbb{T}_{n}}$ a {\em $C$-pattern} and a {\em $G$-pattern}, respectively. Uniformly, we call $B$-, $C$-, $G$-patterns the \emph{matrix pattern}.
\par
{\em Sign-coherence} is one of the most important properties of $C$-, $G$-matrices (See Definition~\ref{def: sign-coherence}.) It was conjectured by \cite{FZ07} and solved by different steps. For the skew-symmetric case, it was solved  by \cite{DWZ10, Pla11, Nag13} with the method of algebraic representation theory, and for the skew-symmetrizable case, this conjecture was completely proved by \cite{GHKK18} with the method of scattering diagrams. Moreover, under this conjecture, some important dualities among $C$-, $G$-matrices were obtained \cite{NZ12}. 
\par
Although $C$-, $G$-matrices are defined by the special information of cluster variables, they are still equipped with the information of periodicity. To state the claim, we define the action via a permutation $\sigma \in \mathfrak{S}_{n}$. For the matrices, we define two actions $\sigma,\tilde{\sigma}$ on $\mathrm{M}_{n}(\mathbb{Z})$ as in (\ref{eq: permutation actions}). For ${\bf x}=(x_1,x_2,\dots,x_n)$, we define $\sigma {\bf x}=(x_{\sigma^{-1}(1)},\dots,x_{\sigma^{-1}(n)})$. Then, as the following theorem indicates, the periodicity for seeds and cluster variables is inherited by $C$-, $G$-matrices.
\begin{theorem}[{\cite[Thm.~5.2]{Nak21}, \cite[Cor.~II.7.10]{Nak23}}, Synchronicity]\label{thm: ordinary synchronicity}
For any $t,t' \in \mathbb{T}_n$ and $\sigma \in \mathfrak{S}_n$, the following conditions are equivalent.
\begin{itemize}
\item The periodicity for seeds $({\bf x}_{t'},B_{t'})=(\sigma{\bf x}_{t},\sigma B_{t})$.
\item The periodicity for clusters ${\bf x}_{t'}=\sigma {\bf x}_{t}$.
\item The periodicity for $C$-matrices $C_{t'}=\tilde{\sigma} C_t$.
\item The periodicity for $G$-matrices $G_{t'}=\tilde{\sigma} G_t$.
\end{itemize}
Moreover, the $G$-fan has the periodicity $\mathcal{C}(G_{t'})=\mathcal{C}(G_{t})$ if and only if the permutation $\sigma \in \mathfrak{S}_{n}$ as above exists.
\end{theorem}
Thus, $C$-, $G$-matrices encode sufficient information to capture the combinatorial structure of cluster variables. One of the most important objects is the {\em $G$-fan} (also called the {\em $g$-vector fan}), see Definition~\ref{def: G-fan}. The fact that $g$-vectors form the fan structure was conjectured in \cite{FZ07}, and proved by \cite{GHKK18}. Thanks to \Cref{thm: ordinary synchronicity}, this can be seen as a certain geometric realization of the cluster complex in \cite{FZ03a}.
\par
When an initial exchange matrix is acyclic and of finite or affine type, the structure of the $G$-fans (and the related structure such as {\em generalized associahedra}) is well-studied by the Coxeter groups \cite{FZ03a, FZ03b, FZ07, Ste13, RS16, RS20}. However, the structure of $G$-fans seems so complicated and varied in general.
\subsection{Purpose and related works}
Originally, $C$-, $G$-matrices are defined based on cluster variables. In this sense, we need to assume that the exchange matrix $B$ has integer components because they appear in exponents of cluster variables. On the other hand, we can give another equivalent definition by the recursion formulas. (See Definition~\ref{def: C-, G-matrices}.) Based on this definition, we can naturally generalize the definition for the {\em real} entries. {\em The purpose of this paper is to generalize and study the structure and sign-coherence of $C$-matrices and $G$-matrices admitting real entries}. To distinguish between this generalized real case and the integer case, we sometimes refer the integer case to the ordinary cluster algebra or the ordinary cluster theory.


\subsubsection*{Related works}
Such generalization was made slightly for $B$-matrices. In \cite{BBH11}, they studied a special type of matrices of rank $3$ called cluster-cyclic. In \cite{FT19}, they constructed a nice geometric realization of rank $3$ mutation equivalence class. In \cite{FT23}, they classified the finite type of $B$-matrices, that is, the number of $B$-matrices obtained by applying mutations is finite. In \cite{DP24, DP25}, some special $C$-, $G$-matrices (related to non-crystallographic root systems) are constructed by using the folding method. In \cite{Lam18}, the mutation of cluster variables was considered in the real setting.
\par
Felikson and Lampe \cite{FL23} introduced a notion of exchange graphs for some $\mathbb{R}$-valued quivers (real skew-symmetric matrices) through a geometric realization of a mutation equivalence class. More precisely, they defined a seed $(\mathbf{v}=(\mathbf{v}_1,\dots,\mathbf{v}_n),B)$ as a pair consisting of a skew-symmetric matrix $B \in \mathrm{M}_n(\mathbb{R})$ and a certain tuple of vectors $\mathbf{v}=(\mathbf{v}_1,\dots,\mathbf{v}_n)$ that realizes $B$. Then, by choosing a specific point referred to as a \emph{reference point}, they constructed the exchange graph as a quotient graph of these seeds. Even for type $A_2$, depending on a reference point, their construction might give another exchange graph. In \cite{FL23}, for each mutation equivalence class of finite type, they proved that there exists a geometric realization with a certain condition called \emph{compatibility}, and then the exchange graphs coincides with the ordinary one. We also introduce the exchange graphs in \Cref{def: exchange graph} in a completely different way. If the sign-coherence of $c$-vectors holds, the corresponding exchange graphs of \cite{FL23} coincide with those introduced in this paper since each vector $\mathbf{v}_i$ corresponds precisely to a $c$-vector. See \cite[Rem.~5.6]{FL23} and \cite{Sev15, Sev19}. On the other hand, if sign-coherence fails, the exchange graphs of \cite{FL23} no longer coincide with ours. For example, consider the quiver in \Cref{cap: non-sign-coherent quiver}. Then, in \cite[Fig.~6]{FT19}, it was shown that the corresponding exchange graph of type $H_3'$ is finite consisting of $48$ vertices. On the other hand, by a direct calculation, we verify that our exchange graph is much larger than that and is expected to be infinite. 
\begin{figure}[htbp]
\begin{tikzpicture}[dot/.style={circle, fill, inner sep=1.5pt}]
\node[dot] (1) at (0,0) {};
\node[dot] (2) at (2,0) {};
\node[dot] (3) at (4,0) {};
\draw[->] (1) node [xshift=-20] {$H_3':$}-> (2) node [pos=0.5, above] {$2\cos{\frac{2}{5}\pi}$};
\draw[->] (2)->(3) node [pos=0.5, above] {$1$};
\end{tikzpicture}
\caption{Example of sign-incoherent case (type $H_3'$ in \cite{FT23}).} \label{cap: non-sign-coherent quiver}
\end{figure}
\subsubsection*{Motivation of the generalization}
Before explaining main theorems, we explain the motivation of this generalization, which would be helpful even if we only focus on ordinary cluster algebras. The basic idea has already appeared in \cite{Rea14}, and the proof can be done by combining two known results in \cite{FZ03a, Nak21}.
\par
Let $B=(b_{ij}) \in \mathrm{M}_n(\mathbb{R})$ be a skew-symmetrizable matrix with a skew-symmetrizer $D$. Define the map $\mathrm{Sk}(B)=(\mathrm{sign}\sqrt{|b_{ij}b_{ji}|}) \in \mathrm{M}_n(\mathbb{R})$. Then, $\mathrm{Sk}(B)$ is a skew-symmetric matrix. Now, we consider the following two matrix patterns; one is ${\bf B}(B)=\{B_{t}\}$, ${\bf C}(B)=\{C_{t}\}$, and ${\bf G}(B)=\{G_{t}\}$, and the other is ${\bf B}(\mathrm{Sk}(B))=\{\hat{B}_{t}\}$, ${\bf C}(\mathrm{Sk}(B))=\{\hat{C}_{t}\}$, and ${\bf G}(\mathrm{Sk}(B))=\{\hat{G}_{t}\}$. Then, by combining the facts in \cite{FZ03a,Nak21}, these two kinds of matrices have the following relationship.
\begin{proposition}[\Cref{prop: skew-symmetrizing method}, Skew-symmetrizing method]\label{prop: intro skew-symmetrizing method}
We have the following correspondence.
\begin{equation}
B_t=D^{-\frac{1}{2}}\hat{B}_{t}D^{\frac{1}{2}},
\ 
C_t=D^{-\frac{1}{2}}\hat{C}_{t}D^{\frac{1}{2}},
\ 
G_t=D^{-\frac{1}{2}}\hat{G}_{t}D^{\frac{1}{2}}.
\end{equation}
\end{proposition}
Thanks to these equalities, we can reduce some problems of the skew-symmetrizable pattern ${\bf B}(B)$, ${\bf C}(B)$, and ${\bf G}(B)$ to the skew-symmetric pattern ${\bf B}(\mathrm{Sk}(B))$, ${\bf C}(\mathrm{Sk}(B))$, and ${\bf G}(\mathrm{Sk}(B))$. However, even if $B$ is an integer matrix, $\mathrm{Sk}(B)$ is not necessarily an integer matrix. This is one important reason why we want to consider the generalization to real entries.
\par
Thanks to \Cref{prop: intro skew-symmetrizing method}, to understand $c$- and $g$-vectors in ordinary cluster algebras, it is enough to consider 
\begin{equation}\label{eq: intro quasi-integer type}
\widehat{\mathcal{S}}_{n}=\{\mathrm{Sk}(B) \mid \textup{$B \in \mathrm{M}_{n}(\mathbb{Z})$ is an {\em integer} skew-symmetrizable matrix}\},
\end{equation}
and we say that each matrix $B \in \widehat{\mathcal{S}}_n$ is of \emph{quasi-integer type}.
By using the known fact from \cite{MS20}, the condition for determining whether $B \in \widehat{\mathcal{S}}_n$ or not is characterized combinatorially (\Cref{prop: classification of quasi-integer type}).

\begin{remark}
In the earlier version of this manuscript, \Cref{prop: classification of quasi-integer type} was stated as one of the main theorems. However, the authors subsequently noticed that this fact has already been shown in \cite{MS20}.
\end{remark}

\subsection{Fundamental properties and questions}
Although we introduce $C$- and $G$-matrices for an arbitrary real skew-symmetrizable matrix $B\in \mathrm{M}_n(\mathbb{R})$, not all such matrices seem to possess the rich structures familiar from ordinary cluster algebra theory. Thus, it is necessary to identify suitable conditions under which these structures can still be recovered in the real setting. In this paper, we work under the assumption of the \emph{sign-coherence} of $C$- and $G$-matrices, see Definition~\ref{def: sign-coherence}. We note that this property does not hold in general, as illustrated in Example~\ref{ex: sign-incoherent}.

Under the sign-coherence assumption for a single $C$-pattern $\mathbf{C}(B)$, some basic properties can be naturally generalized, see \Cref{prop: fundamental properties under sign-coherency}. However, further difficulties arise when deriving the \emph{dual mutation formulas} in \Cref{prop: dual mutation}. This leads us to consider two additional
conjectures (Conjecture~\ref{conj: standard hypothesis} and
Conjecture~\ref{conj: discreteness conjecture}). Under these assumptions, we can
further obtain enriched structures such as dualities and $G$-fan structures,
which are analogous to those in ordinary cluster algebra theory.
\begin{conjecture}[Conjectures~\ref{conj: standard hypothesis},~\ref{conj: discreteness conjecture},~\ref{conj: standard and discreteness conjecture}]
Let $B$ be a skew-symemtrizable matrix with a skew-symmetrizer $D=\mathrm{diag}(d_1,\dots,d_n)$. Suppose that $B$ satisfies the sign-coherent property.
\\
\textup{($a$)} All mutation-equivalent matrices $B' \in {\bf B}(B)$ also satisfy the sign-coherent property.
\\
\textup{($b$)} If a $c$-vector ${\bf c}_{i;t}$ is parallel to ${\bf e}_j$ ($i,j=1,\dots,n$), the length of ${\bf c}_{i;t}$ is $\sqrt{d_{i}d_{j}^{-1}}$.
\end{conjecture}
In particular, the condition ($b$) (Conjecture~\ref{conj: discreteness conjecture}) can be easily shown in ordinary cluster algebras, see Proposition~\ref{prop: discreteness lemma for integer case}.
\par
By assuming these conjectures, we can obtain the same phenomenon in ordinary cluster algebras. In particular, the following proposition holds.
\begin{proposition}[Proposition~\ref{prop: dual mutation}, \Cref{prop: fan}]
By assuming Conjecture~\ref{conj: standard and discreteness conjecture}, we obtain the following:
\begin{itemize}
\item Third duality (\ref{eq: thied duality}) and the dual mutation formula (\ref{eq: dual mutation formula}).
\item The $g$-vectors form a $G$-fan structure (Definition~\ref{def: G-fan}).
\end{itemize}
\end{proposition}
Hence, although two conjectures remain, we can naturally generalize the structures of the $G$-fans under the sign-coherence. However, for a given skew-symmetrizable matrix $B \in \mathrm{M}_n(\mathbb{R})$, it is quite difficult to determine whether $B$ satisfies the sign-coherent property. Thus, one of the most fundamental problems is the following.
\begin{question}\label{prob: sign-coherence}
When does the sign-coherence hold?
\end{question}
In \cite{AC25}, we showed that all the cluster-cyclic exchange matrices of rank $3$ satisfy this property including real entries.
\par
Another problem happens from \Cref{ex: bad phenomenon for periodicity}. The fundamental objects in real $C$- and $G$-matrices are the $C$-pattern, $G$-pattern, and the $G$-fan. In the ordinary cluster algebras, as stated in \Cref{thm: ordinary synchronicity}, all of them share the same periodicity. However, for the generalized setting, this is no longer true as \Cref{ex: bad phenomenon for periodicity} shows. Thus, we need to consider the following problem.
\begin{question}\label{prob: periodicity}
How are the periodicities appearing in these patterns related to each other?
\end{question}

\subsection{Main results}

We begin with considering \Cref{prob: sign-coherence}. The quasi-integer type in \eqref{eq: intro quasi-integer type} naturally generalizes the integer type considered in \cite{GHKK18}, and its combinatorial characterization implies sign-coherence in this setting. Beyond quasi-integer type, even the existence of sign-coherent matrices is non-trivial. Although a complete classification remains open, we identify several additional classes satisfying sign-coherence.
\par
Firstly, we can classify the case of rank $2$, which is  simplest but the most essential.
\begin{theorem}[Theorem~\ref{thm: rank 2 classification}]
Let the exchange matrix be $B=\left(\begin{smallmatrix}
0 & -a\\
b & 0
\end{smallmatrix}\right)$ with $a,b \in \mathbb{R}_{\geq  0}$. Then, all $C$-matrices are sign-coherent if and only if either of the following holds.
\begin{itemize}
\item $\sqrt{ab}=2\cos{\frac{\pi}{m}}$ holds for some $m \in \mathbb{Z}_{\geq 2}$.
\item $\sqrt{ab} \geq 2$.
\end{itemize}
\end{theorem}
Another classification is for the finite type via Coxeter diagrams. Note that, by the skew-symmetrizing method, it suffices to consider the skew-symmetric case. Since each skew-symmetric matrix corresponds to an $\mathbb{R}$-valued quiver, we use the quiver notation defined in \Cref{def: quiver}.
\begin{theorem}[Theorem~\ref{thm: finite type classifcation}]\label{thm: intro finite type classification}
Let $B \in \mathrm{M}_{n}(\mathbb{R})$ be skew-symmetric. Suppose that the corresponding quiver is connected. Then, $B$ satisfies both of
\begin{itemize}
\item for any $B' \in {\bf B}(B)$, $B'$ satisfies the sign-coherent property.
\item for any $B' \in {\bf B}(B)$, the number of $C$-matrices is finite.
\end{itemize}
if and only if the corresponding quiver is mutation-equivalent to any of the Coxeter quivers in Figure~\ref{fig: Coxeter diagrams}.
\end{theorem}
In \cite{FZ03a}, it was shown that the cluster algebras of finite type can be classified by Dynkin diagrams, which correspond to crystallographic root systems. On the other hand, by generalizing real entries, this classification can be done by Coxeter diagrams, which correspond to arbitrary root systems (including non-crystallographic ones).  It seems that these types are also related to the Coxeter groups. In Appendix~\eqref{eq: Number finite type}, we obtain the number of $g$-vectors and the $G$-matrices including the non-crystallographic type. The number of $g$-vectors coincides with the number of almost positive roots in the corresponding Coxeter group, and the number of $G$-matrices coincides with the number of chambers induced by the almost positive roots in the Coxeter arrangements, which has already been explored by \cite{FZ03b, FR07}.
\par
Next, we consider \Cref{prob: periodicity}.
When we discuss the periodicity, the following matrices are technical and useful.
\begin{equation}
\tilde{C}_{t}=C_{t}D^{-\frac{1}{2}},\quad \tilde{G}_{t}=G_{t}D^{-\frac{1}{2}},
\end{equation}
where $D$ is a fixed skew-symmetrizer of the initial exchange matrix $B$. We call $\tilde{C}_{t}$ and $\tilde{G}_{t}$ a {\em modified $C$-matrix} and a {\em modified $G$-matrix}, respectively. The most important motivation to introduce these matrices is that if $B$ satisfies the sign-coherent property, we obtain the following equivalence without any conjectural assumptions, see also \Cref{lem: equivalency of the periodicity among modified c- g-vectors}:
\begin{equation}
\tilde{C}_{t'}=\tilde{\sigma}\tilde{C}_{t} \Longleftrightarrow \tilde{G}_{t'}=\tilde{\sigma}\tilde{G}_{t'}.
\end{equation}
Moreover, by using these modified matrices, we can understand the relationship among the matrix patterns and the $G$-fan.
\begin{theorem}[{\Cref{thm: relationship among exchange graphs}}]
Let $B \in \mathbf{SC}$, and suppose that \Cref{conj: standard and discreteness conjecture} holds for this $B$. Then, the following canonical graph isomorphisms hold.
\begin{gather}
{\bf EG}(\tilde{\bf C}(B)) \cong {\bf EG}(\tilde{\bf G}(B)) \cong {\bf EG}(\Delta_{\bf G}(B)),
\\
{\bf EG}({\bf C}(B)) \cong {\bf EG}({\bf G}(B)).
\end{gather}
\end{theorem}
In particular, when $B$ is skew-symmetric, we can take $D$ as the identity matrix. Thus, each modified matrix is identical to the corresponding matrix. Thus, we conclude that the same statement in \Cref{thm: ordinary synchronicity} holds when $B$ is skew-symmetric.

\subsection{Further problems}
We provide a definition based on the recursion (Definition~\ref{def: C-, G-matrices}). While an explicit computation can be done, this definition makes it difficult to grasp the underlying structure.
On the other hand, based on the results in this paper, such as \Cref{thm: intro finite type classification}, the real $C$-, $G$-matrices with sign-coherence are related to topics in geometric combinatorics. From this perspective, it is worthwhile to consider the following open problem:
\begin{problem}
Provide an alternative and structural definition of real $C$-, $G$-matrices with sign-coherence.
\end{problem}
More specifically, the following problem remains to be addressed.
\begin{problem}
Give an alternative proof of \Cref{thm: intro finite type classification} by establishing a one-to-one correspondence between the mutation-equivalence classes of Coxeter quivers in Figure~\ref{fig: Coxeter diagrams} and the Coxeter groups of finite type.
\end{problem}
In ordinary cluster algebras, one good correspondence was constructed by \cite{FZ03a} for cases where an initial exchange matrix is bipartite. This correspondence was generalized to acyclic initial exchange matrices in \cite{RS16} via the {\em Cambrian fan}, which is a quotient of the weak order of the Coxeter group. Since the number of cluster variables is independent of the choice of initial exchange matrices, it suffices to consider such special cases. On the other hand, the counterpart in our setting is the dual mutation formula in \Cref{prop: dual mutation}. This formula ensures that the number of $C$-, $G$-matrices is independent of the initial exchange matrix. However, the current formula relies on sign-coherence, which significantly increases the difficulty of the problem.

Another natural open problem that we consider is how to properly define real cluster algebras based on the real patterns we studied. However, it seems quite difficult to give a definition compatible with the Laurent phenomenon, positivity and so on. 
\subsection{Structure of the paper}
Most of the notations in this paper follow from those of \cite{FZ07, NZ12, Nak23}. Additionally, some claims in Sections~\ref{sec: pre}, \ref{sec: sign-coherence}, \ref{sec: dual and fan} can be shown by doing the same arguments as in \cite[\S.~II.1, II.2]{Nak23}, so we omit their details and refer the proofs to them. This paper is organized as follows.
\par
In Section~\ref{sec: pre}, we define real $B$-, $C$-, $G$-matrices and introduce some basic facts and properties.
\par
In Section~\ref{sec: Skew-symmetrizing method}, we introduce the skew-symmetrizing method (\Cref{prop: skew-symmetrizing method}), which gives the motivation to generalize integer $C$-, $G$-matrices to the real entries.
\par
In Section~\ref{sec: sign-coherence}, we introduce the sign-coherence of real $C$- and $G$-matrices and show that the basic properties of the integer case still hold under the assumption of sign-coherence.
\par
In Section~\ref{sec: conjectures}, we introduce two conjectures (Conjecture~\ref{conj: standard hypothesis} and Conjecture~\ref{conj: discreteness conjecture}), which are needed to obtain the fan structure related to $G$-matrices.
\par
In Section~\ref{sec: dual and fan}, we prove the dual mutation and third duality (Proposition~\ref{prop: dual mutation}) and the $G$-fan structure (\Cref{prop: fan}) under certain conditions (\Cref{conj: standard and discreteness conjecture}).
\par
In Section~\ref{sec: classification of rank 2}, we classify the rank 2 sign-coherent class and give some examples of their $G$-fans (\Cref{thm: rank 2 classification}).
\par
In Section~\ref{sec: finite type classification}, we give a classification of sign-coherent finite type via Coxeter diagrams (Theorem~\ref{thm: finite type classifcation}).
\par
In Section~\ref{sec: modified and synchro}, we introduce modified $C$-, $G$-patterns, and we show some properties similar to Theorem~\ref{thm: ordinary synchronicity}, such as synchronicity among the matrix patterns and a $G$-fan (\Cref{prop: matrix-cone synchronicity} and \Cref{prop: CG synchronicity}).
\par
In Section~\ref{sec: exchange graphs}, we study the isomorphism among different exchange graphs (\Cref{thm: relationship among exchange graphs}).
\section{Preliminaries}\label{sec: pre}
\subsection{Basic notations}\label{sec: basic notations}
We fix a positive integer $n \in \mathbb{Z}_{\geq 2}$, and we refer it as a {\em rank}. We fix the notations for the following special matrices, sets, and operations.
\begin{itemize}
\item Let $E_{ij} \in \mathrm{M}_{n}(\mathbb{R})$ be a matrix obtained from the zero matrix by replacing $(i,j)$th entry with $1$.
\item Let $\mathrm{diag}(d_1,d_2,\dots,d_n)=d_1E_{11}+\dots+d_nE_{nn}$ be the diagonal matrix. We say that a diagonal matrix is {\em positive} if all diagonal entries are strictly positive.
\item Let $I_n=\mathrm{diag}(1,1,\dots,1)$ be the identity matrix of order $n$.
\item For each $k = 1,2,\dots,n$, let $J_k$ be the matrix obtained by replacing the $(k,k)$th entry of $I_n$ with $-1$.
\item For each $k =1,2,\dots,n$ and $A=(a_{ij}) \in \mathrm{M}_n(\mathbb{R})$, let $A^{k \bullet}=E_{kk}A \in \mathrm{M}_{n}(\mathbb{R})$ (resp. $A^{\bullet k}=AE_{kk} \in \mathrm{M}_{n}(\mathbb{R})$) be the matrix obtained by replacing all entries with $0$ except for the $k$th row (resp. the $k$th column).
\item For each $x \in \mathbb{R}$, let $[x]_{+}=\max(x,0)$.
\item For each $A=(a_{ij}) \in \mathrm{M}_{n}(\mathbb{R})$, let $[A]_{+}=([a_{ij}]_{+}) \in \mathrm{M}_{n}(\mathbb{R})$.
\item  Let $\mathbb{R}_{+}=\{x \geq 0\}$ and $\mathbb{R}_{-}=\{x \leq 0\}$. For any $\epsilon_1,\dots,\epsilon_{n} \in \{\pm 1\}$, we indicate the closed orthant $\mathfrak{O}_{\epsilon_1,\dots,\epsilon_{n}}=\mathbb{R}_{\epsilon_1}\times\cdots\times\mathbb{R}_{\epsilon_{n}} \subset \mathbb{R}^n$. In particular, we denote by $\mathfrak{O}_{+}^n=\mathfrak{O}_{+,+,\dots,+}$ and $\mathfrak{O}_{-}^{n}=\mathfrak{O}_{-,-,\dots,-}$.
\item For any $x\in \mbR$, $\mathrm{sign}(x)$ is defined by $1$, $0$, and $-1$ if $x>0$, $x=0$, and $x<0$, respectively. 
\end{itemize}
\subsection{$C$-, $G$-matrices and first duality}
A real matrix $B \in \mathrm{M}_{n}(\mathbb{R})$ is said to be {\em skew-symmetrizable} if there exists a positive diagonal matrix $D=\mathrm{diag}(d_1,d_2,\dots,d_n)$, where $d_1,d_2,\dots,d_n \in \mathbb{R}_{>0}$, such that $DB$ is skew-symmetric. This $D$ is called a {\em skew-symmetrizer} of $B$. We verify that every skew-symmetrizable matrix $B=(b_{ij}) \in \mathrm{M}_n(\mathbb{R})$ is {\em sign skew-symmetric}, that is, $\mathrm{sign}(b_{ij})=-\mathrm{sign}(b_{ji})$. 
\begin{definition}
For a skew-symmetrizable matrix $B \in \mathrm{M}_{n}(\mathbb{R})$ and an index $k=1,2,\dots,n$, we define the {\em mutation $\mu_{k}(B)$ in direction $k$} as
\begin{equation}\label{eq: mutation of B matrix}
\mu_{k}(B)=(J_k+[-B]_{+}^{\bullet k})B(J_{k}+[B]_{+}^{k \bullet}).
\end{equation}
This mutation is called a {\em mutation of $B$-matrix}.
\end{definition}
Let $B=(b_{ij}) \in \mathrm{M}_{n}(\mathbb{R})$. By a direct calculation, we verify that the $(i,j)$th entry $b'_{ij}$ of $\mu_{k}(B)$ is given by
\begin{equation}
b'_{ij}=\begin{cases}
-b_{ij}, & \textup{if $i=k$ or $j=k$},\\
b_{ij}+b_{ik}[b_{kj}]_{+}+[-b_{ik}]_{+}b_{kj}, & \textup{if $i,j \neq k$}.
\end{cases}
\end{equation}
Since $x=[x]_{+}-[-x]_{+}$ for any $x \in \mathbb{R}$, the following expression is independent of the choice of $\varepsilon= \pm 1$. That is to say, 
\begin{equation}
\mu_{k}(B)=(J_{k}+[-\varepsilon B]_{+}^{\bullet k})B(J_{k}+[\varepsilon B]_{+}^{k \bullet})
\end{equation}
and, equivalently,
\begin{equation}
b'_{ij}=b_{ij}+b_{ik}[\varepsilon b_{kj}]_{+}+[-\varepsilon b_{ik}]_{+}b_{kj}=b_{ij}+\mathrm{sign}(b_{ik})[b_{ik}b_{kj}]_{+}. 
\end{equation}
The following fundamental properties are satisfied even if we generalize to the real entries.
\begin{lemma}[cf. \cite{FZ02}]
Let $B \in \mathrm{M}_{n}(\mathbb{R})$ be a skew-symmetrizable matrix with a skew-symmetrizer $D$. For any $k=1,2,\dots,n$, we have  
\\
\textup{($a$)} $\mu_{k}(B)$ is also skew-symmetrizable with the same skew-symmetrizer $D$.
\\
\textup{($b$)} Let $B'=\mu_k(B)$. Then, we have $\mu_{k}(B')=B$. Namely, the mutation of $B$-matrix is an involution.
\end{lemma}
Let $\mathbb{T}_{n}$ be the (labeled) $n$-regular tree, that is, a simple graph where every vertex has degree $n$ and these edges are labeled by $1,2,\dots,n$ distinctly. If two vertices are connected by an edge labeled by $k$, we say that these two vertices are {\em k-adjacent}. We define the {\em distance} $d(t,t')$ between the two vertices $t$ and $t'$ by the number of edges in the shortest path from $t$ to $t'$. 
\par
As in the ordinary cluster theory, we define some collections indexed by $t \in \mathbb{T}_{n}$.
\begin{definition}
A collection of skew-symmetrizable matrices ${\bf B}=\{B_{t}\}_{t \in \mathbb{T}_{n}}$ is called a {\em $B$-pattern} if it satisfies the following condition:
\begin{quote}
For any $k$-adjacent vertices $t,t' \in \mathbb{T}_{n}$, it holds that $B_{t'}=\mu_{k}(B_{t})$.
\end{quote}
We call an element of $B$-pattern a {\em $B$-matrix}.
In the ordinary cluster theory, we also call them {\em exchange matrices}.
If $B$ and $B'$ are in the same $B$-pattern, then $B$ and $B'$ are said to be {\em mutation-equivalent}.
\par
For any skew-symmetrizable matrix $B$, if we set the initial condition $B=B_{t_0}$, then other $B_{t}$ are determined recursively. In this sense, we sometimes write ${\bf B}={\bf B}^{t_0}(B)$, and we refer to $B_{t_0}=B$ as an {\em initial exchange matrix}.
\end{definition}
The main object in this paper is the following patterns.
\begin{definition}\label{def: C-, G-matrices}
Let ${\bf B}=\{B_t\}$ be a $B$-pattern. Then, we define the {\em $C$-pattern} ${\bf C}^{t_0}({\bf B})=\{C^{t_0}_{t}\}_{t \in \mathbb{T}_{n}}$ and the {\em $G$-pattern} ${\bf G}^{t_0}({\bf B})=\{G^{t_0}_{t}\}_{t \in \mathbb{T}_{n}}$ with an initial vertex $t_0 \in \mathbb{T}_{n}$ as follows:
\begin{itemize}
\item $C^{t_0}_{t_0}=G^{t_{0}}_{t_0}=I_{n}$.
\item If $t$ and $t'$ are $k$-adjacent, it holds that
\begin{equation}\label{eq: mutation of C, G matrices}
\begin{aligned}
C^{t_0}_{t'}&=C^{t_0}_{t}J_k+C_{t}^{t_0}[B_t]^{k\bullet}_{+}+[-C_{t}^{t_0}]_{+}^{\bullet k}B_{t},\\
G^{t_{0}}_{t'}&=G^{t_0}_{t}J_{k}+G^{t_0}_{t}[-B_{t}]^{\bullet k}_{+}-B_{t_0}[-C^{t_0}_{t}]^{\bullet k}_{+}.
\end{aligned}
\end{equation}
\end{itemize}
For a given skew-symmetrizable matrix $B \in \mathrm{M}_n(\mathbb{R})$, we also write ${\bf C}^{t_0}(B)={\bf C}^{t_0}({\bf B}^{t_0}(B))$ and ${\bf G}^{t_0}(B)={\bf G}^{t_0}({\bf B}^{t_0}(B))$.
These matrices $C_t$ and $G_t$ are called {\em $C$-matrices} and {\em $G$-matrices}, respectively.
If we fix an initial vertex $t_0$, we omit $t_0$ and simply write $C_t=C^{t_0}_{t}$ and $G_t=G^{t_0}_t$. In this case, we sometimes write $C_{t'}=\mu_{k}(C_{t})$ and $G_{t'}=\mu_{k}(G_{t})$ for $k$-adjacent vertices $t,t' \in \mathbb{T}_{n}$, and call them {\em mutations of a $C$-matrix and a $G$-matrix}, respectively.
\end{definition}
\begin{definition}
For each $C$-matrix $C^{t_0}_{t}$ and $G$-matrix $G^{t_0}_t$, their column vectors are called {\em $c$-vectors} and {\em $g$-vectors}, respectively. We write the $i$th column vector  of $C^{t_0}_{t}$ and $G^{t_0}_{t}$ by ${\bf c}^{t_0}_{i;t}$ and ${\bf g}^{t_0}_{i;t}$.
\end{definition}
For short, the $B$-, $C$-, $G$-patterns  are collectively called {\em matrix patterns} and they have some significant dualities.

The following duality can be shown without any assumption. 
\begin{lemma}[cf. {\cite[(6.14)]{FZ07}}, First duality]
For any $B$-pattern ${\bf B}$ and $t_0,t \in \mathbb{T}_{n}$, we have
\begin{equation}\label{eq: first duality}
G^{t_0}_{t}B_{t}=B_{t_0}C^{t_0}_{t}.
\end{equation}
\end{lemma}
By using this equality, for each $\varepsilon = \pm 1$, the mutation of $C$- and $G$-matrices are also expressed as follows (cf. \cite[(6.12), (6.13)]{FZ07}, \cite[(2.4)]{NZ12}):
\begin{equation}\label{eq: epsilon expression for C-, G-mutations}
\begin{aligned}
C^{t_0}_{t'}&=C^{t_0}_{t}J_k+C_{t}^{t_0}[\varepsilon B_t]^{k\bullet}_{+}+[-\varepsilon C_{t}^{t_0}]_{+}^{\bullet k}B_{t},\\
G^{t_{0}}_{t'}&=G^{t_0}_{t}J_{k}+G^{t_0}_{t}[- \varepsilon B_{t}]^{\bullet k}_{+}-B_{t_0}[-\varepsilon C^{t_0}_{t}]^{\bullet k}_{+}.
\end{aligned}
\end{equation}
\begin{lemma}[cf. \cite{FZ07}]
The mutations of $C$-, $G$-patterns are involutions. Namely, for any $t \in \mathbb{T}_n$ and $k=1,\dots,n$, we have $\mu_k(\mu_k(C_t))=C_t$ and $\mu_k(\mu_k(G_t))=G_t$.
\end{lemma}
Last, we focus on the entries of these real matrices. If we focus on the integer skew-symmetrizable matrix, only integer entries appear in the mutated matrices. However, we cannot expect this property now. Since $\mathbb{R}$ is rather bigger than the ring that we need to consider, we introduce the following subring of $\mathbb{R}$.
\begin{definition}
For each skew-symmetrizable matrix $B = (b_{ij}) \in \mathrm{M}_{n}(\mathbb{R})$, let $\mathbb{Z}_{B}=\mathbb{Z}[\{b_{ij} \mid i,j=1,\dots,n\}]$ be the subring of $\mathbb{R}$ generated by $\{b_{ij} \mid i,j=1,\dots,n\}$. Note that we have $\mbZ \subset \mathbb{Z}_{B} \subset \mbR$.
\end{definition}
As the following proposition shows, this is a natural subring to consider real $C$-, $G$-matrices.
\begin{proposition}\label{prop: ring ZB}
Let $B \in \mathrm{M}_{n}(\mathbb{R})$ be a skew-symmetrizable matrix. Consider its $B$-pattern ${\bf B}^{t_0}(B)=\{B_{t}\}_{t \in \mathbb{T}_{n}}$.
\\
\textup{($a$)} For any $t \in \mathbb{T}_{n}$, we have $\mathbb{Z}_{B_{t}}=\mathbb{Z}_{B}$.
\\
\textup{($b$)} For any $t \in \mathbb{T}_{n}$, we have
\begin{equation}
B_t,C_t,G_t \in \mathrm{M}_{n}(\mathbb{Z}_{B}).
\end{equation}
\end{proposition}
\begin{proof}
($a$) It suffices to show $\mathbb{Z}_{B_{t}}=\mathbb{Z}_{B_{t'}}$ for any adjacent vertices $t,t' \in \mathbb{T}_{n}$. Suppose that $t$ and $t'$ are $k$-adjacent. Then, by the definition of mutation $B_{t'}=\mu_{k}(B)$, each entry of $B_{t'}$ belongs to $\mathbb{Z}_{B_t}$. Thus, $\mathbb{Z}_{B_{t'}} \subset \mathbb{Z}_{B_{t}}$ holds. Since $\mu_{k}$ is an involution, we can do the same argument by considering $B_t=\mu_k(B_{t'})$. Thus, $\mathbb{Z}_{B_{t}} \subset \mathbb{Z}_{B_{t'}}$ also holds. These two inclusions imply $\mathbb{Z}_{B_t}=\mathbb{Z}_{B_{t'}}$.
\\
($b$) We can show the claim by induction because the mutation formulas \eqref{eq: mutation of B matrix} and \eqref{eq: mutation of C, G matrices} define the closed operation within $\mathrm{M}_{n}(\mathbb{Z}_{B})$.
\end{proof}

\subsection{Periodicity}
In Section~\ref{sec: modified and synchro} and Section~\ref{sec: exchange graphs}, we focus on the periodicity of $C$-, $G$-patterns. For this purpose, we introduce some notations and recall the basic properties.
\begin{definition}
Let $\mathfrak{S}_{n}$ be the symmetric group of degree $n$. Then, we introduce the following two kinds of left group action of $\mathfrak{S}_{n}$ on $\mathrm{M}_{n}(\mathbb{R})$ by
\begin{equation}\label{eq: permutation actions}
\sigma A = (a_{\sigma^{-1}(i)\sigma^{-1}(j)}),
\quad
\tilde{\sigma} A = (a_{i\sigma^{-1}(j)}),
\end{equation}
where $A=(a_{ij}) \in \mathrm{M}_{n}(\mathbb{R})$ and $\sigma \in \mathfrak{S}_{n}$.
\end{definition}
Let $P_{\sigma}=(\delta_{i,\sigma^{-1}(j)}) \in \mathrm{M}_{n}(\mathbb{R})$. Then, these operations can also be expressed as
\begin{equation}\label{eq: periodicity and matrix product}
\sigma A = P_{\sigma}^{\top}AP_{\sigma},
\quad
\tilde{\sigma} A = AP_{\sigma}.
\end{equation}

\begin{proposition}[cf. \cite{FZ07}]\label{prop: periodicity}
For any $B$-matrices $B_{t},B_{t'}$, and $C$-matrices $C_{t},C_{t'}$, suppose that there exists $\sigma \in \mathfrak{S}_{n}$ such that $B_{t'}=\sigma B_{t}$ and $C_{t'}=\tilde{\sigma} C_{t}$. Then, for any $k=1,2,\dots,n$, we have
\begin{equation}
\mu_{\sigma(k)}(B_{t'})=\sigma(\mu_{k}(B_{t})),
\ 
\mu_{\sigma(k)}(C_{t'})=\tilde{\sigma}(\mu_{k}(C_{t})).
\end{equation}
Additionally, we assume $G_{t'}=\tilde{\sigma} G_{t}$. Then, we have $\mu_{\sigma(k)}(G_{t'})=\tilde{\sigma}(\mu_{k}(G_{t}))$.
\end{proposition}

\subsection{Projection of matrix patterns}
In this section, fix one initial exchange matrix $B \in \mathrm{M}_{n}(\mathbb{R})$ and an initial vertex $t_0 \in \mathbb{T}_{n}$. Here, we consider a pattern by restricting mutation direction to $J \subset \{1,2,\dots,n\}$.
\par
To state the claim, we introduce the following subtree $\mathbb{T}_{J} \subset \mathbb{T}_{n}$:
\begin{itemize}
\item $t_0 \in \mathbb{T}_{J}$ and each vertex of $\mathbb{T}_{J}$ has the degree $|J|$ as the subgraph $\mathbb{T}_{J}$.
\item For each vertex, the edges whose one endpoint is this vertex are labeled by the elements of $J$. 
\end{itemize}
\begin{definition}
Let $J \subset \{1,2,\dots,n\}$ and $A=(a_{ij}) \in \mathrm{M}_{n}(\mathbb{R})$. Then, we define the {\em submatrix of $A$ restricted to $J$} by the $|J|\times|J|$ square matrix $A' \in \mathrm{M}_{|J|}(\mathbb{R})$ indexed by $J$ whose entries are the same in $A$. We write it by $A|_{J}$.
\end{definition}
By using these notations, we can consider the $B$-pattern ${\bf B}(B|_{J})=\{(B|_{J})_{t}\}_{t \in \mathbb{T}_J}$, ${\bf C}(B|_{J})=\{(C|_{J})_{t}\}_{t \in \mathbb{T}_{n}}$, and ${\bf G}(B|_{J})=\{(G|_{J})_{t}\}_{t \in \mathbb{T}_{n}}$. This pattern corresponds to the original pattern as follows. This idea has appeared in the various papers, such as \cite{FZ03a}.
\begin{lemma}\label{lem: submatrix}
Let $B \in \mathrm{M}_{n}(\mathbb{R})$ be a skew-symmetrizable matrix and $J \subset \{1,2,\dots,n\}$. Then, for any $t \in \mathbb{T}_{J}$, we have
\begin{equation}
B_t|_{J}=(B|_{J})_{t},\ C_{t}|_{J}=(C|_{J})_{t},\ G_{t}|_{J}=(G|_{J})_{t}.
\end{equation}
\end{lemma}
For example, if $J=\{1,2,\dots,s\}$, we show that for any $t \in \mathbb{T}_{n}$, 
\begin{equation}
B_{t}=\left(\begin{matrix}
(B|_{J})_{t} & X_{t}\\
Y_{t} & Z_{t}
\end{matrix}\right),
\quad
C_{t}=\left(\begin{matrix}
(C|_{J})_{t} & U_{t}\\
O & I_{n-|J|}
\end{matrix}\right),
\quad
G_{t}=\left(\begin{matrix}
(G|_{J})_{t} & O\\
V_{t} & I_{n-|J|}
\end{matrix}\right),
\end{equation}
where $X_{t},Y_{t},Z_{t},U_{t},V_{t}$ are some matrices. The same argument was done in \cite[(4.8)]{FG19} for $C$-matrices.
\subsection{Quiver setting}
We identify a skew-symmetric matrix with an {\em $\mathbb{R}$-valued quiver}. By using this identification, the statements sometimes become more simpler. So, we introduce a quiver notation corresponding to a skew-symmetric matrix.
\begin{definition}\label{def: quiver}
For any skew-symmetric matrix $B \in \mathrm{M}_{n}(\mathbb{R})$, we define the corresponding $\mathbb{R}$-valued quiver $Q(B)$.
\begin{itemize}
\item The vertices are labeled by $1,2,\dots,n$.
\item If $b_{ij}>0$, then there is an arrow from $i$ to $j$, denoted by $i\overset{b_{ij}}{\longrightarrow} j$.
\end{itemize}
We refer the number of vertices as the {\em rank}.
Conversely, for a given quiver $Q$ of rank $n$ (without loops and $2$-cycles), we define the skew-symmetric matrix $B(Q) \in \mathrm{M}_n(\mathbb{R})$ by the above correspondence. So, we often identify an $\mathbb{R}$-valued quiver $Q$ as a skew-symmetric matrix, and we write $Q=(q_{ij}) \in \mathrm{M}_{n}(\mathbb{R})$. Each real number $q_{ij}$ with $i\neq j$ is called a \emph{weight} of $Q$.
\end{definition}
\begin{definition}
For an $\mathbb{R}$-valued quiver $Q$, we define the {\em underlying graph} $\Gamma(Q)$ as the graph obtained by ignoring the direction of the quiver $Q$. (We keep the information of indices of vertices and weight of edges.) A path of $\Gamma(Q)$ is called an {\em undirected path} of $Q$. We write an undirected path consisting of edges $k_0-k_1$, $k_1-k_2$,\dots, $k_{r-1}-k_{r}$ by $(k_0,k_1,\dots,k_r)$. In particular, if $k_0=k_r$ and $k_i \neq k_j$ for any $i,j =1,\dots,n$ with $i\neq j$, we say that this undirected path $(k_0,k_1,\dots,k_r)$ is an {\em undirected cycle} of $Q$.
\end{definition}
If $\Gamma(Q)$ is connected, such quiver $Q$ is said to be {\em connected}.

\section{Skew-symmetrizing method}\label{sec: Skew-symmetrizing method}
In the previous section, we introduce $B$-, $C$-, $G$-patterns including real entries. The reason we want to introduce them is that we can reduce some problems into the skew-symmetric case. A similar idea has already appeared in \cite{FZ03a} for $B$-matrices and in \cite{Rea14} for $G$-fans, which is called {\em rescaling}. In this section, we explain this method.
\par
We fix one skew-symmetrizable matrix $B \in \mathrm{M}_{n}(\mathbb{R})$ with a skew-symmetrizer $D$. We take one positive diagonal matrix $H=\mathrm{diag}(h_1,h_2,\dots,h_n)$, and consider the following transformation.
\begin{definition}[Positive conjugation]
Let $H$ be a positive diagonal matrix. Set
\begin{equation}
\hat{B}=HBH^{-1}.
\end{equation}
This is also a skew-symmetrizable matrix with a skew-symmetrizer $H^{-1}DH^{-1}$. We call such transformation a {\em positive conjugation}.
\end{definition}
In this section, denote by ${\bf B}^{t_0}(\hat{B})=\{\hat{B}_{t}\}_{t \in \mathbb{T}_{n}}$, ${\bf C}^{t_0}(\hat{B})=\{\hat{C}_{t}\}_{t \in \mathbb{T}_{n}}$, and ${\bf G}^{t_0}(\hat{B})=\{\hat{G}_{t}\}_{t \in \mathbb{T}_{n}}$ with the initial exchange matrix $\hat{B}_{t_0}=\hat{B}$.
Then, the real positive conjugation directly induces the following three equalities, see also \cite[Lem.~5.24]{Nak21}. 
\begin{lemma}\label{lem: positive conjugation}
For any $t \in \mathbb{T}_{n}$, we have
\begin{equation}
\hat{B}_{t}=HB_{t}H^{-1},
\ 
\hat{C}_{t}=HC_{t}H^{-1},
\ 
\hat{G}_{t}=HG_{t}H^{-1}.
\end{equation}
\end{lemma}
Roughly speaking, the properties of matrix patterns are the same under the positive conjugations.
\par
Now, we set $H=D^{\frac{1}{2}}=\mathrm{diag}(\sqrt{d_1},\sqrt{d_2},\dots,\sqrt{d_n})$. (Algebraically speaking, there are $2^{n}$ choices for $D^{\frac{1}{2}}$ due to the signs of diagonal entries, but we fix $D^{\frac{1}{2}}$ such that all diagonal entries are positive.) We write $D^{-\frac{1}{2}}=(D^{\frac{1}{2}})^{-1}$. By this setting, we can obtain one similar simple representative under the real positive conjugations, which has appeared in \cite[Lem. 8.3]{FZ03a}. 
\begin{lemma}\label{lem: skew-symmetrizing for B}
For any skew-symmetrizable matrix $B=(b_{ij}) \in \mathrm{M}_{n}(\mathbb{R})$ with a skew-symmetrizer $D$, the $(i,j)$th entry of $D^{\frac{1}{2}}BD^{-\frac{1}{2}}$ is $\mathrm{sign}(b_{ij})\sqrt{|b_{ij}b_{ji}|}$. In particular, the matrix $D^{\frac{1}{2}}BD^{-\frac{1}{2}}$ is independent of the choice of a skew-symmetrizer $D$, and it is skew-symmetric.
\end{lemma}
\begin{definition}\label{def: Sk map}
For each skew-symmetrizable matrix $B \in \mathrm{M}_{n}(\mathbb{R})$, let
\begin{equation}
\mathrm{Sk}(B)=D^{\frac{1}{2}}BD^{-\frac{1}{2}}=\left(\mathrm{sign}(b_{ij})\sqrt{|b_{ij}b_{ji}|}\right) \in \mathrm{M}_n(\mathbb{R}).
\end{equation}
Note that by Lemma~\ref{lem: skew-symmetrizing for B}, this matrix is skew-symmetric. 
\end{definition}
By \cite[Prop.~8.20]{Rea14} and \Cref{lem: positive conjugation}, we obtain the following corollary immediately.
\begin{proposition}[Skew-symmetrizing method]\label{prop: skew-symmetrizing method}
Let $B \in \mathrm{M}_{n}(\mathbb{R})$ be a skew-symmetrizable matrix with a skew-symmetrizer $D$. Fix an initial vertex $t_0 \in \mathbb{T}_{n}$, and set ${\bf B}^{t_0}(B)=\{B_t\}$, ${\bf C}^{t_0}(B)=\{C_t\}$, ${\bf G}^{t_0}(B)=\{G_t\}$ and ${\bf B}^{t_0}(\mathrm{Sk}(B))=\{\hat{B}_t\}$, ${\bf C}^{t_0}(\mathrm{Sk}(B))=\{\hat{C}_t\}$, ${\bf G}^{t_0}(\mathrm{Sk}(B))=\{\hat{G}_t\}$. Then, we can recover $B_t$, $C_t$, and $G_t$ from the skew-symmetric patterns $\hat{B}_{t}$, $\hat{C}_{t}$, and $\hat{G}_{t}$ by the following correspondence.
\begin{equation}
B_t=D^{-\frac{1}{2}}\hat{B}_{t}D^{\frac{1}{2}},
\ 
C_t=D^{-\frac{1}{2}}\hat{C}_{t}D^{\frac{1}{2}},
\ 
G_t=D^{-\frac{1}{2}}\hat{G}_{t}D^{\frac{1}{2}}.
\end{equation}
Moreover, by setting $C_t=({\bf c}_{1;t},\dots,{\bf c}_{n;t})$, $G_t=({\bf g}_{1;t},\dots,{\bf g}_{n;t})$, $\hat{C}_t=(\hat{\bf c}_{1;t},\dots,\hat{\bf c}_{n;t})$, $\hat{G}_t=(\hat{\bf g}_{1;t},\dots,\hat{\bf g}_{n;t})$, we obtain the following correspondence for any $i=1,2,\dots,n$.
\begin{equation}
{\bf c}_{i;t}=\sqrt{d_i}D^{-\frac{1}{2}}\hat{\bf c}_{i;t},
\quad
{\bf g}_{i;t}=\sqrt{d_i}D^{-\frac{1}{2}}\hat{\bf g}_{i;t}.
\end{equation}
\end{proposition}
For $G$-fan, this idea has appeared in \cite{Rea14}. Here, we can establish this relationship between $C$-, $G$-matrices therein.
\begin{remark}\label{rmk: rk2}
Even if $B \in \mathrm{M}_{n}(\mathbb{Z})$ is an integer skew-symmetrizable matrix, $\mathrm{Sk}(B)$ is not necessarily an integer matrix. For example,
$B=
\left(\begin{smallmatrix}
0 & -2\\
1 & 0
\end{smallmatrix}\right)$
is an integer skew-symmetrizable matrix, but $\mathrm{Sk}(B)=\left(\begin{smallmatrix}
0 & -\sqrt{2}\\
\sqrt{2} & 0
\end{smallmatrix}\right)$
is not an integer matrix. This is the reason why we need to consider the generalization for real entries.
\end{remark}
From \Cref{lem: positive conjugation}, \Cref{prop: skew-symmetrizing method} and \Cref{rmk: rk2}, the following set of skew-symmetric matrices serves an important class for the ordinary cluster algebras.
\begin{equation}
\widehat{\mathcal{S}}_n=\{\mathrm{Sk}(B)\mid \textup{$B \in \mathrm{M}_{n}(\mathbb{Z})$ is an integer skew-symmetrizable matrix}\}.
\end{equation}
We say that every element of $\widehat{\mathcal{S}}_n$ is of \emph{quasi-integer type}, and the following classification is known. Recall from \Cref{def: quiver} that each skew-symmetric matrix $B$ can be identified with an $\mathbb{R}$-valued quiver $Q=Q(B)$.
\begin{proposition}[cf. {\cite[Prop.~1]{MS20}}]\label{prop: classification of quasi-integer type}
Let $Q = (q_{ij})$ be an $\mathbb{R}$-valued quiver.
Then, $Q$ is of quasi-integer type if and only if the following two conditions hold.
\begin{itemize}
\item For any $i,j=1,\dots,n$, we have $q_{ij}^{2} \in \mathbb{Z}$.
\item For any undirected cycle $(k_0,k_1,\dots,k_r)$ of $Q$ with $k_0=k_r$, the product of all weights
\begin{equation}\label{eq: cyclic product condition}
\prod_{j=1}^{r}q_{i_{j-1}i_j}
\end{equation}
is an integer.
\end{itemize}
\end{proposition}

\section{Sign-coherence}\label{sec: sign-coherence}
In the ordinary cluster theory, one of the most important properties is called {\em sign-coherence}. However, when we generalize the real entries, this condition does not always hold in general. In this section, we study the real $C$-, $G$-matrices under this assumption.
\subsection{Definition of sign-coherence}
In this section, we fix an initial vertex $t_0 \in \mathbb{T}_n$, and we write $C$-, $G$-matrices by $C_t$, $G_t$ for any $t\in \mbT_n$.
\par
To define the sign-coherence, we introduce a partial order $\leq$ on $\mathbb{R}^{n}$ such that each corresponding entry satisfies the inequality on $\mathbb{R}$.
\begin{definition}[Sign-coherence]\label{def: sign-coherence}
Consider the $C$-, $G$-patterns ${\bf C}^{t_0}(B)=\{C_t\}_{t \in \mathbb{T}_{n}}$ and ${\bf G}^{t_0}(B)=\{G_{t}\}_{t \in \mathbb{T}_{n}}$ with a skew-symmetrizable matrix $B \in \mathrm{M}_n(\mathbb{R})$, which is associated with an initial vertex $t_0$. 
\\
\textup{($a$)} We say that a $C$-matrix $C_t$ is {\em (column) sign-coherent} if every $c$-vector ${\bf c}_{i;t}$ ($i=1,2,\dots,n$) satisfies ${\bf c}_{i;t} \in \mathbb{R}_{\geq 0}^n \setminus \{\mathbf{0}\}$ or ${\bf c}_{i;t} \in \mathbb{R}_{\leq 0}^n \setminus \{\mathbf{0}\}$. In this case, let $\varepsilon_{i;t}^{t_0}=\varepsilon_{i;t} \in \{\pm1\}$ be the sign of this $c$-vector ${\bf c}_{i;t}$. We say that a $C$-pattern ${\bf C}^{t_0}(B)$ is {\em sign-coherent} if every $C$-matrix $C_t$ is sign-coherent.
\\
\textup{($b$)} We say that a $G$-matrix $G_{t}$ is {\em (row) sign-coherent} if every {\em row vector} of $G_t$ satisfies the similar condition. (Note that this does not mean the sign-coherence of $g$-vectors.) In this case, let $\tau^{t_0}_{i;t}=\tau_{i;t} \in \{\pm 1\}$ be the sign of the $i$th row vector of $G_{t}$. We say that a $G$-pattern ${\bf G}^{t_0}(B)$ is {\em sign-coherent} if every $G$-matrix $G_t$ is row sign-coherent.
\end{definition}
Without ambiguity, we sometimes simplify the three patterns ${\bf B}^{t_0}(B),{\bf C}^{t_0}(B),{\bf G}^{t_0}(B)$ to ${\bf B}(B),{\bf C}(B),{\bf G}(B)$, that is omitting the information of the initial vertex $t_0$.
\begin{definition}\label{def: sign-coherent calss}
Let $B \in \mathrm{M}_{n}(\mathbb{R})$ be a skew-symmetrizable matrix. 
We say that $B$ satisfies the {\em sign-coherent property} if both $C$-pattern and $G$-pattern are sign-coherent. Let ${\bf SC}$ be the set of all skew-symmetrizable matrices which satisfies the sign-coherent property, and we call it the {\em sign-coherent class}.
\par
If $C^{t_0}_{t}$ and $G^{t_0}_{t}$ are sign-coherent whenever $d(t_0,t) \leq d$, we say that $B$ satisfies the {\em sign-coherent property up to $d$}, and we write the set of all these matrices by ${\bf SC}^{\leq d}$. Similarly, we say that a $C$-pattern and a $G$-pattern are sign-coherent up to $d$.
\end{definition}
In the ordinary cluster theory, the following fact is known, and this is the essential fact to control $C$-, $G$-matrices.
\begin{theorem}[\cite{GHKK18}]\label{thm: sign-coherency for integer matrices}
Every integer skew-symmetrizable matrix belongs to the sign-coherent class ${\bf SC}$.
\end{theorem}

\begin{example}\label{ex: sign-incoherent}
In ordinary integer cluster algebra theory, all the $C$-patterns are sign-coherent. However, if we generalize the real entries, we obtain a counterexample. For example, we take an initial exchange matrix
\begin{equation}
B=\left(\begin{matrix}
0 & \frac{1}{2}\\
-\frac{1}{2} & 0
\end{matrix}\right).
\end{equation}
Then, we find a non sign-coherent $C$-matrix as follows:
\begin{equation}
\left(\begin{matrix}
1 & 0\\
0 & 1
\end{matrix}\right)
\overset{1}{\mapsto}
\left(\begin{matrix}
-1 & \frac{1}{2}\\
0 & 1
\end{matrix}\right)
\overset{2}{\mapsto}
\left(\begin{matrix}
-\frac{3}{4} & -\frac{1}{2}\\
\frac{1}{2} & -1
\end{matrix}\right).
\end{equation}
We will give more examples in Theorem~\ref{thm: rank 2 classification} and Theorem~\ref{thm: finite type classifcation}.
\end{example}

\begin{proposition}\label{prop: positive conjegation and sign-coherency}
For any skew-symmetrizable matrix $B \in \mathrm{M}_{n}(\mathbb{R})$, the following three conditions are equivalent.
\begin{itemize}
\item $B$ satisfies the sign-coherent property.
\item $\mathrm{Sk}(B)$ satisfies the sign-coherent property.
\item For any positive diagonal matrix $H \in \mathrm{M}_{n}(\mathbb{R})$, $HBH^{-1}$ satisfies the sign-coherent property.
\end{itemize}
\end{proposition}
\begin{proof}
This can be shown easily by \Cref{lem: positive conjugation} and \Cref{prop: skew-symmetrizing method}.
\end{proof}
In particular, the following class satisfies the sign-coherent property.
\begin{corollary}\label{cor: sign-coherency for quasi-integer matrices}
Every quasi-integer skew-symmetrizable matrix belongs to the sign-coherent class ${\bf SC}$. 
\end{corollary}
According to \Cref{prop: classification of quasi-integer type}, this class has been completely and clearly characterized by a combinatorial condition of quivers. 

\subsection{Recursion and second duality under the sign-coherence}
When we assume the sign-coherence of $C$-matrices, we simplify the recursion \eqref{eq: mutation of C, G matrices} as follows, see also \cite{NZ12}.
\begin{lemma}
Let $B \in \mathrm{M}_{n}(\mathbb{R})$ be a skew-symmetrizable matrix. Let $t \in \mathbb{T}_{n}$ and suppose that this $C$-matrix $C_{t}$ is sign-coherent. Then, for any $k$-adjacent vertex $t' \in \mathbb{T}_{n}$ to $t$, we have
\begin{equation}\label{CG rec}
\begin{aligned}
C_{t'}&=C_t(J_k+[\varepsilon_{k;t}B_{t}]^{k \bullet}_{+}),\\
G_{t'}&=G_{t}(J_k+[-\varepsilon_{k;t}B_t]^{\bullet k}_{+}).
\end{aligned}
\end{equation}
Moreover, for any $i=1,2,\dots,n$, we obtain the following recursions for $c$- and $g$-vectors.
\begin{equation}\label{eq: mutation for c-,g-vectors}
\begin{aligned}
{\bf c}_{i;t'}&=\begin{cases}
-{\bf c}_{k;t} & i=k,\\
{\bf c}_{i;t}+[\varepsilon_{k;t}b_{ki;t}]_{+}{\bf c}_{k;t} & i \neq k,
\end{cases}
\\
{\bf g}_{i;t'}&=\begin{cases}
-{\bf g}_{k;t}+\sum_{j=1}^{n}[-\varepsilon_{k;t}b_{jk;t}]_{+}{\bf g}_{j;t} & i=k,\\
{\bf g}_{i;t} & i \neq k.
\end{cases}
\end{aligned}
\end{equation}
\end{lemma}
Based on this recursion, we obtain the following fundamental properties of $C$-, $G$-matrices.
\begin{proposition}\label{prop: fundamental properties under sign-coherency}
Let $B \in \mathrm{M}_n(\mathbb{R})$ be a skew-symmetrizable matrix with a skew-symmetrizer $D$. Suppose that its $C$-pattern ${\bf C}(B)$ is sign-coherent up to $d \in \mathbb{Z}_{\geq 0}$. Then, for any $t \in \mathbb{T}_{n}$ with $d(t_0,t)=d+1$, the following statements hold. \textup{(}We do not have to assume the sign-coherence of this $C_{t}$.\textup{)}
\\
\textup{$(a)$} We have $|C_{t}|=|G_{t}| \in \{\pm 1\}$. In particular, $C_{t}$ and $G_{t}$ are unimodular matrices over $\mathbb{Z}_{B}$. Namely,
\begin{itemize}
\item every entry of their inverse matrices $C_{t}^{-1}$, $G_{t}^{-1}$ also belongs to $\mathbb{Z}_{B}$.
\item each $\{{\bf c}_{i;t} \mid i=1,\dots,n\}$ and $\{{\bf g}_{i;t} \mid i=1,\dots,n\}$ is a basis of $(\mathbb{Z}_B)^{\times n}$ as a free $\mathbb{Z}_{B}$-module.
\end{itemize}
\textup{$(b)$} The second duality relation holds:
\begin{equation}\label{eq: second duality}
D^{-1}G_{t}^{\top}DC_{t}=I.
\end{equation}
\textup{($c$)} We have
\begin{equation}\label{eq: from C to B}
DB_t=C_{t}^{\top}DB_{t_0}C_{t}.
\end{equation}
\end{proposition}
\begin{proof} Firstly, according to \eqref{CG rec}, by the fact that 
\begin{align}
|J_k+[\varepsilon_{k;t}B_{t}]^{k \bullet}_{+}|=|J_k+[-\varepsilon_{k;t} B_{t}]^{\bullet k}_{+}|=-1
\end{align}
and the induction, we obtain that $|C_{t}|=|G_{t}| \in \{\pm 1\}$ for any $t \in \mathbb{T}_{n}$. Then, by Lemma~\ref{prop: ring ZB}, we have $C_{t}, G_{t}\in \mathrm{M}_{n}(\mathbb{Z}_B)$. Thus, $C_t$ and $G_t$ are unimodular matrices. Thus, the claim $(a)$ holds. The proof of claims $(b)$ and $(c)$ can be referred to \cite[Eq.(3.11)]{NZ12}, \cite[Prop 2.3]{Nak23} and \cite[Eq. (2.9)]{NZ12}, \cite[Prop. 2.6]{Nak23} respectively.
\end{proof}
\subsection{Geometric property under the sign-coherence}\label{sec: geometric properties under the sign-coherency}
The second duality (\ref{eq: second duality}) can be seen as a geometric properties in $c$-, $g$-vectors.
In this section, we fix one initial exchange matrix $B \in \mathrm{M}_{n}(\mathbb{R})$ with a skew-symmetrizer $D=\mathrm{diag}(d_1,\dots,d_n)$. We introduce an inner product $\langle\,,\,\rangle_{D}$ on $\mathbb{R}^n$ by
\begin{equation}\label{eq: definition of inner product}
\langle {\bf a},{\bf b}\rangle_{D}={\bf a}^{\top}D{\bf b}.
\end{equation}
Note that the $(i,j)$th entry of $G_{t}^{\top}DC_{t}$ is $\langle {\bf g}_{i;t},{\bf c}_{j;t}\rangle_{D}$. Thus, by considering (\ref{eq: second duality}), we obtain the following geometric relationship between $c$-, $g$-vectors.
\begin{proposition}[cf. {\cite[Prop.~2.16]{Nak23}}]\label{prop: orhogonal relations}
Let $B \in \mathrm{M}_{n}(\mathbb{R})$ be a skew-symmetrizable matrix. Suppose that its $C$-pattern ${\bf C}(B)$ is sign-coherent up to $d \in \mathbb{Z}_{\geq 0}$. Then, for any $t \in \mathbb{T}_{n}$ and $i,j=1,\dots,n$, if $d(t_0,t) \leq d+1$, we have
\begin{equation}\label{eq: orthogonal lemma}
\langle {\bf g}_{i;t},{\bf c}_{j,t} \rangle_{D}= \begin{cases}
d_i & i=j,\\
0 & i \neq j.
\end{cases}
\end{equation}
\end{proposition}
Based on this property, we can rephrase the sign-coherence of $c$-vectors into the geometric property of $g$-vectors. To state it, we introduce the notion of cone.
\begin{definition}[cone]
\textup{($a$)} Let ${\bf a}_{1},\dots,{\bf a}_{r}$ be a set of vectors. Then, the following set $\mathcal{C}({\bf a}_1,\dots,{\bf a}_{r})$ is called a {\em \textup{(}polyhedral\textup{)} cone}.
\begin{equation}
\mathcal{C}({\bf a}_1,\dots,{\bf a}_{r})=\left\{\left.\sum_{i=1}^{r}\lambda_{i}{\bf a}_{i}\  \right|\ \lambda_{i} \geq 0\right\} \subset \mathbb{R}^n.
\end{equation}
If we can take ${\bf a_1},\dots,{\bf a}_{r}$ as linearly independent vectors, we say that a cone $\mathcal{C}({\bf a}_1,\dots,{\bf a}_r)$ is {\em simplicial}.
Conventionally, we write $\mathcal{C}(\emptyset)=\{{\bf 0}\}$, and we also call it a simplicial cone.
We denote its {\em relative interior} by $\mathcal{C}^{\circ}({\bf a}_1,\dots,{\bf a}_{r})$, which is the interior of $\mathcal{C}({\bf a}_1,\dots,{\bf a}_{r})$ in the linear subspace spanned by $\mathcal{C}({\bf a}_1,\dots,{\bf a}_{r})$. In particular, the relative interior of a simplicial cone $\mathcal{C}({\bf a}_1,\dots,{\bf a}_{r})$ is given by
\begin{equation}
\mathcal{C}^{\circ}({\bf a}_1,\dots,{\bf a}_{r})=\left\{\left.\sum_{i=1}^{r}\lambda_{i}{\bf a}_{i}\  \right|\ \lambda_{i} > 0\right\}.
\end{equation}
\textup{($b$)} For any cone $\mathcal{C}({\bf a}_1,\dots,{\bf a}_{r})$ and $J \subset \{1,\dots,r\}$, we define the following set.
\begin{equation}
\mathcal{C}_{J}({\bf a}_1,\dots,{\bf a}_{r})=\left\{\left.\sum_{j \in J}\lambda_{j}{\bf a}_{j}\ \right|\ \lambda_{j} \geq 0\right\}.
\end{equation}
When ${\bf a}_1,\dots,{\bf a}_r$ are linearly independent, this set is called a {\em face} of $\mathcal{C}({\bf a}_1,\dots,{\bf a}_r)$.
Note that $\mathcal{C}_{\{1,\dots,r\}}({\bf a}_1,\dots,{\bf a}_{r})=\mathcal{C}({\bf a}_1,\dots,{\bf a}_{r})$ is a face of the cone $\mathcal{C}({\bf a}_1,\dots,{\bf a}_{r})$ itself. Conventionally, we set the trivial face $\mathcal{C}_{\emptyset}({\bf a}_1,\dots,{\bf a}_{r})=\{{\bf 0}\}$.
\end{definition}
\begin{definition}[$G$-cone]\label{def: G-cone}
Let $B \in \mathrm{M}_{n}(\mathbb{R})$ be a skew-symmetrizable matrix. For each $G$-matrix $G_{t}=({\bf g}_{1;t},{\bf g}_{2;t},\dots,{\bf g}_{n;t})$, we call the following set $\mathcal{C}(G_{t})$ a {\em $G$-cone}.
\begin{equation}
\mathcal{C}(G_t)=\mathcal{C}({\bf g}_{1;t}, \dots, {\bf g}_{n;t}).
\end{equation}
We write $\mathcal{C}_J(G_t)=\mathcal{C}_{J}({\bf g}_{1;t},\dots,{\bf g}_{n;t})$ for $J \subset \{1,2,\dots,n\}$. 
\end{definition}
By Proposition~\ref{prop: fundamental properties under sign-coherency}~($a$), if the $C$-pattern ${\bf C}(B)$ is sign-coherent, every $G$-cone $\mathcal{C}(G_t)$ is simplicial.
\par
For each ${\bf v} \in \mathbb{R}^n \setminus \{\mathbf{0}\}$, we define
\begin{equation}
\begin{aligned}
\mathcal{H}_{\bf v}&=\{{\bf x} \in \mathbb{R}^n \mid \langle {\bf x},{\bf v}\rangle_{D}=0\},\\
\mathcal{H}_{\bf v}^{+}&=\{{\bf x} \in \mathbb{R}^n \mid \langle {\bf x},{\bf v}\rangle_{D}>0\},\\
\mathcal{H}_{\bf v}^{-}&=\{{\bf x} \in \mathbb{R}^n \mid \langle {\bf x},{\bf v}\rangle_{D}<0\}.
\end{aligned}
\end{equation}
We write $\overline{\mathcal{H}}^{+}_{\bf v}=\mathcal{H}^{+}_{\bf v}\cup\mathcal{H}_{\bf v}$ and $\overline{\mathcal{H}}^{-}_{\bf v}=\mathcal{H}^{-}_{\bf v}\cup\mathcal{H}_{\bf v}$, which are their closures.
Then, we express $G$-cones by $c$-vectors.
\begin{lemma}\label{lem: c-vector expression of G-cones}
Let $B \in \mathrm{M}_n(\mathbb{R})$ be a skew-symmetrizable matrix with a skew-symmetrizer $D$. Suppose that ${\bf C}(B)$ is sign-coherent up to $d \in \mathbb{Z}_{\geq 0}$. Then, for any $t \in \mathbb{T}_{n}$ with $d(t_0,t) \leq d+1$, we have
\begin{equation}
\mathcal{C}(G_t)=\bigcap_{i=1}^{n} \overline{\mathcal{H}}_{{\bf c}_{i;t}}^{+}.
\end{equation}
\end{lemma}
\begin{proof}
By (\ref{eq: orthogonal lemma}), it is direct that one inclusion $\mathcal{C}(G_t) \subset \bigcap_{i=1}^{n}\overline{\mathcal{H}}_{{\bf c}_{i;t}}^{+}$ holds. Now, we aim to show $\bigcap_{i=1}^{n}\overline{\mathcal{H}}_{{\bf c}_{i;t}}^{+} \subset \mathcal{C}(G_t)$. Take any ${\bf x} \in \bigcap_{i=1}^{n}\overline{\mathcal{H}}_{{\bf c}_{i;t}}^{+}$. Since $\{{\bf g}_{1;t},\dots,{\bf g}_{n;t}\}$ is a basis of $\mathbb{R}^n$, we express ${\bf x}=\sum_{i=1}^{n}x_i{\bf g}_{i;t}$. By (\ref{eq: orthogonal lemma}), we have
\begin{equation}
d_ix_i=\langle {\bf x},{\bf c}_{i;t}\rangle_{D} \geq 0.
\end{equation}
In particular, $x_i \geq 0$ holds for any $i=1,\dots,n$. Thus, we have ${\bf x}=\sum x_i{\bf g}_{i;t} \in \mathcal{C}(G_t)$.
\end{proof}
Through the above expression, sign-coherence imposes the strong restriction of $G$-cones.
\begin{proposition}[cf.~{\cite[Cond.~II.2.28]{Nak23}}]\label{prop: lemma for cones}
Let $B \in \mathrm{M}_{n}(\mathbb{R})$ be a skew-symmetrizable matrix with a skew-symmetrizer $D$. Suppose that $B$ satisfies the sign-coherent property. Then, we have the following statements.
\\
\textup{($a$)} Every $G$-cone $\mathcal{C}(G_t)$ is a subset of the orthant $\mathfrak{O}_{\tau_{1;t},\tau_{2;t},\dots,\tau_{n;t}}$.
\\
\textup{($b$)} The intersection between a $G$-cone $\mathcal{C}(G_t)$ and the positive orthant $\mathfrak{O}_{+}^{n}$ \textup{(}resp. the negative orthant $\mathfrak{O}_{-}^{n}$\textup{)} is expressed as
\begin{equation}\label{eq: G-cones and positive orthant}
\mathcal{C}(G_t) \cap \mathfrak{O}_{+}^{n} = \mathcal{C}_{J}({\bf e}_1,\dots,{\bf e}_n) \quad (\textup{resp}.\ \mathcal{C}(G_t) \cap \mathfrak{O}_{-}^{n} = \mathcal{C}_{J}(-{\bf e}_1,\dots,-{\bf e}_n))
\end{equation}
for some $J \subset \{1,2,\dots,n\}$.
\end{proposition}
The claim ($a$) follows from the row sign-coherence of $G_t$. The claim ($b$) can be shown by a similar argument to that in \cite[Cond.~II.2.28]{Nak23}. More precisely, in \cite{Nak23}, it was shown there that, if a $g$-vector $\mathbf{g}_{i;t}$ is contained in $\mathfrak{O}_{+}^n$, then $\mathbf{g}_{i;t}=\mathbf{e}_l$ holds for some $l \in \{1,\dots,n\}$. Although this is no longer true for the real case (see \Cref{ex: non uniqueness of the length}), we conclude that $\mathcal{C}(\mathbf{g}_{i;t})=\mathcal{C}(\mathbf{e}_l)$ holds instead, which is sufficient to prove ($b$).

\section{Conjectures for real $C$-, $G$-matrices}\label{sec: conjectures}
For the real $C$-, $G$-matrices, there are some significant differences from the integer ones. In this section, we introduce two conjectures to overcome these differences.
\subsection{Totally sign-coherence conjecture}
First conjecture is for the sign-coherence. Recall that ${\bf SC}$ means the set of all skew-symmetrizable matrices such that they satisfy the sign-coherent property, which means that all corresponding $C$-, $G$-matrices are sign-coherent. When we discuss a given exchange matrix $B$ and the corresponding $C$-, $G$-pattern ${\bf C}(B)$ and ${\bf G}(B)$, we sometimes want to suppose that other matrices related to $B$ also belong to ${\bf SC}$. Here, we define a class of exchange matrices under this assumption.
This conjecture has essentially appeared in \cite{Rea14}, and this is an important assumption to consider changing the initial exchange matrix in the same mutation-equivalent class.
\begin{conjecture}[{cf. \cite[Def.~8.2]{Rea14}} Totally sign-coherence conjecture]\label{conj: standard hypothesis}
If $B \in \mathrm{M}_{n}(\mathbb{R})$ satisfies the sign-coherent property, then all mutation-equivalent matrices $B' \in {\bf B}(B)$ also satisfy the sign-coherent property.
\end{conjecture}
\begin{remark}
In \cite{Rea14}, this conjecture was given as the condition $(c)$ in \Cref{cor: equivalency of sign-coherent class} and this is called the {\em standard hypotheses}. Although these two conditions are equivalent under one assumption, we choose the statement as in Conjecture~\ref{conj: standard hypothesis} because this form can be stated in one single $B$-pattern. (For the other forms in \Cref{cor: equivalency of sign-coherent class}, we need to consider other $B$-patterns such as ${\bf B}(B^{\top})$ and ${\bf B}(-B)$.)
\end{remark}

\subsection{Discreteness conjecture}
Another conjecture is for the periodicity of $c$-vectors, and this property becomes trivial when we focus on the integer case.
\begin{conjecture}[Discreteness conjecture]\label{conj: discreteness conjecture}
Let $B \in {\bf SC}$ be a real exchange matrix with a skew-symmetrizer $D=\mathrm{diag}(d_1,\dots,d_n)$. If there exists a $c$-vector ${\bf c}_{i;t}$ \textup{($i=1,\dots,n$)} satisfying ${\bf c}_{i;t}=\alpha{\bf e}_{j}$ for some $j=1,\dots,n$, then we have
\begin{equation}
\alpha=\pm\sqrt{\frac{d_{i}}{d_j}}.
\end{equation}
\end{conjecture}
This conjecture can be rephrased by $g$-vectors.
\begin{lemma}\label{lem: for the parallel lemma}
Let $B \in {\bf SC}$ with a skew-symmetrizer $D=\mathrm{diag}(d_1,d_2,\dots,d_n)$, and let $i,j=1,2,\dots,n$ and $t \in \mathbb{T}_{n}$.
\\
\textup{($a$)} A $c$-vector ${\bf c}_{i;t}$ is expressed ${\bf c}_{i;t}=\alpha{\bf e}_{j}$ for some $\alpha \in \mathbb{R}$ if and only if every $j$th entry of g-vectors ${\bf g}_{l;t}$ is $0$ except for $l = i$.
\\
\textup{($b$)} Suppose that the condition in (a) holds. Let $\beta$ be the $j$th entry of ${\bf g}_{i;t}$. Then, we have $\alpha\beta=d_id_{j}^{-1}$.
\par
In particular, for the above $\alpha,\beta$, the following three conditions are equivalent.
\begin{itemize}
\item $\alpha = \pm \sqrt{d_{i}d_{j}^{-1}}$. \textup{(}Conjecture~\ref{conj: discreteness conjecture}\textup{)}
\item $\beta = \pm \sqrt{d_{i}d_{j}^{-1}}$.
\item $\alpha=\beta$.
\end{itemize}
\end{lemma}
\begin{proof}
This is essentially shown by Proposition~\ref{prop: orhogonal relations}. Suppose that ${\bf c}_{i;t} = \alpha{\bf e}_{j}$. Then, for any $l \neq i$, we have $0=\langle {\bf g}_{l;t},{\bf c}_{i;t} \rangle_{D}=\alpha\langle {\bf g}_{l;t},{\bf e}_{j} \rangle_{D}$. Since $\alpha \neq 0$ (if $\alpha=0$, it contradicts with $|C_{t}| \neq 0$), we have $\langle {\bf g}_{l;t},{\bf e}_{j}\rangle_{D}=0$ and this implies that the $j$th entry of ${\bf g}_{l;t}$ is zero. Conversely, suppose that the $j$th entry of ${\bf g}_{l;t}$ is $0$ except for $l=i$. By Proposition~\ref{prop: orhogonal relations}, ${\bf c}_{i;t}$ should belong to the orthogonal complement of $\langle {\bf g}_{l;t} \mid l \neq i\rangle_{\mathrm{vec}}$. Then, by the assumption, ${\bf e}_{j}$ should belong to its orthogonal complement $\langle {\bf g}_{l;t} \mid l \neq i\rangle_{\mathrm{vec}}^{\perp}$. Since $\{{\bf g}_{l;t} \mid l=1,\dots,n\}$ is a basis of $\mathbb{R}^{n}$, the dimension of $\langle {\bf g}_{l;t} \mid l \neq i\rangle_{\mathrm{vec}}^{\perp}$ is one. Thus, it should be spanned by ${\bf e}_j$. In particular, ${\bf c}_{i;t}=\alpha{\bf e}_{j;t}$ for some $\alpha \in \mathbb{R}$. Thus, we conclude that $(a)$ holds. Let $\beta$ be the $j$th entry of ${\bf g}_{i;t}$. Then, by Proposition~\ref{prop: orhogonal relations}, we have $d_{i}=\langle {\bf g}_{i;t},{\bf c}_{i;t} \rangle_{D}=\alpha\langle {\bf g}_{i;t},{\bf e}_{j}\rangle_{D}=d_{j}\alpha\beta$. Thus, $\alpha\beta=d_id_j^{-1}$ holds. The equivalency of three conditions can be shown directly by this equality.
\end{proof}
\begin{example}\label{ex: non uniqueness of the length}
This conjecture is not emphasized in the ordinary cluster algebras because we can easily show it as more stronger condition, see Proposition~\ref{prop: discreteness lemma for integer case}. However, when we consider the real case, this problem seems to be not so easy. To support this conjecture, we give one example which is not the integer case. Set the initial exchange matrx $B=\left(\begin{smallmatrix}
0 & -\frac{1}{2}\\
2 & 0
\end{smallmatrix}\right)$. Note that we can take a skew-symmetrizer $D=\mathrm{diag}(4,1)$. Then, the $C$-pattern is in Figure~\ref{fig: C-pattern with quasi-integer type} and the $G$-pattern is in Figure~\ref{fig: G-pattern with quasi-integer type}. (By this calculation, we show that this $B$ satisfies the sign-coherent property.) Focus on a $C$-matrix and a $G$-matrix inside the boxes. Then, its $c$-vector located on the first column is parallel to ${\bf e}_2$. The length of this vector is $2$. Similarly, the $c$-vector located on the second column is parallel to ${\bf e}_1$, and its length is $\frac{1}{2}$. Thus, Conjecture~\ref{conj: discreteness conjecture} is true for this $B$. We can also see the equivalent phenomenon for $G$-matrices as in Lemma~\ref{lem: for the parallel lemma}.
\begin{figure}[htbp]
\centering
\begin{tikzpicture}
\node (t0) at (0,0) {$\left(\begin{smallmatrix}
1 & 0\\
0 & 1
\end{smallmatrix}\right)$};
\node (t1) at (2,0) {$\left(\begin{smallmatrix}
-1 & 0\\
0 & 1
\end{smallmatrix}\right)$};
\node (t2) at (4,0) {$\left(\begin{smallmatrix}
-1 & 0\\
0 & -1
\end{smallmatrix}\right)$};
\node (t3) at (6,0) {$\left(\begin{smallmatrix}
1 & -\frac{1}{2}\\
0 & -1
\end{smallmatrix}\right)$};
\node (t4) at (8,0) {$\left(\begin{smallmatrix}
0 & \frac{1}{2}\\
-2 & 1
\end{smallmatrix}\right)$};
\node [draw] (t5) at (8,-1.5) {$\left(\begin{smallmatrix}
0 & \frac{1}{2}\\
2 & 0
\end{smallmatrix}\right)$};
\node (t6) at (6,-1.5) {$\left(\begin{smallmatrix}
0 & -\frac{1}{2}\\
2 & 0
\end{smallmatrix}\right)$};
\node (t7) at (4,-1.5) {$\left(\begin{smallmatrix}
0 & -\frac{1}{2}\\
-2 & 0
\end{smallmatrix}\right)$};
\node (t8) at (2,-1.5) {$\left(\begin{smallmatrix}
-1 & \frac{1}{2}\\
-2 & 0
\end{smallmatrix}\right)$};
\node (t9) at (0,-1.5) {$\left(\begin{smallmatrix}
1 & 0\\
2 & -1
\end{smallmatrix}\right)$};
\draw[<->] (t0.east)--($(t0)!0.5!(t1)$) node [above] {$1$}->(t1.west);
\draw[<->] (t1.east)--($(t1)!0.5!(t2)$) node [above] {$2$}->(t2.west);
\draw[<->] (t2.east)--($(t2)!0.5!(t3)$) node [above] {$1$}->(t3.west);
\draw[<->] (t3.east)--($(t3)!0.5!(t4)$) node [above] {$2$}->(t4.west);
\draw[<->] (t4.south)--($(t4)!0.5!(t5)$) node [right] {$1$}->(t5.north);
\draw[<->] (t5.west)--($(t5)!0.5!(t6)$) node [above] {$2$}->(t6.east);
\draw[<->] (t6.west)--($(t6)!0.5!(t7)$) node [above] {$1$}->(t7.east);
\draw[<->] (t7.west)--($(t7)!0.5!(t8)$) node [above] {$2$}->(t8.east);
\draw[<->] (t8.west)--($(t8)!0.5!(t9)$) node [above] {$1$}->(t9.east);
\draw[<->] (t9.north)--($(t9)!0.5!(t0)$) node [left] {$2$}->(t0.south);
\end{tikzpicture}
\caption{$C$-pattern of $B=\left(\begin{smallmatrix}
0 & -\frac{1}{2}\\
2 & 0
\end{smallmatrix}\right)$}\label{fig: C-pattern with quasi-integer type}
\end{figure}
\begin{figure}[htbp]
\centering
\begin{tikzpicture}
\node (t0) at (0,0) {$\left(\begin{smallmatrix}
1 & 0\\
0 & 1
\end{smallmatrix}\right)$};
\node (t1) at (2,0) {$\left(\begin{smallmatrix}
-1 & 0\\
0 & 1
\end{smallmatrix}\right)$};
\node (t2) at (4,0) {$\left(\begin{smallmatrix}
-1 & 0\\
0 & -1
\end{smallmatrix}\right)$};
\node (t3) at (6,0) {$\left(\begin{smallmatrix}
1 & 0\\
-2 & -1
\end{smallmatrix}\right)$};
\node (t4) at (8,0) {$\left(\begin{smallmatrix}
1 & \frac{1}{2}\\
-2 & 0
\end{smallmatrix}\right)$};
\node [draw] (t5) at (8,-1.5) {$\left(\begin{smallmatrix}
0 & \frac{1}{2}\\
2 & 0
\end{smallmatrix}\right)$};
\node (t6) at (6,-1.5) {$\left(\begin{smallmatrix}
0 & -\frac{1}{2}\\
2 & 0
\end{smallmatrix}\right)$};
\node (t7) at (4,-1.5) {$\left(\begin{smallmatrix}
0 & -\frac{1}{2}\\
-2 & 0
\end{smallmatrix}\right)$};
\node (t8) at (2,-1.5) {$\left(\begin{smallmatrix}
0 & \frac{1}{2}\\
-2 & -1
\end{smallmatrix}\right)$};
\node (t9) at (0,-1.5) {$\left(\begin{smallmatrix}
1 & \frac{1}{2}\\
0 & -1
\end{smallmatrix}\right)$};
\draw[<->] (t0.east)--($(t0)!0.5!(t1)$) node [above] {$1$}->(t1.west);
\draw[<->] (t1.east)--($(t1)!0.5!(t2)$) node [above] {$2$}->(t2.west);
\draw[<->] (t2.east)--($(t2)!0.5!(t3)$) node [above] {$1$}->(t3.west);
\draw[<->] (t3.east)--($(t3)!0.5!(t4)$) node [above] {$2$}->(t4.west);
\draw[<->] (t4.south)--($(t4)!0.5!(t5)$) node [right] {$1$}->(t5.north);
\draw[<->] (t5.west)--($(t5)!0.5!(t6)$) node [above] {$2$}->(t6.east);
\draw[<->] (t6.west)--($(t6)!0.5!(t7)$) node [above] {$1$}->(t7.east);
\draw[<->] (t7.west)--($(t7)!0.5!(t8)$) node [above] {$2$}->(t8.east);
\draw[<->] (t8.west)--($(t8)!0.5!(t9)$) node [above] {$1$}->(t9.east);
\draw[<->] (t9.north)--($(t9)!0.5!(t0)$) node [left] {$2$}->(t0.south);
\end{tikzpicture}
\caption{$G$-pattern of $B=\left(\begin{smallmatrix}
0 & -\frac{1}{2}\\
2 & 0
\end{smallmatrix}\right)$}\label{fig: G-pattern with quasi-integer type}
\end{figure}
\end{example}

\begin{example}
Here, we present one counterexample of this conjecture when the sign-coherence fails. Take an initial exchange matrix as
\begin{equation}
B_{t_0}=\left(\begin{matrix}
0 & 2\cos{\frac{2}{5}\pi} & 0\\
-2\cos{\frac{2}{5}\pi} & 0 & 1\\
0 & -1 & 0
\end{matrix}\right).
\end{equation}
By applying mutations along $[2,1,3,1,2,1,3]$, we obtain the following $C$-matrix:
\begin{equation}
\mu_3\mu_1\mu_2\mu_1\mu_3\mu_1\mu_2(C_{t_0})=\left(\begin{smallmatrix}
\frac{9-5\sqrt{5}}{2} & \frac{-7+3\sqrt{5}}{2} & \frac{3-\sqrt{5}}{2}\\
0 & 0 & 1\\
0 & -1 & 0
\end{smallmatrix}\right).
\end{equation}
The first column vector fails to satisfy the condition in \Cref{conj: discreteness conjecture}.
\end{example}

For some class including all integer cases, this conjecture can be shown as follows.
\begin{proposition}\label{prop: discreteness lemma for integer case}
Let $B \in {\bf SC}$ with a skew-symmetrizer $D=\mathrm{diag}(d_1,d_2,\dots,d_n)$. If the group of units $(\mathbb{Z}_{B})^{\times}$ of the ring $\mathbb{Z}_{B}$ is trivial, that is
\begin{equation}
(\mathbb{Z}_{B})^{\times}=\{\pm 1\},
\end{equation}
then ${\bf c}_{i;t}=\alpha{\bf e}_{j}$ implies that $\alpha=\pm1$ and $d_i=d_j$. In particular, Conjecture~\ref{conj: discreteness conjecture} holds.
\end{proposition}
\begin{proof}
Since ${\bf c}_{i;t}=\alpha{\bf e}_{j}$, all entries of the $i$th column of the $C$-matrix $C_t$ is $0$ except for $\alpha$. Thus, its determinant is expressed as $|C_{t}|=\alpha|A|$, where $A$ is a matrix obtained by eliminating $j$th row and $i$th column from $C_t$. Since $C_t \in \mathrm{M}_n(\mathbb{Z}_{B})$, we have $|A| \in \mathbb{Z}_{B}$. By Proposition~\ref{prop: fundamental properties under sign-coherency}, it implies that $\alpha|A|=|C_t|=\pm 1$. In particular, $\alpha$ is a unit element of $\mathbb{Z}_{B}$. Since $(\mathbb{Z}_{B})^{\times}=\{\pm 1\}$, we have $\alpha = \pm 1$. Next, we show $d_i=d_j$. By Lemma~\ref{lem: for the parallel lemma}, every entry of the $j$th row in $G_t$ is $0$ except for the $i$th one. Moreover, this $i$th entry $\beta$ is given by $\beta=\pm d_{i}d_j^{-1}$ since $\alpha\beta=d_{i}d_j^{-1}$. Let $A'$ be the matrix obtained by eliminating $j$th row and $i$th column from $G_t$. Then, we have $\beta|A'|= \pm 1$. In particular, $\beta$ is also a unit element of $\mathbb{Z}_{B}$. Since $d_i,d_j > 0$, we have $d_{i}d_{j}^{-1}=1$. This completes the proof. 
\end{proof}
\begin{remark}
For the ordinary cluster algebras, since $B \in \mathrm{M}_{n}(\mathbb{Z})$, then we have $(\mathbb{Z}_{B})^{\times}=\mbZ^{\times}=\{\pm 1\}$ and the property above holds.
\end{remark}
Moreover, this property is preserved under the positive conjugations.
\begin{proposition}\label{prop: positive conjugation and discreteness lemma}
Let $B \in \mathrm{M}_{n}(\mathbb{R})$ be a skew-symmetrizable matrix. Suppose that Conjecture~\ref{conj: discreteness conjecture} holds for this $B$. Then, for any positive diagonal matrix $H \in \mathrm{M}_{n}(\mathbb{R})$, $HBH^{-1}$ also satisfies the Conjecture~\ref{conj: discreteness conjecture}.
\end{proposition}
\begin{proof}
Let $H=\mathrm{diag}(h_1,\dots,h_n)$ and $\hat{B}=HBH^{-1}$. Note that we can take a skew-symmetrizer $H^{-1}DH^{-1}=\mathrm{diag}(d_1h_1^{-2},d_2h_2^{-2},\dots,d_nh_n^{-2})$. Suppose that $\hat{C}_{t} \in {\bf C}(\hat{B})$ satisfies the assumption of Conjecture~\ref{conj: discreteness conjecture}, that is, its $i$th column vector $\hat{\bf c}_{i;t}$ satisfies $\hat{\bf c}_{i;t}=\alpha {\bf e}_{j}$. Then, by Lemma~\ref{lem: positive conjugation}, the original $C$-matrix $C_t$ satisfies $\hat{C}_t=H{C}_tH^{-1}$, and it implies $C_{t}=H^{-1}\hat{C}_tH$. This induces 
\begin{equation}
{\bf c}_{i;t}=h_iH^{-1}\hat{\bf c}_{i;t}=h_iH^{-1}(\alpha{\bf e}_{j})=\alpha h_ih_j^{-1}{\bf e}_j.
\end{equation}
Thus, ${\bf c}_{i;t}$ also satisfies the assumption of Conjecture~\ref{conj: discreteness conjecture}. Then, we have
$\alpha h_ih_j^{-1} = \sqrt{d_id_j^{-1}}$,
and it implies the following desired equality:
\begin{equation}
\alpha=\frac{h_j}{h_i}\sqrt{\frac{d_i}{d_j}}=\sqrt{\dfrac{d_ih_i^{-2}}{d_jh_j^{-2}}}.
\end{equation}
\end{proof}
For some technical reasons, we sometimes need to assume that both Conjecture~\ref{conj: standard hypothesis} and Conjecture~\ref{conj: discreteness conjecture} hold for all the mutation-equivalent matrices. To refer to these conjectures, we combine them together as the following conjecture.
\begin{conjecture}\label{conj: standard and discreteness conjecture}
Let $B \in {\bf SC}$. Then, for any $B' \in {\bf B}(B)$, Conjecture~\ref{conj: standard hypothesis} and Conjecture~\ref{conj: discreteness conjecture} hold.
\end{conjecture}
\begin{proposition}
Every skew-symmetrizable matrix of quasi-integer type satisfies Conjecture~\ref{conj: standard and discreteness conjecture}.
\end{proposition}
\begin{proof}
For any quasi-integer type matrix $B \in \mathrm{M}_{n}(\mathbb{R})$, we can show that its $B$-pattern consists of quasi-integer type matrices by Lemma~\ref{lem: positive conjugation}. Thus, this is shown by \Cref{cor: sign-coherency for quasi-integer matrices}, Proposition~\ref{prop: discreteness lemma for integer case}, and Proposition~\ref{prop: positive conjugation and discreteness lemma}.
\end{proof}

\section{Dual mutation and $G$-fan under the sign-coherence}\label{sec: dual and fan}
In this section, we investigate the dual mutation and the $G$-fan structure under \Cref{conj: standard and discreteness conjecture} for the real case, which generalizes the ordinary integer case.
\subsection{Dual mutation and third duality} For each $t_0 \in \mathbb{T}_n$ (not assuming $B_{t_0}=B$), we consider the $C$- and $G$-patterns with the initial vertex $t_0$. We write them by ${\bf C}^{t_0}(B_{t_0})=\{C^{t_0}_t\}$ and ${\bf G}^{t_{0}}(B_{t_0})=\{G^{t_0}_t\}$. We consider its $B$-pattern $\tilde{\bf B}=\{\tilde{B}_t\}$, which satisfies $\tilde{B}_t=B_t^{\top}.$ Recall that for any $t_0,t \in \mathbb{T}_{n}$, $\varepsilon_{i;t}^{t_0}$ is the sign of the $i$th column vector of $C^{t_0}_{t}$ and $\tau_{i;t}^{t_0}$ is the sign of the $i$th row vector of $G^{t_0}_{t}$.

Based on \cite[Prop. 1.4]{NZ12}, we obtain the following proposition for the real sign-coherent case. However, there are still some essential differences caused by \Cref{conj: discreteness conjecture}.
\begin{proposition}\label{prop: dual mutation}
Let $B \in {\bf SC}$. Suppose that Conjecture~\ref{conj: standard and discreteness conjecture} holds for this $B$. Then, for any $t_0,t \in \mathbb{T}_{n}$, the following statements hold.
\\
\textup{($a$)} We have the third duality relation:
\begin{equation}\label{eq: thied duality}
C^{t_0}_{t}=(\tilde{G}^{t}_{t_0})^{\top},\quad
G^{t_0}_{t}=(\tilde{C}^{t}_{t_0})^{\top}.
\end{equation}
In particular, we have $\tilde{\varepsilon}^{t}_{k;t_0}=\tau^{t_0}_{k;t}$ and $\tilde{\tau}^{t}_{k;t_0}=\varepsilon^{t_0}_{k;t}$ for any $k=1,2,\dots,n$.
\\
\textup{($b$)} Conjecture~\ref{conj: standard and discreteness conjecture} holds for $B^{\top}$.
\\
\textup{($c$)} For any $k=1,\dots,n$, set $t_1$ as the $k$-adjacent vertex to $t_0$. Then, we have
\begin{equation}\label{eq: dual mutation formula}
\begin{aligned}
C^{t_1}_{t}&=(J_k+[-\tau^{t_0}_{k;t}B_{t_0}]_{+}^{k \bullet})C^{t_0}_{t},\\
G^{t_1}_{t}&=(J_k+[\tau^{t_0}_{k;t}B_{t_0}]^{\bullet k}_{+})G^{t_0}_{t}.
\end{aligned}
\end{equation}
\end{proposition}
\begin{proof}
For the most part, the proof in \cite{NZ12} remains valid. In their proof, they assumed the sign-coherence of $\mathbf{C}(B_t)$ and $\mathbf{C}(B_t^{\top})$ for any $B' \in \Gamma(B)$. Although we need to slightly modify the inductive steps, the sign-coherence of $\mathbf{C}((B')^{\top})$ is guaranteed by ($a$) and the sign-coherence of $G$-matrices for all $B_t \in \mathbf{B}(B)$. However, the following two parts require major modifications. 
\\
Part 1. Prove that \Cref{conj: discreteness conjecture} holds for $B^{\top}$. Suppose that a $c$-vector $\tilde{\mathbf{c}}_{i;t}^{t_0}$ is expressed as $\tilde{\mathbf{c}}_{i;t}^{t_0}=\alpha \mathbf{e}_j$. Note that $D$ is not necessarily a skew-symmetrizer of $B^{\top}$ but $D^{-1}$ is certainly a skew-symmetrizer because $D^{-1}B^{\top}=D^{-1}(DB)^{\top}D^{-1}$ is skew-symmetric. Thus, our desired equality is $\alpha=\sqrt{d_i^{-1}d_j}$.
By $(a)$, we have $G^{t}_{t_0}=(\tilde{C}^{t_0}_{t})^{\top}$. Since $\tilde{\mathbf{c}}_{i;t}^{t_0}$ is the $i$th column vector of $\tilde{C}^{t_0}_{t}$, the vector $(\tilde{\mathbf{c}}_{i;t}^{t_0})^{\top}=\alpha \mathbf{e}_j^{\top}$ appears as the $i$th row vector of $G^{t}_{t_0}$. Since we assume \Cref{conj: discreteness conjecture} for $B_t$, we can apply \Cref{lem: for the parallel lemma}. Thus, we have $\alpha=\pm\sqrt{d_i^{-1}d_j}$ as we desired.
\\
Part 2. We can do the same argument until \cite[(3.10)]{NZ12}. Now, we need to show the following claim.
\begin{quote}
 Suppose that \cite[Claim~2]{NZ12}; that is, suppose that the only non-zero entry in the $l$th column of $C_{t}^{t_0}$ is the $(k,l)$th entry for some $t_0,t \in \mathbb{T}_n$ and $k,l \in \{1,\dots,n\}$. Then, we have \cite[(3.10)]{NZ12}; that is,
  \begin{equation}
  B_{t_0}^{k \bullet}C_{t}^{t_0}=C_{t}^{t_0}B_t^{l \bullet}.
 \end{equation}
\end{quote}
This can be shown as follows. Recall from \Cref{sec: basic notations} that, for arbitrary matrices $X,Y \in \mathrm{M}_n(\mathbb{R})$ and $k=1,\dots,n$, we have
\begin{equation}\label{eq: change of bullet}
X^{\bullet k}Y=XE_{kk}Y=XY^{k \bullet},
\quad
X^{k \bullet}Y=E_{kk}XY=(XY)^{k \bullet}.
\end{equation}
Hence, we have $C_{t}^{t_0}B_t^{l \bullet}=(C^{t_0}_t)^{\bullet l}B_t$. By the assumption, all the entries of $(C^{t_0}_t)^{\bullet l}$ are $0$ except for the $(k,l)$th entry. By \Cref{lem: for the parallel lemma}, this situation also happens for $(G^{t_0}_{t})^{k \bullet}$. Moreover, since we assume \Cref{conj: discreteness conjecture} for $B_{t_0}$, it holds that
\begin{equation}\label{eq: nontivial point in the real setting}
(C^{t_0}_t)^{\bullet l}=(G^{t_0}_{t})^{k \bullet}.
\end{equation}
Namely, we have
\begin{equation}
  C_{t}^{t_0}B_t^{l \bullet}=(C^{t_0}_t)^{\bullet l}B_t=(G^{t_0}_t)^{k \bullet}B_t\overset{\eqref{eq: change of bullet}}{=}(G^{t_0}_tB_t)^{k \bullet} \overset{\eqref{eq: first duality}}{=}(B_{t_0}C_t^{t_0})^{k \bullet} \overset{\eqref{eq: change of bullet}}{=}B_{t_0}^{k \bullet}C_t^{t_0},
\end{equation}
which is our desired equality. By using this equality, we can prove $(c)$ similarly.
\end{proof}
\begin{remark}
In this paper, the equality \eqref{eq: nontivial point in the real setting} is the unique point that the same arguments in the integer case cannot be applicable to the real case except for the sign-coherence.
\end{remark}
By considering (\ref{eq: thied duality}), we give some equivalent conditions to Conjecture~\ref{conj: standard hypothesis}.
\begin{corollary}[cf.~{\cite[Prop.~4.2]{NZ12}, \cite[Prop.~8.19]{Rea14}}]\label{cor: equivalency of sign-coherent class}
Let $B \in \mathrm{M}_{n}(\mathbb{R})$ be a skew-symmetrizable matrix. Suppose that, for any $B' \in {\bf B}(B)$, Conjecture~\ref{conj: discreteness conjecture} holds.
The following conditions are equivalent:
\begin{itemize}
\item[\textup{($a$)}] For any $B' \in {\bf B}(B)$, its $C$-pattern ${\bf C}(B')$ and $G$-pattern ${\bf G}(B')$ are sign-coherent. (Conjecture~\ref{conj: standard hypothesis})
\item[\textup{($b$)}] For any $B' \in {\bf B}(B) \cup {\bf B}(B^{\top})$, its $C$-pattern ${\bf C}(B')$ is sign-coherent.
\item[\textup{($c$)}] For any $B' \in {\bf B}(B) \cup {\bf B}(-B)$, its $C$-pattern ${\bf C}(B')$ is sign-coherent.
\item[\textup{($d$)}] For any $B' \in {\bf B}(B) \cup {\bf B}(B^{\top})$, its $G$-pattern ${\bf G}(B')$ is sign-coherent.
\item[\textup{($e$)}] For any $B' \in {\bf B}(B) \cup {\bf B}(-B)$, its $G$-pattern ${\bf G}(B')$ is sign-coherent.
\end{itemize}
\end{corollary}
\begin{proof}
When we suppose ($a$), the sign-coherence for $B' \in \mathbf{B}(B^{\top})$ is shown by \eqref{eq: thied duality}. Thus, $(a) \Rightarrow (b),(d)$ holds. Moreover, by $-B=D^{-1}B^{\top}D$ and \Cref{lem: positive conjugation}, $(b) \Leftrightarrow (c)$ and $(d) \Leftrightarrow (e)$ hold. The remaining implications $(b) \Rightarrow (a)$ and $(d) \Rightarrow (a)$ are shown by the same way in \cite{NZ12}.
\end{proof}

\subsection{Structure of $G$-fans}
In Section~\ref{sec: geometric properties under the sign-coherency}, we introduce a geometric structure called $G$-cone $\mathcal{C}(G_t)$. (See Definition~\ref{def: G-cone}.) In ordinary cluster theory, it is known that the set of all $G$-cones has the fan structure. We can also investigate this structure under Conjecture~\ref{conj: standard and discreteness conjecture} for the real case.
\begin{definition}\label{def: fan}
A nonempty set $\Delta$ of simplicial cones is called a (simplicial) {\em fan} if it satisfies the following conditions:
\begin{itemize}
\item For any cone $\mathcal{C} \in \Delta$, all faces of $\mathcal{C}$ also belong to $\Delta$.
\item For any pair of cones $\mathcal{C},\mathcal{C}' \in \Delta$, their intersection $\mathcal{C} \cap \mathcal{C}'$ is a face of both $\mathcal{C}$ and $\mathcal{C}'$.
\end{itemize}
\end{definition}
In the cluster algebra theory, the following is one of the most important object.
\begin{definition}[$G$-fan]\label{def: G-fan}
Let $B \in {\bf SC}$. We define the set of simplicial cones
\begin{equation}
\Delta_{{\bf G}}(B)=\{\mathcal{C}_{J}(G^{t_0}_{t}) \mid t \in \mathbb{T}_{n}, J \subset \{1,2,\dots,n\}\},
\end{equation}
and we call it a {\em $G$-fan}.
\end{definition}
The following fact was shown in \cite{GHKK18} in the integer case. Although their proof cannnot be directly applicable, we obtain the following proposition under \Cref{conj: standard and discreteness conjecture} with the same arguments of \cite[Thm.~8.7]{Rea14} and \cite[Thm.~II. 2.17]{Nak23}.
\begin{proposition}[cf.~\cite{Rea14,GHKK18,Nak23}]\label{prop: fan}
Let $B \in {\bf SC}$. Suppose that Conjecture~\ref{conj: standard and discreteness conjecture} holds for this $B$. Then, the $G$-fan $\Delta_{\bf G}(B)$ is really a fan in the sense of Definition~\ref{def: fan}.
\end{proposition}
Since the proof in \cite{Nak23} relies on \eqref{eq: dual mutation formula}, we can do the same argument. Note that \cite[Cond.~II.2.28]{Nak23} is verified as stated in \Cref{prop: lemma for cones}.
\par
In the ordinary cluster algebras, this is a geometric realization of a cluster complex. As in \Cref{thm: ordinary synchronicity}, the periodicity appearing in the $G$-fan is the same as the one of $C$-, $G$-patterns. However, by generalizing the real entries, we can observe the following different phenomenon.
\begin{example}\label{ex: bad phenomenon for periodicity}
Consider the initial exchange matrix
\begin{equation}
B=\left(\begin{matrix}
0 & -\frac{1}{2}\\
2 & 0
\end{matrix}\right).
\end{equation}
Note that it is expressed as
\begin{equation}
B=\left(\begin{matrix}
\frac{1}{2} & 0\\
0 & 1
\end{matrix}\right)
\left(\begin{matrix}
0 & -1\\
1 & 0
\end{matrix}\right)
\left(\begin{matrix}
2 & 0\\
0 & 1
\end{matrix}\right).
\end{equation}
Thus, this matrix is of quasi-integer type. In particular, Conjecture~\ref{conj: standard and discreteness conjecture} holds. By calculating, we obtain the $G$-matrices as in Figure~\ref{fig: G-pattern with quasi-integer type}. So, we draw the $G$-fan as in Figure~\ref{fig: bad example}. In Figure~\ref{fig: bad example}, the blue lines imply the mutation of $G$-matrices. On the other hand, the red lines imply how we can obtain the $G$-cones by the mutation. As this example indicates, the periodicity of $G$-cones (namely, $\mathcal{C}(G_{t'})=\mathcal{C}(G_{t})$) does not imply the periodicity of $G$-matrices. For example, if a $G$-cone $\mathcal{C}(G_{t})$ is the positive orthant $\mathfrak{O}_{+}^{2}$, there are two possibilities $G_t=\left(\begin{smallmatrix}
1 & 0\\
0 & 1
\end{smallmatrix}\right)$
or
$G_t=\left(\begin{smallmatrix}
0 & \frac{1}{2}\\
2 & 0
\end{smallmatrix}\right)$. This means that \Cref{thm: ordinary synchronicity} does not hold by generalizing to the real entries. In Section~\ref{sec: syncro}, we will consider this problem.
\begin{figure}[hbtp]
\begin{tikzpicture}
\draw (-3,0)--(3,0);
\draw (0,-2.5)--(0,3);
\draw (0,0)--(1.25,-2.5);
\draw[red] (1,1) -- (-1,1) -- (-1,-1) -- (3/8,-3/2)--(1.5,-0.75) -- (1,1);
\fill[red] (1,1) circle [radius=0.06];
\fill[red] (-1,1) circle [radius=0.06];
\fill[red] (-1,-1) circle [radius=0.06];
\fill[red] (3/8,-3/2) circle [radius=0.06];
\fill[red] (1.5,-0.75) circle [radius=0.06];

\draw (2,2) node {$\left(\begin{smallmatrix}
1 & 0\\
0 & 1
\end{smallmatrix}\right)$};
\draw (-2,2) node {$\left(\begin{smallmatrix}
-1 & 0\\
0 & 1
\end{smallmatrix}\right)$};
\draw (-2,-2) node {$\left(\begin{smallmatrix}
-1 & 0\\
0 & -1
\end{smallmatrix}\right)$};
\draw (3/4,-3) node {$\left(\begin{smallmatrix}
1 & 0\\
-2 & -1
\end{smallmatrix}\right)$};
\draw (3,-1.5) node {$\left(\begin{smallmatrix}
1 & \frac{1}{2}\\
-2 & 0
\end{smallmatrix}\right)$};
\draw (4,3) node {$\left(\begin{smallmatrix}
0 & \frac{1}{2}\\
2 & 0
\end{smallmatrix}\right)$};
\draw (-3,3) node {$\left(\begin{smallmatrix}
0 & -\frac{1}{2}\\
2 & 0
\end{smallmatrix}\right)$};
\draw (-3,-3) node {$\left(\begin{smallmatrix}
0 & -\frac{1}{2}\\
-2 & 0
\end{smallmatrix}\right)$};
\draw (9/8,-4) node {$\left(\begin{smallmatrix}
0 & \frac{1}{2}\\
-2 & -1
\end{smallmatrix}\right)$};
\draw (5.5,-9/4) node {$\left(\begin{smallmatrix}
1 & \frac{1}{2}\\
0 & -1
\end{smallmatrix}\right)$};
\draw[blue] (1.2,2)--(-1.1,2);
\draw[blue] (-2,1.5)--(-2,-1.5);
\draw[blue] (-1.6,-2.3)--(0,-3);
\draw[blue] (1.6,-2.8)--(2.8,-1.9);
\draw[blue] (3.1,-1)--(4,2.5);
\draw[blue] (3,3)--(-2.1,3);
\draw[blue] (-3.2,2.5)--(-3.2,-2.5);
\draw[blue] (-2.3,-3.5)--(0.2,-4);
\draw[blue] (2,-3.8)--(4.8,-2.8);
\draw[blue] (5,-1.8)--(2.2,1.7);
\end{tikzpicture}
\caption{$G$-pattern and $G$-fan associated with $B=\left(\begin{smallmatrix}
0 & -\frac{1}{2}\\
2 & 0
\end{smallmatrix}\right)$}\label{fig: bad example}
\end{figure}
\end{example}

\section{Classification of sign-coherent class of rank $2$}\label{sec: classification of rank 2}
In this section, we give a classification of sign-coherent class and $G$-fans of rank $2$. In the ordinary cluster theory, a formula for $c$-, $g$-vectors is obtained by \cite{Rea14,GN22} explicitly. We can also obtain such formula for the real cases if we focus on the sign-coherent class.
\subsection{Rank 2 sign-coherent class}
For the rank $2$ case, the classification is given as follows.
\begin{theorem}\label{thm: rank 2 classification}
Let the initial exchange matrix be $B=\left(\begin{smallmatrix}
0 & -a\\
b & 0
\end{smallmatrix}\right)$ with $a,b \in \mathbb{R}_{\geq 0}$. Then, $B$ belongs to ${\bf SC}$ if and only if either of the following holds.
\begin{itemize}
\item $\sqrt{ab}=2\cos{\frac{\pi}{m}}$ holds for some $m \in \mathbb{Z}_{\geq 2}$.
\item $\sqrt{ab} \geq 2$.
\end{itemize}
\end{theorem}
\begin{remark}
In \cite{DP24, DP25}, real $C$-, $G$-matrices in the case of $\sqrt{ab}=2\cos\frac{\pi}{m}$ have already been constructed by using unfolding from the other finite type. Here, we simply calculate $C$-, $G$-matrices based on the recursion, and we may show that this is a maximal setting to generalize $C$-, $G$-matrices.
\end{remark}
When $m=2$, then $B$ is a zero matrix and it is easy to check that $B\in {\bf SC}$. In addition, thanks to \Cref{prop: skew-symmetrizing method}, it suffices to consider the skew-symmetric case. Thus, we set
\begin{equation}
B=\left(\begin{matrix}
0 & -p\\
p & 0
\end{matrix}\right),
\end{equation}
where $p \in \mathbb{R}_{> 0}$.
\par
Firstly, we will give some examples of the sign-coherent class. For the rank 2 integer case, an explicit formula for $g$-vectors has already known in \cite[Lem.~3.2]{LS15} and \cite[Prop.~9.6]{Rea14}, and also $c$-vectors in \cite[Prop.~3.1]{GN22}. We refer to the expression of \cite{GN22} based on {\em Chebyshev polynomials} of the second kind $U_{n}(p)$ ($n \geq -2$), which is defined as follows:
\begin{equation}
U_{-2}(p)=-1,\ U_{-1}(p)=0,\ U_{n+2}(p)=2pU_{n+1}(p)-U_{n}(p).
\end{equation}
Note that $U_{0}(p)=1$ and $U_{1}(p)=2p$. Set $u_n(p)=U_{n}(\frac{p}{2})$. Then, based on the property of the Chebyshev polynomials, we obtain the following properties for $u_n(p)$.
\begin{lemma}[e.g.~{\cite[(1.4), (1.33b)]{HM03}}]\label{lem: lemma for the rank 2 classification}
We have $u_{n+2}(p)=pu_{n+1}(p)-u_{n}(p)$ for any $n \geq -2$. Moreover, for any $\theta \in \mathbb{R}$, it holds that
\begin{equation}\label{eq: sin expression of un}
\sin{\theta} \cdot u_n(2\cos{\theta})=\sin{(n+1)\theta}, \quad 
\sinh{\theta} \cdot u_n(2\cosh{\theta})=\sinh{(n+1)\theta}.
\end{equation}
\end{lemma}
The calculation of the forthcoming examples depends on the following lemma, but for the proof of Theorem~\ref{thm: rank 2 classification}, we show more general setting. The following expression was essentially obtained by \cite[Prop. 3.1]{GN22} for $p \geq 2$.
\begin{lemma}\label{lem: calculation of rank 2 c-vectors}
Fix an initial vertex $t_0 \in \mathbb{T}_{n}$.
\\
\textup{($a$)} We set the vertices $t_i^{2} \in \mathbb{T}_{n}$ \textup{($i=0,1,2,\dots$)} as follows:
\begin{figure}[H]
\centering
\begin{tikzpicture}
\draw (0,0) node [above] {$t_0=t_0^2$} -- (6,0);
\fill (0,0) circle [radius=0.06];
\foreach \x in {1,2,3}
    {
    \fill (2*\x,0) node [above] {$t_{\x}^{2}$} circle [radius=0.06];
    };
\draw (1,0) node [above] {$2$};
\draw (3,0) node [above] {$1$};
\draw (5,0) node [above] {$2$};
\draw[dotted] (6,0)--(7,0); 
\end{tikzpicture}
\end{figure}
\noindent
Let $k \in \mathbb{Z}_{\geq 1}$. Suppose that all $C_{t_0^2}$, $C_{t_1^2}$, \dots, $C_{t_{k-1}^2}$ are sign-coherent, and their tropical signs $(\varepsilon_{1;t_{i}^2},\varepsilon_{2;t_{i}^2})$ are given by $(\varepsilon_{1;t_0^2},\varepsilon_{2;t_0^2})=(+,+)$, and for any $i=1,\dots,k-1$,
\begin{equation}
(\varepsilon_{1;t_{i}^2},\varepsilon_{2;t_{i}^2})=\begin{cases}
(+,-) & \textup{if $i$ is odd},\\
(-,+) & \textup{if $i$ is even}.
\end{cases}
\end{equation}
(Note that we do not assume the sign-coherence of $C_{t_{k}^{2}}$.) Then, for any $i=0,1,2,\dots,k$, we have
\begin{equation}
C_{t_{i}^2}=\begin{cases}
\left(\begin{matrix}
-u_{i-2}(p) & u_{i-1}(p)\\
-u_{i-1}(p) & u_{i}(p)
\end{matrix}\right)
& \textup{if $i$ is even},\\
\left(\begin{matrix}
u_{i-1}(p) & -u_{i-2}(p)\\
u_{i}(p) & -u_{i-1}(p)
\end{matrix}\right)
& \textup{if $i$ is odd}.
\end{cases}
\end{equation}
\textup{($b$)} We set the vertices $t_{i}^{1} \in \mathbb{T}_{n}$ \textup{($i=0,1,2,\dots$)} as follows:
\begin{figure}[H]
\centering
\begin{tikzpicture}
\draw (0,0) node [above] {$t_0=t_0^1$} -- (6,0);
\fill (0,0) circle [radius=0.06];
\foreach \x in {1,2,3}
    {
    \fill (2*\x,0) node [above] {$t_{\x}^{1}$} circle [radius=0.06];
    };
\draw (1,0) node [above] {$1$};
\draw (3,0) node [above] {$2$};
\draw (5,0) node [above] {$1$};
\draw[dotted] (6,0)--(7,0); 
\end{tikzpicture}
\end{figure}
\noindent
Let $k \in \mathbb{Z}_{\geq 3}$. Suppose that all $C_{t_0^1},C_{t_1^1},\dots,C_{t_{k-1}^1}$ are sign-coherent, and their tropical signs $(\varepsilon_{1;t_{i}^{1}},\varepsilon_{2;t_{i}^{1}})$ are given by $(\varepsilon_{1;t_{0}^{1}},\varepsilon_{2;t_{0}^{1}})=(+,+)$, $(\varepsilon_{1;t_{1}^{1}},\varepsilon_{2;t_{1}^{1}})=(-,+)$,
$(\varepsilon_{1;t_{2}^{1}},\varepsilon_{2;t_{2}^{1}})=(-,-)$, and for any $i=3,4,\dots,k-1$,
\begin{equation}
(\varepsilon_{1;t_{i}^{1}},\varepsilon_{2;t_{i}^{1}})=\begin{cases}
(+,-) & \textup{if $i$ is odd},\\
(-,+) & \textup{if $i$ is even}.
\end{cases}
\end{equation}
Then, for any $i=2,3,\dots,k$, we have
\begin{equation}
C_{t_{i}^{1}}=\begin{cases}
\left(\begin{matrix}
-u_{i-2}(p) & u_{i-3}(p)\\
-u_{i-3}(p) & u_{i-4}(p)
\end{matrix}\right)
& \textup{if $i$ is even},\\
\left(\begin{matrix}
u_{i-3}(p) & -u_{i-2}(p)\\
u_{i-4}(p) & -u_{i-3}(p)
\end{matrix}\right)
& \textup{if $i$ is odd}.
\end{cases}
\end{equation}
\end{lemma}
\begin{proof}
We can show the claim by the induction on $k$. For example, if $k$ is even, then $t_{k-1}^{2}$ and $t_{k}^{2}$ are $1$-adjacent. Thus, we have
\begin{equation}
\begin{aligned}
C_{t_k^{2}}&=C_{t_{k-1}^{2}}(J_1+[\varepsilon_{1;t_{k-1}^{2}}B_{t_{k-1}^{2}}]^{1 \bullet}_{+})=
\left(\begin{matrix}
u_{k-2}(p) & -u_{k-3}(p)\\
u_{k-1}(p) & -u_{k-2}(p)
\end{matrix}\right)
\left(\begin{matrix}
-1 & p\\
0 & 1
\end{matrix}\right)
\\
&=\left(\begin{matrix}
-u_{k-2}(p) & pu_{k-2}(p)-u_{k-3}(p)\\
-u_{k-1}(p) & pu_{k-1}(p)-u_{k-2}(p)
\end{matrix}\right)
=\left(\begin{matrix}
-u_{k-2}(p) & u_{k-1}(p)\\
-u_{k-1}(p) & u_{k}(p)
\end{matrix}\right).
\end{aligned}
\end{equation} The proof of the case that $k$ is odd or for $t_{k}^{1}$ is similar.
\end{proof}
\begin{example}\label{ex: rank 2 C matrices}
Based on Lemma~\ref{lem: calculation of rank 2 c-vectors}, we obtain the expression of all $C$-matrices explicitly. Now, we provide three classes of examples as follows.
\\
({\bf Type $I_2(m)$}) Let $p=2\cos{\theta}$ with $\theta=\frac{\pi}{m}$ ($m \in \mathbb{Z}_{\geq 2}$). In fact, this is of finite type. Since
\begin{equation}
C_{t_1^{2}}=\left(\begin{matrix}
1 & 0\\
p & -1
\end{matrix}\right),
\end{equation}
the assumption of Lemma~\ref{lem: calculation of rank 2 c-vectors} is satisfied for $C_{t_1^{2}}$. Thus, by Lemma~\ref{lem: lemma for the rank 2 classification} and (\ref{eq: sin expression of un}), we have
\begin{equation}\label{eq: Ct2 expression}
C_{t_2^{2}}=\left(\begin{matrix}
-u_{0}(p) & u_{1}(p)\\
-u_{1}(p) & u_{2}(p)
\end{matrix}\right)
=
\frac{1}{\sin \theta}\left(\begin{matrix}
-\sin{\theta} & \sin{2\theta} \\
-\sin{2\theta} & \sin{3\theta}
\end{matrix}\right).
\end{equation}
This is sign-coherent. If $m \geq 3$, its tropical signs are given by $(\varepsilon_{1;{t_{2}^{2}}},\varepsilon_{2;t_{2}^{2}})=(-,+)$. Thus, by using Lemma~\ref{lem: lemma for the rank 2 classification} and (\ref{eq: sin expression of un}) again, we obtain $C_{t_3^{2}}$ like (\ref{eq: Ct2 expression}). By repeating this argument, we show the following claim:

For any $i=1,2,3,\dots,m-1$, the assumption of Lemma~\ref{lem: lemma for the rank 2 classification} holds. Moreover, for any $i=0,1,2,\dots,m$, we have
\begin{equation}\label{eq: Ctk expression}
\begin{aligned}
C_{t_{i}^{2}}=\begin{cases}
\displaystyle{\frac{1}{\sin{\theta}}}\left(\begin{matrix}
-\sin{(i-1)\theta} & \sin{i\theta}\\
-\sin{i\theta} & \sin{(i+1)\theta}
\end{matrix}\right)
& \textup{if $i$ is even},\\
\displaystyle{\frac{1}{\sin\theta}}\left(\begin{matrix}
\sin{i\theta} & -\sin{(i-1)\theta}\\
\sin{(i+1)\theta} & -\sin{i\theta}
\end{matrix}\right)
& \textup{if $i$ is odd}.
\end{cases}
\end{aligned}
\end{equation}
Note that $\theta=\frac{\pi}{m}$. Thus, by $\sin{\pi}=0$ and $\sin{\frac{\pi}{m}}=\sin{\frac{m-1}{m}\pi}=-\sin{\frac{m+1}{m}\pi}$, we have
\begin{equation}
C_{t_{m}^2}=\begin{cases}
\left(\begin{matrix}
-1 & 0\\
0 & -1
\end{matrix}\right)
&
\textup{if $m$ is even},
\\
\left(\begin{matrix}
0 & -1\\
-1 & 0
\end{matrix}\right)
&
\textup{if $m$ is odd}.
\end{cases}
\end{equation}
By a direct calculation, we have
\begin{equation}
C_{t_{m+1}^{2}}=
\begin{cases}
\left(\begin{matrix}
-1 & 0\\
0 & 1
\end{matrix}\right)
&
\textup{if $m$ is even},
\\
\left(\begin{matrix}
0 & -1\\
1 & 0
\end{matrix}\right)
&
\textup{if $m$ is odd},
\end{cases}
\quad
C_{t_{m+2}^{2}}=
\begin{cases}
\left(\begin{matrix}
1 & 0\\
0 & 1
\end{matrix}\right)
&
\textup{if $m$ is even},
\\
\left(\begin{matrix}
0 & 1\\
1 & 0
\end{matrix}\right)
&
\textup{if $m$ is odd}.
\end{cases}
\end{equation}
Thus, a periodicity $C_{t^2_{m+2}}=\tilde{\sigma}C_{t_0}$ appears, where $\sigma=\mathrm{id} \in \mathfrak{S}_2$ if $m$ is even and $\sigma=(1,2) \in \mathfrak{S}_{2}$ if $m$ is odd. Moreover, $B$ also has the same periodicity $B_{t^2_{m+2}}=\sigma B_{t_0}$. Thus, by Proposition~\ref{prop: periodicity}, every $t_{k}^{2}$ ($k\geq m+2$) satisfies $C_{t^2_{k}}=\tilde{\sigma}C_{t^2_{k-m-2}}$. By setting $t_{-k}^{2}=t_{k}^{1}$, the similar relation also holds. Hence, every $C$-matrix is obtained.
\\
({\bf Type $A^{(1)}_{1}$}) Let $p=2$. This is a well-known integer case of affine type. By Lemma~\ref{lem: lemma for the rank 2 classification} and $u_{i}(2)=i+1$, this $C$-pattern is obtained as follows:
\begin{equation}
\begin{aligned}
C_{t_{i}^{2}}&=
\begin{cases}
\left(\begin{matrix}
-i+1 & i\\
-i & i+1
\end{matrix}\right)
&
\textup{if $i \geq 0$ is even},
\\
\left(\begin{matrix}
i & -i+1\\
i+1 & -i
\end{matrix}\right)
&
\textup{if $i \geq 0$ is odd}.
\end{cases}
\\
C_{t_{i+2}^{1}}&=\begin{cases}
\left(\begin{matrix}
-i-1 & i\\
-i & i-1
\end{matrix}\right)
&
\textup{if $i \geq 0$ is even},
\\
\left(\begin{matrix}
i & -i-1\\
i-1 & -i
\end{matrix}\right)
&
\textup{if $i \geq 0$ is odd},
\end{cases}
\end{aligned}
\end{equation}
and $C_{t_1^{1}}=\left(\begin{smallmatrix}
-1 & 0\\
0 & 1
\end{smallmatrix}\right)$.
\\
({\bf Non-affine type}) Let $p > 2$. Then, we can express $p=2\cosh \theta$ for some $\theta > 0$. Since $\sinh(k\theta)>0$ for all $k=1,2,\dots$, we can do the same argument as in (\ref{eq: Ctk expression}) infinitely many times. Thus, we have
\begin{equation}
\begin{aligned}
C_{t_{i}^{2}}&=
\begin{cases}
\displaystyle{\frac{1}{\sinh{\theta}}}\left(\begin{matrix}
-\sinh{(i-1)\theta} & \sinh{i\theta}\\
-\sinh{i\theta} & \sinh{(i+1)\theta}
\end{matrix}\right)
&
\textup{if $i \geq 0$ is even},
\\
\displaystyle{\frac{1}{\sinh{\theta}}}\left(\begin{matrix}
\sinh{i\theta} & -\sinh{(i-1)\theta}\\
\sinh{(i+1)\theta} & -\sinh{i\theta}
\end{matrix}\right)
&
\textup{if $i \geq 0$ is odd}.
\end{cases}
\\
C_{t_{i+2}^{1}}&=\begin{cases}
\displaystyle{\frac{1}{\sinh{\theta}}}\left(\begin{matrix}
-\sinh{(i+1)\theta} & \sinh{i\theta}\\
-\sinh{i\theta} & \sinh{(i-1)\theta}
\end{matrix}\right)
&
\textup{if $i \geq 0$ is even},
\\
\displaystyle{\frac{1}{\sinh \theta}}\left(\begin{matrix}
\sinh{i\theta} & -\sinh{(i+1)\theta}\\
\sinh{(i-1)\theta} & -\sinh{i\theta}
\end{matrix}\right)
&
\textup{if $i \geq 0$ is odd},
\end{cases}
\end{aligned}
\end{equation}
and $C_{t_1^{1}}=\left(\begin{smallmatrix}
-1 & 0\\
0 & 1
\end{smallmatrix}\right)$.
\end{example}
\begin{remark}
In the ordinary cluster theory, there is another affine type $A_{2}^{(2)}$ for the skew-symmetrizable case. The corresponding initial exchange matrix is
\begin{equation}
B=\left(\begin{matrix}
0 & -1\\
4 & 0
\end{matrix}\right).
\end{equation}
However, for $C$-, $G$-patterns, this is similar to the type $A_{1}^{(1)}$. In fact, we can take a skew-symmetrizer by $D=\left(\begin{smallmatrix}
4 & 0\\
0 & 1
\end{smallmatrix}\right)$. Then, we have $\mathrm{Sk}(B) = \left(\begin{smallmatrix}
0 & -2\\
2 & 0
\end{smallmatrix}\right)$. Based on this correspondence, we can recover $C$-, $G$-patterns of this $B$ by \Cref{prop: skew-symmetrizing method}.
\end{remark}
Now, we are ready to prove \Cref{thm: rank 2 classification} as follows.
\begin{proof}[Proof of Theorem~\ref{thm: rank 2 classification}]
The ``if" part is shown by Example~\ref{ex: rank 2 C matrices}. (Note that we can do the same argument for $B^{\top}$. Thus, by $(b) \Rightarrow (a)$ in \Cref{cor: equivalency of sign-coherent class}, the sign-coherence for $G$-patterns also holds.) Now, we aim to show the ``only if" part. 
Let $p$ satisfy $0<p<2$ and $p \neq 2\cos\frac{\pi}{m}$ for any $m \in \mathbb{Z}_{\geq 2}$. Set $p=2\cos \theta$ for some $0<\theta<\frac{\pi}{2}$. Then, there exists $m=2,3,\dots$ such that $\frac{\pi}{m+1}<\theta<\frac{\pi}{m}$.
By doing the same argument as in Example~\ref{ex: rank 2 C matrices} of Type $I_{2}(m)$, we obtain (\ref{eq: Ctk expression}). (Note that $\sin{i\theta} > 0$ holds for any $i=1,\dots,m$ because $i\theta \leq m\theta <\pi$.) Consider $C_{t^2_m}$. Then, the $c$-vector $(\sin{m\theta},\sin{(m+1)\theta})^{\top}/ \sin{\theta}$ or $(\sin{(m+1)\theta},\sin{m\theta})^{\top}/ \sin{\theta}$ appears. However, this is not sign-coherent because $\sin{m\theta}>0>\sin{(m+1)\theta}$ by $m\theta<\pi<(m+1)\theta$. Thus, this $C$-pattern is not sign-coherent.
\end{proof}
Before, we have introduced Conjecture~\ref{conj: standard hypothesis} and Conjecture~\ref{conj: discreteness conjecture}. Now, we have already known the explicit formulas for rank 2 case. We prove that all of them satisfy these conjectures.
\begin{theorem}\label{thm: conjectures are true for rank 2}
Conjecture~\ref{conj: standard hypothesis} and Conjecture~\ref{conj: discreteness conjecture} are true for the sign-coherent class of rank 2. In particular, Conjecture~\ref{conj: standard and discreteness conjecture} also holds.
\end{theorem}
Hence, based on this theorem, \Cref{prop: dual mutation} holds for the sign-coherent class of rank $2$.
\subsection{Rank 2 $G$-fans}
Thanks to \Cref{prop: fan} and Theorem~\ref{thm: conjectures are true for rank 2}, the $G$-fan $\Delta_{\bf G}(B)$ is really a fan for rank 2 and sign-coherent case. We see the examples of these fans. For the integer case, it has already been calculated in \cite[Ex.~9.5, Prop.~9.6]{Rea14}.
\par
Note that, by Proposition~\ref{prop: fundamental properties under sign-coherency}, the relation $C_t^{\top}G_t=I_{2}$ holds for any skew-symmetric $B$. In particular, all $G$-matrices are calculated by $G_{t}=(C_{t}^{-1})^{\top}$.
\begin{example}
In this example, let the initial exchange matrix be $B=\left(\begin{smallmatrix}
0 & -p\\
p & 0
\end{smallmatrix}\right)$ with $p\in \mbR_{>0}$.
\\
({\bf Type $I_{2}(m)$}) Let $p=2\cos\frac{\pi}{m}$ ($m \in \mathbb{Z}_{\geq 3}$). Let $\sigma=\mathrm{id} \in \mathfrak{S}_2$ if $m$ is even and $\sigma=(1,2) \in \mathfrak{S}_2$ if $m$ is odd. Then, every $G$-matrix is obtained by $G_{t_{i}^{2}}=\tilde{\sigma}G_{t_{i-m-2}^{2}}$ and, for any $i=0,1,\dots,m$,
\begin{equation}
G_{t_{i}^{2}}=\begin{cases}
\displaystyle{\frac{1}{\sin{\theta}}}\left(\begin{matrix}
\sin{(i+1)\theta} & \sin{i\theta}\\
-\sin{i\theta} & -\sin{(i-1)\theta}
\end{matrix}\right)
& \textup{if $i$ is even},\\
\displaystyle{\frac{1}{\sin\theta}}\left(\begin{matrix}
\sin{i\theta} &\sin{(i+1)\theta} \\
-\sin{(i-1)\theta}& -\sin{i\theta}
\end{matrix}\right)
& \textup{if $i$ is odd},
\end{cases}
\end{equation}
and
\begin{equation}
G_{t_{m+1}^{2}}=
\begin{cases}
\left(\begin{matrix}
-1 & 0\\
0 & 1
\end{matrix}\right)
&
\textup{if $m$ is even},
\\
\left(\begin{matrix}
0 & -1\\
1 & 0
\end{matrix}\right)
&
\textup{if $m$ is odd},
\end{cases}
\quad
G_{t_{m+2}^{2}}=
\begin{cases}
\left(\begin{matrix}
1 & 0\\
0 & 1
\end{matrix}\right)
&
\textup{if $m$ is even},
\\
\left(\begin{matrix}
0 & 1\\
1 & 0
\end{matrix}\right)
&
\textup{if $m$ is odd}.
\end{cases}
\end{equation}
Thus, the $G$-fan is composed by $m+2$ chambers, see  Figure~\ref{fig: type I27} and Figure~\ref{fig: type I28} for example.
\begin{figure}[hbtp]
\centering
\begin{minipage}{0.45\linewidth}
\centering
\begin{tikzpicture}
\draw (-2,0)--(2,0);
\draw (0,-2)--(0,2);
\foreach \x in {1,2,3}
    {
    \draw (0,0)--(2,{(-2)*(sin(\x*pi/7 r))/(sin((\x+1)*pi/7 r)))});
    };
\foreach \x in {1,2}
    {
    \draw (0,0)--({(2)*(sin(\x*pi/7 r))/(sin((\x+1)*pi/7 r)))},-2);
    };
\end{tikzpicture}
\caption{Type $I_2(7)$}\label{fig: type I27}
\end{minipage}
\begin{minipage}{0.45\linewidth}
\centering
\begin{tikzpicture}
\draw (-2,0)--(2,0);
\draw (0,-2)--(0,2);
\foreach \x in {1,2,3}
    {
    \draw (0,0)--(2,{(-2)*(sin(\x*pi/8 r))/(sin((\x+1)*pi/8 r)))});
    };
\foreach \x in {1,2,3}
    {
    \draw (0,0)--({(2)*(sin(\x*pi/8 r))/(sin((\x+1)*pi/8 r)))},-2);
    };
\end{tikzpicture}
\caption{Type $I_2(8)$}\label{fig: type I28}
\end{minipage}
\end{figure}
\\
({\bf Type} $A_{1}^{(1)}$) Let $p=2$. Then, we have
\begin{equation}
\begin{aligned}
G_{t_{i}^{2}}&=
\begin{cases}
\left(\begin{matrix}
i+1 & i\\
-i & -i+1
\end{matrix}\right)
&
\textup{if $i \geq 0$ is even},
\\
\left(\begin{matrix}
i & i+1\\
-i+1 & -i
\end{matrix}\right)
&
\textup{if $i \geq 0$ is odd}.
\end{cases}
\\
G_{t_{i+2}^{1}}&=\begin{cases}
\left(\begin{matrix}
i-1 & i\\
-i & -i-1
\end{matrix}\right)
&
\textup{if $i \geq 0$ is even},
\\
\left(\begin{matrix}
i & i-1\\
-i-1 & -i
\end{matrix}\right)
&
\textup{if $i \geq 0$ is odd},
\end{cases}
\\
G_{t_1^{1}}&=\left(\begin{matrix}
-1 & 0\\
0 & 1
\end{matrix}\right).
\end{aligned}
\end{equation}
It is known that the $G$-fan covers $\mathbb{R}^2 \setminus \mathcal{C}^{\circ}((1,-1))$ \cite{Rea14}, see  Figure~\ref{fig: G fan of type A11}.
\\
({\bf Non-affine type}) Let $p > 2$. Then, we have
\begin{equation}
\begin{aligned}
G_{t_{i}^{2}}&=
\begin{cases}
\displaystyle{\frac{1}{\sinh{\theta}}}\left(\begin{matrix}
\sinh{(i+1)\theta} & \sinh{i\theta}\\
-\sinh{i\theta} & -\sinh{(i-1)\theta}
\end{matrix}\right)
&
\textup{if $i \geq 0$ is even},
\\
\displaystyle{\frac{1}{\sinh{\theta}}}\left(\begin{matrix}
\sinh{i\theta} & \sinh{(i+1)\theta}\\
-\sinh{(i-1)\theta} & -\sinh{i\theta}
\end{matrix}\right)
&
\textup{if $i \geq 0$ is odd}.
\end{cases}
\\
G_{t_{i+2}^{1}}&=\begin{cases}
\displaystyle{\frac{1}{\sinh{\theta}}}\left(\begin{matrix}
\sinh{(i-1)\theta} & \sinh{i\theta}\\
-\sinh{i\theta} & \sinh{(i+1)\theta}
\end{matrix}\right)
&
\textup{if $i \geq 0$ is even},
\\
\displaystyle{\frac{1}{\sinh \theta}}\left(\begin{matrix}
\sinh{i\theta} & \sinh{(i-1)\theta}\\
-\sinh{(i+1)\theta} & -\sinh{i\theta}
\end{matrix}\right)
&
\textup{if $i \geq 0$ is odd},
\end{cases}
\\
G_{t_1^{1}}&=\left(\begin{matrix}
-1 & 0\\
0 & 1
\end{matrix}\right).
\end{aligned}
\end{equation}
This $G$-fan is illustrated as Figure~\ref{fig: non-affine G-fan}.
By \cite[Prop.~9.6]{Rea14}, it is known that each $G$-fan covers $\{\mathbb{R}^2 \setminus \mathcal{C}({\bf v}_1,{\bf v}_2)\}\cup \{{\bf 0}\}$, where
\begin{equation}
{\bf v}_1=(p-\sqrt{p^2-4},-2)^{\top}, \quad {\bf v}_2=(p+\sqrt{p^2-4},-2)^{\top}.
\end{equation}

\begin{figure}[htbp]
\centering
\begin{minipage}{0.45\linewidth}
\centering
\begin{tikzpicture}
\draw (-2,0)--(2,0);
\draw (0,-2)--(0,2);
\foreach \x in {1,2,...,20}
    {
    \draw (0,0)--(2,{(-2)*(\x/(\x+1))});
    };
\foreach \x in {1,2,...,20}
    {
    \draw (0,0)--({(2)*(\x/(\x+1))},-2);
    };
\draw[dashed] (0,0)--(2,-2);
\end{tikzpicture}
\caption{Type $A_{1}^{(1)}$}\label{fig: G fan of type A11}
\end{minipage}
\begin{minipage}{0.45\linewidth}
\centering
\begin{tikzpicture}
\draw (-2,0)--(2,0);
\draw (0,-2)--(0,2);
\draw (0,0)--(2,-1.232);
\draw (0,0)--(2,-1.348);
\draw (0,0)--(2,-1.402);
\draw (0,0)--(2,-1.43);
\draw (0,0)--(2,-1.444);
\draw (0,0)--(2,-1.451);
\draw (0,0)--(2,-1.455);
\draw (0,0)--(2,-1.457);
\draw (0,0)--(2,-1.458);
\draw (0,0)--(2,-1.459);
\draw (0,0)--(2,-1.459);
\draw (0,0)--(2,-1.459);
\draw (0,0)--(2,-1.46);
\draw (0,0)--(1.232,-2);
\draw (0,0)--(1.348,-2);
\draw (0,0)--(1.402,-2);
\draw (0,0)--(1.43,-2);
\draw (0,0)--(1.444,-2);
\draw (0,0)--(1.451,-2);
\draw (0,0)--(1.455,-2);
\draw (0,0)--(1.457,-2);
\draw (0,0)--(1.458,-2);
\draw (0,0)--(1.459,-2);
\draw (0,0)--(1.459,-2);
\draw (0,0)--(1.459,-2);
\draw (0,0)--(1.46,-2);
\fill[gray, opacity=0.5] (0,0)--(1.46,-2)--(2,-2)--(2,-1.46)--cycle;
\end{tikzpicture}
\caption{Non-affine type}\label{fig: non-affine G-fan}
\end{minipage}
\end{figure}
\end{example}
\section{Classification of sign-coherent class of finite type}\label{sec: finite type classification}
In this section, we aim to classify the sign-coherent class of finite type via Coxeter diagrams. Firstly, we focus on the following class.
\begin{definition}
We say that a $C$-pattern ${\bf C}(B)$ is {\em finite} if the set $\{C_{t} \mid t \in \mathbb{T}_{n}\}$ of all its $C$-matrices is finite.
\end{definition}
Our purpose is to show the following main theorem.
\begin{theorem}\label{thm: finite type classifcation}
Let $B \in \mathrm{M}_{n}(\mathbb{R})$ be skew-symmetric. Suppose that $Q(B)$ is connected. For each $m \in \mathbb{Z}_{\geq 2}$, let $[m]=2\cos\frac{\pi}{m}$. Then, $B$ satisfies both of
\begin{itemize}
\item for any $B' \in {\bf B}(B)$, $B'$ satisfies the sign-coherent property.
\item for any $B' \in {\bf B}(B)$, its $C$-pattern ${\bf C}(B')$ is finite.
\end{itemize}
if and only if the corresponding quiver $Q(B)$ is mutation-equivalent to any of the quiver in Figure~\ref{fig: Coxeter diagrams}. In these diagrams, we omit $[3]=1$.
\begin{figure}[htbp]
\centering
\begin{tikzpicture}[scale=0.8]
\draw (-7.5,0) node {$A_n$\textup{:}};
\foreach \x in {0,1,2}
    {
    \draw[->] (\x-7,0)--(\x-6,0);
    \fill (\x-7,0) circle [radius=0.06];
    };
\draw[dotted] (-4,0)--(-3,0);
\draw[->] (-3,0)--(-2,0);
\fill (-2,0) circle [radius=0.06];

\draw (0.2,0) node {$B_n=C_n$\textup{:}};
\foreach \x in {0,1}
    {
    \draw[->] (\x+1.4,0)--(\x+2.4,0);
    \fill (\x+1.4,0) circle [radius=0.06];
    };
\draw[dotted] (3.4,0)--(4.4,0);
\draw[->] (4.4,0)--(5.4,0);
\fill (5.4,0) circle [radius=0.06];
\draw[->] (5.4,0)--(5.9,0) node [above] {$[4]$}--(6.4,0);
\fill (6.4,0) circle [radius=0.06];
\end{tikzpicture}
\\
\begin{minipage}{0.4 \linewidth}
\centering
\begin{tikzpicture}
\draw (-1.5,0) node {$D_n$\textup{:}};
\draw[->] ({-1/sqrt(2)},{-1/sqrt(2)})--(-0.02,-0.02);
\fill ({-1/sqrt(2)},{-1/sqrt(2)}) circle [radius=0.06];
\draw[->] ({-1/sqrt(2)},{1/sqrt(2)})--(-0.02,0.02);
\fill ({-1/sqrt(2)},{1/sqrt(2)}) circle [radius=0.06];
\foreach \x in {0,1}
    {
    \draw[->] (\x,0)--(\x+1,0);
    \fill (\x,0) circle [radius=0.06];
    };
\draw[dotted] (2,0)--(3,0);
\fill (3,0) circle [radius=0.06];
\draw[->] (3,0)--(4,0);
\fill (4,0) circle [radius=0.06];
\end{tikzpicture}
\end{minipage}
\\
\begin{minipage}{0.4 \linewidth}
\centering
$E_{6}$\textup{:}
\begin{tikzpicture}[scale = 0.8]
\foreach \x in {0,1,...,3}
    {
    \draw[->] (\x,0)--(\x+1,0);
    \fill (\x,0) circle [radius=0.06];
    };
\fill (4,0) circle [radius=0.06];
\draw[->] (2,1)--(2,0.04);
\fill (2,1) circle [radius=0.06];
\end{tikzpicture}
\end{minipage}
\begin{minipage}{0.4 \linewidth}
\centering
$E_{7}$\textup{:}
\begin{tikzpicture}[scale = 0.8]
\foreach \x in {0,1,...,4}
    {
    \draw[->] (\x,0)--(\x+1,0);
    \fill (\x,0) circle [radius=0.06];
    };
\fill (5,0) circle [radius=0.06];
\draw[->] (2,1)--(2,0.04);
\fill (2,1) circle [radius=0.06];
\end{tikzpicture}
\end{minipage}
\\
\begin{minipage}{0.4 \linewidth}
\centering
$E_{8}$\textup{:}
\begin{tikzpicture}[scale = 0.8]
\foreach \x in {0,1,...,5}
    {
    \draw[->] (\x,0)--(\x+1,0);
    \fill (\x,0) circle [radius=0.06];
    };
\fill (6,0) circle [radius=0.06];
\draw[->] (2,1)--(2,0.04);
\fill (2,1) circle [radius=0.06];
\end{tikzpicture}
\end{minipage}
\begin{minipage}{0.3\linewidth}
\centering
$F_{4}$\textup{:}
\begin{tikzpicture}[scale=0.8]
\draw[->] (0,0)--(1,0);
\draw[->] (1,0)--(1.5,0) node [above] {$[4]$}-- (2,0);
\draw[->] (2,0)--(3,0);
\foreach \x in {0,1,2,3}
    {
    \fill (\x,0) circle [radius=0.06];
    };
\end{tikzpicture}
\end{minipage}
\begin{minipage}{0.3\linewidth}
\centering
$H_{3}$\textup{:}
\begin{tikzpicture}[scale=0.8]
\draw[->] (0,0)--(1,0);
\draw[->] (1,0)--(1.5,0) node [above] {$[5]$}-- (2,0);
\foreach \x in {0,1,2}
    {
    \fill (\x,0) circle [radius=0.06];
    };
\end{tikzpicture}
\end{minipage}
\begin{minipage}{0.3\linewidth}
\centering
$H_{4}$\textup{:}
\begin{tikzpicture}[scale=0.8]
\draw[->] (0,0)--(1,0);
\draw[->] (1,0)--(2,0);
\draw[->] (2,0)--(2.5,0) node [above] {$[5]$}-- (3,0);
\foreach \x in {0,1,2,3}
    {
    \fill (\x,0) circle [radius=0.06];
    };
\end{tikzpicture}
\end{minipage}
\begin{minipage}{0.3\linewidth}
\centering
$I_{2}(m)$\textup{:}
\begin{tikzpicture}[scale=0.8]
\draw[->] (0,0)--(0.5,0) node [above] {$[m]$}--(1,0);
\foreach \x in {0,1}
    {
    \fill (\x,0) circle [radius=0.06];
    };
\end{tikzpicture}
\end{minipage}
\caption{Coxeter quivers}\label{fig: Coxeter diagrams}
\end{figure}
\end{theorem}
\begin{remark}
For the reader's convenience, we give some values of $[m]=2\cos{\frac{\pi}{m}}$ as follows.
\begin{equation}
\begin{array}{c|ccccc}
m & 2 & 3 & 4 & 5 & 6 \\ \hline
[m] & 0 & 1 & \sqrt{2} & \frac{1+\sqrt{5}}{2} & \sqrt{3} 
\end{array}
\end{equation}
In Figure~\ref{fig: Coxeter diagrams}, there are some coincidences such as $A_2=I_2(3)$ and $B_2=C_2=I_2(4)$. In the ordinary cluster algebras, there is type $G_2$, which can be covered by $G_2=I_2(6)$.
\end{remark}
\begin{remark}
Let $X=A_{n},\dots,I_{2}(m)$. Then, the Coxeter diagram of type $X$ is defined (e.g., \cite[Fig.~1]{Hum90}). The quiver in Figure~\ref{fig: Coxeter diagrams} of type $X$ is obtained by changing the order $m$ of each edge of the Coxeter diagram to $[m]=2\cos{\frac{\pi}{m}}$, and giving the orientation as in the figure. In this procedure, we might consider another orientation, but it does not give an essential problem. As in \cite[Thm.~8.6]{FZ03a}, for any quiver $Q'$ whose underlying graph $\Gamma(Q')$ is obtained from the Coxeter diagram of type $X$ by replacing the order $m$ of each edge to $[m]$, $Q'$ is mutation-equivalent to $Q$. Hence, Theorem~\ref{thm: finite type classifcation} means that the finite $C$-pattern can be classified by the Coxeter diagrams. 
\end{remark}
The proof depends on the following two properties. The claim ($b$) is suggested by Salvatore Stella in the personal communication.
\begin{proposition}
Let $B \in {\bf SC}$. Suppose that its $C$-pattern is sign-coherent and finite. Then, the following statements hold.
\\
\textup{($a$)} Its $B$-pattern ${\bf B}(B)$ is also finite.
\\
\textup{($b$)} Each component of $B$ should be expressed as $0$ or $[m]=2\cos{\frac{\pi}{m}}$ for some $m \in \mathbb{Z}_{\geq 3}$.
\end{proposition}
\begin{proof}
By (\ref{eq: from C to B}), the claim ($a$) holds. Now, we aim to show ($b$). If there exists an entry $b_{ij}$ ($i,j=1,2,\dots,n$) such that $b_{ij} \neq 0,[m]$ ($m \in \mathbb{Z}_{\geq 3}$), then it implies that $i\neq j$. Consider sub $C$-matrices induced by $\{i,j\}$. Then, by \Cref{lem: submatrix}, these submatrices are the same as in Example~\ref{ex: rank 2 C matrices}. If $|b_{ij}|<2$, then these sub $C$-matrices are not sign-coherent. Thus, the original $C$-matrices are not sign-coherent. If $|b_{ij}|\geq 2$, then there are infinitely many sub $C$-matrices. Thus, the original $C$-pattern also has infinitely many $C$-matrices, which is a contradiction. Hence, this completes the proof.
\end{proof}
Thanks to this property, the classification of finite $C$-patterns can be reduced to the classification of finite $B$-patterns. Such $B$ is said to be {\em mutation-finite}, and its classification has already been completed by \cite{FT23}.
\begin{proposition}[{\cite[Thm.~A]{FT23}}]
\textup{($a$)} Let $Q$ be an $\mathbb{R}$-valued quiver that has at least one non-integer weight. Suppose that $Q$ satisfies the following conditions.
\begin{itemize}
\item the number of vertices is larger or equal to $3$.
\item $Q$ does not arise from a triangulated orbifold in the sense of \cite{FST12}.
\item $Q$ is mutation-finite.
\end{itemize}
Then, $Q$ is mutation-equivalent to a quiver in the list of \cite[Table~1.1]{FT23}.
\\
\textup{($b$)} In \cite[Table~1.1]{FT23}, consider the quivers satisfying the following condition:
\begin{quote}
Every weight of edges has the form of $[m]=2\cos\frac{\pi}{m}$ for some $m=3,4,\dots$.
\end{quote}
Then, such quivers are only of type $F_4$, $H_3$, $H_4$, and $\tilde{F}_{4}$ (Figure~\ref{fig: type F4tilde}).
\begin{figure}[hbtp]
\centering
\begin{minipage}{0.4\linewidth}
$\tilde{F}_4$\textup{:}
\begin{tikzpicture}
\draw[->] (0,0)->(1,0);
\draw[->] (1,0)--(1.5,0) node [above] {$[4]$}->(2,0);
\draw[->] (2,0)->(3,0);
\draw[->] (3,0)->(4,0);
\foreach \x in {0,1.05,2.05,3.05,4.05}
    \fill (\x,0) circle (2pt);
\end{tikzpicture}
\end{minipage}
\caption{Type $\tilde{F}_4$}\label{fig: type F4tilde}
\end{figure}
\end{proposition}
\begin{remark}
The orbifold is a connected and bordered oriented $2$-dimensional surface with a finite set of marked points and orbifold points with no intersection. Then, the compatible arcs can be defined according to certain conditions. A triangulation of the orbifold is a maximal collection of distinct pairwise compatible arcs and it corresponds to a quiver $Q$. For more details, we will not mention here, but we can refer to \cite{FST12}.
\end{remark}
Note that quivers arising from a triangulated orbifold are of quasi-integer type. Thus, its classification of finite type is given by the ordinary cluster theory. Moreover, we easily check that $F_4$ and $\tilde{F}_4$ are of quasi-integer type. In the ordinary cluster theory, we have already known that $F_4$ is of finite type and $\tilde{F}_4$ is not of finite type (of affine type). Hence, the remaining problem is that to show the following lemma, cf.~{\cite[Thm.~1.4]{DP24}}.
\begin{lemma}\label{lem: type H}
For any quiver $Q$ mutation-equivalent to  one of type $H_3$ or $H_4$, the corresponding $C$-pattern is sign-coherent and finite.
\end{lemma}
In \cite{DP24}, they showed that $C$-pattern ${\bf C}(B)$ is sign-coherent and finite if the underlying graph $\Gamma(Q(B))$ is the same as the one in Figure~\ref{fig: Coxeter diagrams}. However, this is not enough to show our claim. We focus on all the quivers mutation-equivalent to any of the quiver oriented to the diagram in Figure~\ref{fig: Coxeter diagrams} and will show this claim by using the computer program. (See Appendix~\ref{sec: proof of Lemma type H}.)
\par
By calculating explicitly, the conjectures are true for every quiver mutation-equivalent to any of \Cref{fig: Coxeter diagrams}.
\begin{theorem}\label{thm: conjecture for finite type}
Conjecture~\ref{conj: standard and discreteness conjecture} is true for any quiver mutation-equivalent to a Coxeter quiver in Figure~\ref{fig: Coxeter diagrams}.
\end{theorem}
Hence, \Cref{prop: dual mutation} and \Cref{prop: fan} holds for the sign-coherent class of finite type.
\section{Modified $C$-, $G$-matrices and their synchronicity}\label{sec: modified and synchro}
As in Example~\ref{ex: bad phenomenon for periodicity}, by generalizing to the real entries, the periodicity appearing in the $G$-fan is different from the one in the $G$-pattern. By Theorem~\ref{thm: ordinary synchronicity}, this phenomenon does not appear in the integer case. In Section~\ref{sec: modified matrices}, we introduce another two matrix patterns called {\em modified $C$-, $G$-patterns}, which is more closely related to the $G$-fan structure. In Section~\ref{sec: syncro}, we obtain their synchronicity properties which are analogue to Theorem~\ref{thm: ordinary synchronicity}.
\subsection{Modified $C$-, $G$-matrices}\label{sec: modified matrices}
We introduce the following two different matrix patterns.
\begin{definition}[Modified $C$-, $G$-matrices]
Let $B \in \mathrm{M}_{n}(\mathbb{R})$ be a skew-symmetrizable matrix with a skew-symmetrizer $D=\mathrm{diag}(d_1,\dots,d_n)$. Let ${\bf C}(B)=\{C_t\}$ and ${\bf G}(B)=\{G_t\}$ be the $C$-pattern and the $G$-pattern. We define the {\em modified $C$-pattern} $\tilde{\bf C}(B;D^{-\frac{1}{2}})=\{\tilde{C}_t\}$ and the {\em modified $G$-pattern} $\tilde{\bf G}(B;D^{-\frac{1}{2}})=\{\tilde{G}_{t}\}$ {\em with a modification factor $D^{-\frac{1}{2}}=\mathrm{diag}\left(\sqrt{d_1}^{-1},\dots,\sqrt{d_{n}}^{-1}\right)$} by
\begin{equation}
\tilde{C}_{t}=C_{t}D^{-\frac{1}{2}},\quad
\tilde{G}_{t}=G_{t}D^{-\frac{1}{2}}.
\end{equation}
We often fix one modification factor $D^{-\frac{1}{2}}$, and the difference of the modification factor does not affect our argument. (The difference can be ignored by taking the inner product $\langle\, ,\, \rangle_{D}$ defined by (\ref{eq: definition of inner product}).) In this case, we omit $D^{-\frac{1}{2}}$ and simply write $\tilde{\bf C}(B;D^{-\frac{1}{2}})=\tilde{\bf C}(B)$ and $\tilde{\bf G}(B;D^{-\frac{1}{2}})=\tilde{\bf G}(B)$.
These matrices $\tilde{C}_{t}$ and $\tilde{G}_{t}$ are called a {\em modified $C$-matrix} and a {\em modified $G$-matrix}. We call each column vector $\tilde{\bf c}_{i;t}$ and $\tilde{\bf g}_{i;t}$ a {\em modified $c$-vector} and a {\em modified $g$-vector}, respectively, and they are given by
\begin{equation}\label{eq: modified c-, g-vectors}
\tilde{\bf c}_{i;t}=\frac{1}{\sqrt{d_i}}{\bf c}_{i;t},
\quad
\tilde{\bf g}_{i;t}=\frac{1}{\sqrt{d_i}}{\bf g}_{i;t}.
\end{equation}
\end{definition}
Then, we give the self-contained recursion for these modified patterns as follows.
\begin{proposition}
Let $B \in \mathrm{M}_{n}(\mathbb{R})$ be a skew-symmetrizable matrix with a skew-symmetrizer $D$. Consider its modified $C$-pattern ${\bf C}(B)=\{\tilde{C}_{t}\}_{t \in \mathbb{T}_{n}}$ and $G$-pattern ${\bf G}(B)=\{\tilde{G}_{t}\}_{t \in \mathbb{T}_{n}}$ with the initial exchange matrix $B_{t_0}=B$. Then, they can be obtained by the following recursion:
\begin{itemize}
\item $\tilde{C}_{t_0}=\tilde{G}_{t_0}=D^{-\frac{1}{2}}$.
\item For any $k$-adjacent vertices $t,t' \in \mathbb{T}_{n}$, it holds that
\begin{equation}\label{eq: recursions for modified C-, G-matrices}
\begin{aligned}
\tilde{C}_{t'}&=\tilde{C}_{t}J_k+\tilde{C}_{t}[\varepsilon \mathrm{Sk}(B_t)]^{k\bullet}_{+}+[-\varepsilon \tilde{C}_{t}]_{+}^{\bullet k}\mathrm{Sk}(B_{t}),\\
\tilde{G}_{t'}&=\tilde{G}_{t}J_{k}+\tilde{G}_{t}[-\varepsilon\mathrm{Sk}(B_{t})]^{\bullet k}_{+}-B_{t_0}[-\varepsilon \tilde{C}_{t}]^{\bullet k}_{+},
\end{aligned}
\end{equation}
where $\varepsilon = \pm1$ is chosen arbitrary. Here, $\mathrm{Sk}(B_t)=D^{\frac{1}{2}}B_tD^{-\frac{1}{2}}$ is given in Definition~\ref{def: Sk map}.
\end{itemize}
\end{proposition}
In the second equality of (\ref{eq: recursions for modified C-, G-matrices}), note that $B_{t_0}$ appearing in the last term is not $\mathrm{Sk}(B_{t_0})$. However, if we assume $B_{t_0} \in {\bf SC}$, it is not an essential problem.
\begin{proof}
By multiplying $D^{-\frac{1}{2}}$ from right to (\ref{eq: epsilon expression for C-, G-mutations}), we obtain this recursion. For example, the second recursion is obtained by
\begin{equation}
\begin{aligned}
&\ (G_{t}J_{k}+G_{t}[-\varepsilon B_{t}]^{\bullet k}_{+}-B_{t_0}[-\varepsilon C_{t}]^{\bullet k}_{+})D^{-\frac{1}{2}}\\
=&\ (G_{t}D^{-\frac{1}{2}})(D^{\frac{1}{2}}J_kD^{-\frac{1}{2}})+(G_{t}D^{-\frac{1}{2}})[-\varepsilon D^{\frac{1}{2}}B_tD^{-\frac{1}{2}}]^{\bullet k}_{+}-B_{t_0}[-\varepsilon C_{t}D^{-\frac{1}{2}}]\\
=&\ \tilde{G}_{t}J_k+\tilde{G}_{t}[-\varepsilon \mathrm{Sk}(B_t)]^{\bullet k}_{+}-B_{t_0}[-\varepsilon \tilde{C}_{t}]^{\bullet k}_{+}.
\end{aligned}
\end{equation}
\end{proof}
This recursion is essentially controlled by $\mathrm{Sk}(B_t)$. So, we write ${\bf B}(\mathrm{Sk}(B))=\{\tilde{B}_t\}$. (In Section~\ref{sec: Skew-symmetrizing method}, we write it by $\hat{B}_t$, but here we write $\tilde{B}_t$ to align the notation.)
\par
If we assume $B \in {\bf SC}$, we obtain the following recursion.
\begin{proposition}
Let $B \in {\bf SC}$ with a skew-symmetrizer $D$. Then, the recursion (\ref{eq: recursions for modified C-, G-matrices}) for modified $C$- and $G$-matrices is expressed as
\begin{equation}\label{eq: reduced recursion for modified C, G matrices}
\begin{aligned}
\tilde{C}_{t'}=\tilde{C}_{t}(J_k+[\varepsilon_{k;t} \tilde{B}_{t}]^{k\bullet}_{+}),
\quad
\tilde{G}_{t'}=\tilde{G}_{t}(J_{k}+[-\varepsilon_{k;t}\tilde{B}_{t}]^{\bullet k}_{+}).
\end{aligned}
\end{equation}
Moreover, the following recursion for modified $c$-, $g$-vectors holds.
\begin{equation}
\tilde{\bf c}_{i;t'}=\begin{cases}
-\tilde{\bf c}_{i;t} & i = k,\\
\tilde{\bf c}_{i;t}+[\varepsilon_{k;t}\tilde{b}_{ki;t}]_{+}\tilde{\bf c}_{k;t} & i \neq k,
\end{cases}
\quad
\tilde{\bf g}_{i;t}=\begin{cases}
-\tilde{\bf g}_{k;t}+\sum_{j=1}^{n}[-\varepsilon_{k;t}\tilde{b}_{jk;t}]_{+}\tilde{\bf g}_{j;t} & i=k,\\
\tilde{\bf g}_{i;t} & i \neq k,
\end{cases}
\end{equation}
where we set $\tilde{B}_{t}=(\tilde{b}_{ij}) \in \mathrm{M}_{n}(\mathbb{R})$.
\end{proposition}
\begin{proof}
By substituting $\varepsilon=\varepsilon_{k;t}$ into (\ref{eq: recursions for modified C-, G-matrices}), we obtain the claim. (Note that $[-\varepsilon_{k;t}C_{t}]^{\bullet k}_{+}=O$ by definition.)
\end{proof}
\begin{remark}
Although the recursion (\ref{eq: reduced recursion for modified C, G matrices}) for the modified $C$-, $G$-matrices are changed like the skew-symmetric case, the dual mutation formula is given as
\begin{equation}\label{eq: dual mutation formula for modified C-, G-matrices}
\begin{aligned}
\tilde{C}^{t_1}_{t}&=(J_k+[-\tau^{t_0}_{k;t}B_{t_0}]_{+}^{k \bullet})\tilde{C}^{t_0}_{t},\\
\tilde{G}^{t_1}_{t}&=(J_k+[\tau^{t_0}_{k;t}B_{t_0}]^{\bullet k}_{+})\tilde{G}^{t_0}_{t},
\end{aligned}
\end{equation}
where $t_0$ and $t_1$ are $k$-adjacent vertices. We obtain it by multiplying $D^{-\frac{1}{2}}$ from right to both sides of (\ref{eq: dual mutation formula}).
\end{remark}
For later, we obtain some relations for modified $C$-, $G$-matrices.
\begin{proposition}\label{prop: fundamental properties of modified matrices}
Let $B \in {\bf SC}$ with a skew-symmetrizer $D=\mathrm{diag}(d_1,d_2,\dots,d_n)$. Consider its modified $C$-, $G$-patterns with the initial exchange matrix $B_{t_0}=B$.
\\
\textup{($a$)} For any $t \in \mathbb{T}_n$ and $i=1,\dots,n$, we have
\begin{equation}
\tilde{\bf c}_{i;t}=\frac{1}{\sqrt{d_i}}{\bf c}_{i;t},\quad \tilde{\bf g}_{i;t}=\frac{1}{\sqrt{d_i}}{\bf g}_{i;t}.
\end{equation}
\textup{($b$)} For any $t \in \mathbb{T}_n$, we have $\tilde{C}_t^{\top}D\tilde{G}_t=I_{n}$. In particular, we have
\begin{equation}\label{eq: orthogonality between modified c-, g-vectors}
\langle \tilde{\bf g}_{i;t},\tilde{\bf c}_{j;t} \rangle_{D}=\begin{cases}
1 & i=j,\\
0 & i \neq j.
\end{cases}
\end{equation}
\textup{($c$)} We have
\begin{equation}\label{eq: duality for modified C-, G-matrices}
\tilde{G}_t\tilde{B}_t=B_{t_0}\tilde{C}_t,
\quad
\tilde{C}_t^{\top}DB_{t_0}\tilde{C}_t=\tilde{B}_t,
\end{equation}
where $\tilde{B}_t=\mathrm{Sk}(B_t)=D^{\frac{1}{2}}B_tD^{-\frac{1}{2}}$.
\\
\textup{($d$)} For any $t \in \mathbb{T}_{n}$, we have
\begin{equation}
\mathcal{C}(\tilde{G}_t)=\mathcal{C}(G_t).
\end{equation}
\end{proposition}
\begin{proof}
The claim ($a$) has been shown by (\ref{eq: modified c-, g-vectors}). The equality $\tilde{C}_t^{\top}D\tilde{G}_t=I_{n}$ is obtained by (\ref{eq: second duality}), and by considering $(i,j)$th entry of this matrix, we obtain (\ref{eq: orthogonality between modified c-, g-vectors}). The claim ($c$) follows from (\ref{eq: first duality}) and (\ref{eq: second duality}). The claim ($d$) is obvious by the definition of a $G$-cone.
\end{proof}

\subsection{Synchronicity among the $G$-fan and matrix patterns}\label{sec: syncro}
Here, we give some relationship of periodicity.
\par
Firstly, we can obtain the following synchronicity without Conjecture~\ref{conj: standard and discreteness conjecture}.
\begin{lemma}\label{lem: equivalency of the periodicity among modified c- g-vectors}
Let $B \in {\bf SC}$. For any $t,t' \in \mathbb{T}_n$ and $\sigma \in \mathfrak{S}_{n}$, we have the following equivalence.
\begin{equation}
\tilde{C}_{t'}=\tilde{\sigma}\tilde{C}_{t} \Longleftrightarrow \tilde{G}_{t'}=\tilde{\sigma}\tilde{G}_{t}.
\end{equation}
Moreover, if the above condition holds, then we have
\begin{equation}
\mathrm{Sk}(B_{t'})=\sigma \mathrm{Sk}(B_{t}).
\end{equation}
\end{lemma}
\begin{proof}
We only show $\Rightarrow$ here since another side is similar. By Proposition~\ref{prop: fundamental properties of modified matrices}~(b), we have
\begin{equation}
\tilde{G}_t=D^{-1}(\tilde{C}_t^{\top})^{-1},
\quad
\tilde{G}_{t'}=D^{-1}(\tilde{C}_{t'}^{\top})^{-1}.
\end{equation}
By (\ref{eq: periodicity and matrix product}) and $P_{\sigma}^{\top}=P_{\sigma}^{-1}$, the equality $\tilde{C}_{t'}=\tilde{\sigma}\tilde{C}_{t}$ implies $(\tilde{C}_{t'}^{\top})^{-1}=(\tilde{C}_{t}^{\top})^{-1}P_{\sigma}$. Thus, we have
\begin{equation}
\tilde{G}_{t'}=D^{-1}(\tilde{C}_{t'}^{\top})^{-1}=D^{-1}(C_{t}^{\top})^{-1}P_{\sigma}=\tilde{G}_{t}P_{\sigma}=\tilde{\sigma}\tilde{G}_{t}.
\end{equation}
Moreover, by substituting $\tilde{C}_{t'}=\tilde{C}_{t}P_{\sigma}$ into (\ref{eq: duality for modified C-, G-matrices}), we have $\mathrm{Sk}(B_{t'})=P_{\sigma}^{\top}\tilde{C}_{t}^{\top}DB_{t_0}\tilde{C}_{t}P_{\sigma}=\sigma\mathrm{Sk}(B_{t})$.
\end{proof}
By assuming Conjecture~\ref{conj: standard and discreteness conjecture}, we improve this phenomenon as follows.
\begin{proposition}[Cone-Matrix Synchronicity]\label{prop: matrix-cone synchronicity}
Let $B \in {\bf SC}$ with a skew-symmetrizer $D=\mathrm{diag}(d_1,\dots,d_n)$. Suppose that Conjecture~\ref{conj: standard and discreteness conjecture} holds for this $B$. Then, for any $t,t' \in \mathbb{T}_{n}$, the following three conditions are equivalent.
\begin{itemize}
\item[\textup{($a$)}] It holds that $\mathcal{C}(G_t)=\mathcal{C}(G_{t'})$.
\item[\textup{($b$)}] There exists $\sigma \in \mathfrak{S}_{n}$ such that $\tilde{C}_{t'}=\tilde{\sigma}\tilde{C}_{t}$.
\item[\textup{($c$)}] There exists $\sigma \in \mathfrak{S}_{n}$ such that $\tilde{G}_{t'}=\tilde{\sigma}\tilde{G}_{t}$.
\end{itemize}
Moreover, if the above conditions hold, then we can take the same $\sigma \in \mathfrak{S}_n$ such that $\tilde{C}_{t'}=\tilde{\sigma}\tilde{C}_{t}$ and $\tilde{G}_{t'}=\tilde{\sigma}\tilde{G}_{t}$, and it induces
\begin{equation}
\mathrm{Sk}(B_{t'})=\sigma \mathrm{Sk}(B_t).
\end{equation}
\end{proposition}
\begin{proof}
The implication $(b)\Rightarrow (a)$ is shown by Lemma~\ref{lem: c-vector expression of G-cones}, and the implication $(c) \Rightarrow (a)$ is shown by definition of cones.
We now show $(a) \Rightarrow (b),(c)$. For the proof, we need to consider the dual mutation formula (\ref{eq: dual mutation formula}). So, we write the $C,G$-matrices by $C^{t_0}_{t}$, $G^{t_0}_t$ and show the claim by the induction on $d=d(t_0,t')$. For $d=0$, suppose that $\mathcal{C}(G^{t_0}_{t})=\mathcal{C}(G^{t_0}_{t_0})=\mathfrak{O}_{+}^{n}$. Then, since $\mathfrak{O}_{+}^{n}$ is a simplicial cone spanned by ${\bf e}_1,\dots,{\bf e}_n$, there exists $\sigma \in \mathfrak{S}_{n}$ and $\beta_i>0$ such that
\begin{equation}
{\bf g}^{t_0}_{i;t}=\beta_{i}{\bf e}_{\sigma(i)},
\end{equation}
Thus, for each $j$th row ($j=1,2,\dots,n$), $G^{t_0}_{t}$ satisfies the condition ($a$) in Lemma~\ref{lem: for the parallel lemma}. Thus, by Lemma~\ref{lem: for the parallel lemma}~($b$), we have $\beta_{i}=\sqrt{d_id_{\sigma(i)}^{-1}}$ and
\begin{equation}
\tilde{\bf c}^{t_0}_{i;t}=\tilde{\bf c}^{t_0}_{\sigma(i);t_0},
\quad
\tilde{\bf g}^{t_0}_{i;t}=\tilde{\bf g}^{t_0}_{\sigma(i);t_0}.
\end{equation}
Note that $\tilde{\bf c}^{t_0}_{\sigma(i);t_0}=\tilde{\bf g}^{t_0}_{\sigma(i);t_0}=\sqrt{d_{\sigma(i)}}^{-1}{\bf e}_{\sigma(i)}$.
Suppose that the claim holds for some $d$, and let $t' \in \mathbb{T}_{n}$ be the vertex satisfying $d(t_0,t')=d+1$. Suppose that $\mathcal{C}(G^{t_0}_{t})=\mathcal{C}(G^{t_0}_{t'})$. By doing a similar argument, we express that ${\bf g}^{t_0}_{i;t}=\beta_i{\bf g}^{t_0}_{\sigma(i);t'}$.
Take the $k$-adjacent vertex $t_1$ to $t_0$ such that $d(t_1,t')=d$. Then, by Proposition~\ref{prop: dual mutation}, we have
\begin{equation}
{\bf g}^{t_1}_{i;t}=(J_k+[\tau^{t_0}_{k;t}B_{t_0}]^{\bullet k}_{+}){\bf g}^{t_0}_{i;t},
\quad
{\bf g}^{t_1}_{\sigma(i);t'}=(J_k+[\tau^{t_0}_{k;t'}B_{t_0}]^{\bullet k}_{+}){\bf g}^{t_0}_{\sigma(i);t'}.
\end{equation}
Since $\mathcal{C}(G^{t_0}_{t})=\mathcal{C}(G^{t_0}_{t'})$, it belongs to the same orthant. In particular, $\tau^{t_0}_{k;t}=\tau^{t_0}_{k;t'}$ holds. Thus, two vectors ${\bf g}^{t_1}_{i;t}$ and ${\bf g}^{t_1}_{\sigma(i);t'}$ are obtained by applying the same linear transformation $J_k+[\tau^{t_0}_{k;t}B_{t_0}]^{\bullet k}_{+}$ to ${\bf g}^{t_0}_{i;t}$ and ${\bf g}^{t_0}_{\sigma(i);t'}$, respectively. Thus, the relation ${\bf g}^{t_0}_{i;t}=\beta_i{\bf g}^{t_0}_{\sigma(i);t'}$ induces ${\bf g}^{t_1}_{i;t}=\beta_i{\bf g}^{t_1}_{\sigma(i);t'}$. In patricular, $\mathcal{C}(G^{t_1}_{t})=\mathcal{C}(G^{t_1}_{t'})$ holds. Since $d(t_1,t')=d$, we can apply the assumption of induction, that is,
\begin{equation}
\tilde{\bf c}^{t_1}_{i;t}=\tilde{\bf c}^{t_1}_{\sigma(i);t'},
\quad
\tilde{\bf g}^{t_1}_{i;t}=\tilde{\bf g}^{t_1}_{\sigma(i);t'}.
\end{equation}
By applying the dual mutation formula (\ref{eq: dual mutation formula for modified C-, G-matrices}), we obtain $\tilde{\bf c}^{t_0}_{i;t}=\tilde{\bf c}^{t_0}_{\sigma(i);t'}$ and $
\tilde{\bf g}^{t_0}_{i;t}=\tilde{\bf g}^{t_0}_{\sigma(i);t'}$ as we desired. (Note that $\tau^{t_1}_{k;t}=\tau^{t_1}_{k;t'}$ holds by the same reason.)
\par
If ($b$) holds, then by (\ref{eq: duality for modified C-, G-matrices}), we obtain $\mathrm{Sk}(B_{t'})=\sigma \mathrm{Sk}(B_{t})$.
\end{proof}
As a corollary of this proposition, we can show that the phenomenon in Example~\ref{ex: bad phenomenon for periodicity} does not occur for the skew-symmetric case.
\begin{corollary}\label{cor: synchronicity for the skew-symmetric case}
Let $B \in {\bf SC}$ be {\em skew-symmetric} (not skew-symmetrizable). Suppose that Conjecture~\ref{conj: standard and discreteness conjecture} holds for this $B$. Then, for any $t,t' \in \mathbb{T}_n$, we have the following equivalence.
\begin{equation}
\mathcal{C}(G_{t'})=\mathcal{C}(G_t) \Longleftrightarrow [G_{t'}]=[G_{t}] \Longleftrightarrow [C_{t'}]=[C_{t}].
\end{equation}
\end{corollary}
\begin{proof}
We can take $D=\mathrm{diag}(1,1,\dots,1)$ because $B$ is skew-symmetric. Thus, the claim holds by \Cref{prop: matrix-cone synchronicity}.
\end{proof}
We can obtain a similar phenomenon for original $C$-, $G$-matrices. The same result has already appeared in \cite{Nak21} for the integer case. However, we need to reconstruct the proof for the real case.
\begin{proposition}[$C$-$G$ Synchronicity]\label{prop: CG synchronicity}
Let $B \in {\bf SC}$ with a skew-symmetrizer $D=\mathrm{diag}(d_1,\dots,d_n)$. Suppose that Conjecture~\ref{conj: standard and discreteness conjecture} holds for this $B$.
\\
\textup{($a$)} Let $\sigma \in \mathfrak{S}_n$. If either $C_{t'}=\tilde{\sigma}C_t$ or $G_{t'}=\tilde{\sigma}G_t$ holds for some $t,t' \in \mathbb{T}_{n}$, then for any $i=1,2,\dots,n$, we have $d_{i}=d_{\sigma(i)}$ and, equivalently, $DP_{\sigma}=P_{\sigma}D$ holds.
\\
\textup{($b$)} For any $t,t' \in \mathbb{T}_{n}$ and $\sigma \in \mathfrak{S}_{n}$, the following two conditions are equivalent.
\begin{itemize}
\item It holds that $C_{t'}=\tilde{\sigma}C_t$.
\item It holds that $G_{t'}=\tilde{\sigma}G_t$.
\end{itemize}
Moreover, if the above conditions hold, then we have
\begin{equation}\label{eq: synchro for B matrix}
B_{t'}=\sigma B_{t}
\end{equation}
and
\begin{equation}\label{eq: synchronicity for modified patterns}
\mathrm{Sk}(B_{t'})=\sigma \mathrm{Sk}(B_{t}),
\quad
\tilde{C}_{t'}= \tilde{\sigma} \tilde{C}_{t},
\quad
\tilde{G}_{t'}=\tilde{\sigma} \tilde{G}_{t}.
\end{equation}
\end{proposition}
\begin{proof}
($a$) Suppose that $C_{t'}=\tilde{\sigma}C_{t}$. Then, we have
\begin{equation}\label{eq: c-vector equality 1}
{\bf c}_{\sigma(i);t'}={\bf c}_{i;t}.
\end{equation}
On the other hand, by Lemma~\ref{lem: c-vector expression of G-cones}, the assumption $C_{t'}=\tilde{\sigma}C_t$ implies $\mathcal{C}(G_{t'})=\mathcal{C}(G_t)$. Thus, by \Cref{prop: matrix-cone synchronicity}, we have
\begin{equation}\label{eq: c-vector equality 2}
\frac{1}{\sqrt{d_{\sigma(i)}}}{\bf c}_{\sigma(i);t'}=\frac{1}{\sqrt{d_{i}}}{\bf c}_{i;t}.
\end{equation}
To satisfy both (\ref{eq: c-vector equality 1}) and (\ref{eq: c-vector equality 2}), we have $d_i=d_{\sigma(i)}$ for any $i$. This means that $D=\sigma D$ and, by $\sigma D = P_{\sigma}^{\top}DP_{\sigma}$ and $P_{\sigma}^{\top}=P_{\sigma}^{-1}$, it implies that $DP_{\sigma}=P_{\sigma}D$. We can do the same argument in the case of $G_{t'}=\tilde{\sigma}G_t$.
\\
($b$) Suppose that $C_{t'}=\tilde{\sigma}C_t$. By (\ref{eq: second duality}), we express
\begin{equation}
G_{t'}=D^{-1}(C_{t'}^{\top})^{-1}D.
\end{equation}
By $C_{t'}=\tilde{\sigma}C_t=C_tP_{\sigma}$, we have $(C_{t'}^{\top})^{-1}=(C_t^{\top})^{-1}P_{\sigma}$, where we use $P_{\sigma}^{\top}=P_{\sigma}^{-1}$. Thus, we have
\begin{equation}
G_{t'}=D^{-1}(C_{t}^{\top})^{-1}P_{\sigma}D\overset{(a)}{=}D^{-1}(C_{t}^{\top})^{-1}DP_{\sigma}\overset{(\ref{eq: second duality})}{=}G_tP_{\sigma}=\tilde{\sigma}G_t.
\end{equation}
The equality (\ref{eq: synchro for B matrix}) can be shown by using (\ref{eq: from C to B}) as follows:
\begin{equation}
\begin{aligned}
B_{t'}&\overset{(\ref{eq: from C to B})}{=}D^{-1}C_{t'}^{\top}DB_{t_0}C_{t'}=D^{-1}P_{\sigma}^{\top}C_{t}^{\top}DB_{t_0}C_tP_{\sigma}\overset{(a)}{=}P_{\sigma}^{\top}D^{-1}C_t^{\top}DB_{t_0}C_tP_{\sigma}
\\
&\overset{(\ref{eq: from C to B})}{=}P_{\sigma}^{\top}B_{t}P_{\sigma}=\sigma B_t.
\end{aligned}
\end{equation}
Moreover, by ($a$), we also have $D^{-\frac{1}{2}}P_{\sigma}=P_{\sigma}D^{-\frac{1}{2}}$. Thus, we can obtain (\ref{eq: synchronicity for modified patterns}) as follows:
\begin{equation}
\tilde{C}_{t'}=C_{t'}D^{-\frac{1}{2}}=C_tP_{\sigma}D^{-\frac{1}{2}}=C_tD^{-\frac{1}{2}}P_{\sigma}=\tilde{\sigma}\tilde{C}_t.
\end{equation}
\end{proof}

\section{Isomorphism of exchange graphs}\label{sec: exchange graphs}
In the ordinary cluster algebra, one of the main object is the exchange graph, which is a combinatorial structure established by unlabeled seeds (triple of cluster variables, coefficients and exchange matrices) in \cite{FZ02}. In \cite{FL23}, it was proved that for all non-integer quivers
of finite type there is a well-defined geometric notion of an exchange graph , which generalizes the classical integer case. Here, we generalize this classical structure to the following five patterns for the real exchange matrix $B$.
\begin{itemize}
\item exchange graph associated with $C$-pattern ${\bf EG}({\bf C}(B))$
\item exchange graph associated with $G$-pattern ${\bf EG}({\bf G}(B))$
\item exchange graph associated with a $G$-fan ${\bf EG}(\Delta_{\bf G}(B))$
\item exchange graph associated with modified $C$-pattern ${\bf EG}(\tilde{\bf C}(B))$
\item exchange graph associated with modified $G$-pattern ${\bf EG}(\tilde{\bf G}(B))$
\end{itemize}
\par
To define them, we introduce a {\em quotient graph}, which is defined as follows:
\begin{definition}
Let $G=(V,E)$ be a graph with a vertex set $V$ and an edge set $E \subset V\times V$. Let $\sim$ be an equivalence relation on $V$. Then, we define the {\em quotient graph} $\tilde{G}=G/{\sim}$ as follows:
\begin{itemize}
\item The vertex set of $\tilde{G}$ is the equivalence class of $V/{\sim}$.
\item Two vertices $[v_1],[v_2] \in \tilde{G}$ are connected in $G/{\sim}$ if and only if there exist vertices $v'_1 \in [v_1]$ and $v'_2 \in [v_2]$ such that $v'_1$ and $v'_2$ are connected in $G$.
\end{itemize}
\end{definition}

\begin{definition}\label{def: cluster of c-, g-vectors}
For any $C$-matrix $C_{t}$ and $G$-matrix $G_t$, we write
\begin{equation}
[C_{t}]=\{{\bf c}_{1;t},\dots,{\bf c}_{n;t}\},
\quad
[G_{t}]=\{{\bf g}_{1;t},\dots,{\bf g}_{n;t}\},
\end{equation}
and we call them an {\em unlabeled cluster of $c$-vectors} and an {\em unlabeled cluster of $g$-vectors}, respectively.
\end{definition}
Here, for simplicity, we omit ``unlabeled" and simply call them clusters of $c$-, $g$-vectors. We also define a {\em cluster of modified $c$-vectors} $[\tilde{C}_{t}]=\{\tilde{\bf c}_{1;t},\dots,\tilde{\bf c}_{n;t}\}$ and a {\em cluster of modified $g$-vectors} $[\tilde{G}_{t}]=\{\tilde{\bf g}_{1;t},\dots,\tilde{\bf g}_{n;t}\}$.
\begin{lemma}\label{lem: unlabeled cluster equality}
Let $B \in {\bf SC}$. Then, for any $C_t,C_{t'} \in {\bf C}(B)$, $[C_{t}]=[C_{t'}]$ holds if and only if there exists $\sigma \in \mathfrak{S}_{n}$ such that
\begin{equation}
C_{t'}=\tilde{\sigma}C_{t}.
\end{equation}
We obtain the same result by replacing $C$-matrices with $G$-matrices, modified $C$-matrices, and modified $G$-matrices.
\end{lemma}
\begin{proof}
By $B \in {\bf SC}$, $[C_{t}]$ is a basis of $\mathbb{R}^{n}$ (Proposition~\ref{prop: fundamental properties under sign-coherency}). In particular, $[C_{t}]$ is the set consisting of distinct $n$ elements. Thus, we obtain the claim.
\end{proof}

\begin{definition}\label{def: exchange graph}
For any real exchange matrix $B$, we define the {\em exchange graph associated with $C$-pattern} ${\bf EG}({\bf C}(B))$ as the quotient graph $\mathbb{T}_{n}/{\sim}$, where
\begin{equation}
t \sim t' \Longleftrightarrow [C_t]=[C_{t'}].
\end{equation}
We also define the {\em exchange graph associated with $G$-pattern ${\bf EG}({\bf G}(B))$, with modified $C$-pattern ${\bf EG}(\tilde{\bf C}(B))$, and with modified $G$-pattern ${\bf EG}(\tilde{\bf G}(B))$} by replacing $C$-matrices with their corresponding matrices. Similarly, we define the {\em exchange graph associated with a $G$-fan ${\bf EG}(\Delta_{\bf G}(B))=\mathbb{T}_{n}/{\sim}$} by
\begin{equation}
t \sim t' \Longleftrightarrow \mathcal{C}(G_{t})=\mathcal{C}(G_{t'}).
\end{equation}
\end{definition}
We often view the vertices of each exchange graph as the objects which we used to define the equivalence relation. For example, a vertex of ${\bf EG}({\bf C}(B))$ is a cluster of $c$-vectors $[C_{t}]$. Of course, it does not affect the graph structure. 
\par
We discuss the relationship among these exchange graphs.
We say that two exchange graphs are {\em canonically isomorphic}  if the equivalence relations on $\mathbb{T}_n$ to define each quotient graph are the same. 
In ordinary cluster theory, we establish the exchange graph by the cluster variables, see \cite[Def.~7.1]{FZ02}. It is also regarded as a quotient graph of $\mathbb{T}_n$ in the same manner. Moreover, by \Cref{thm: ordinary synchronicity}, all of them are the same if we consider the integer case. However, this is not true by generalizing to the real case.
\begin{example}
Consider the $G$-pattern and the $G$-fan in Example~\ref{ex: bad phenomenon for periodicity}. Then, in Figure~\ref{fig: bad example}, the blue graph is the exchange graph associated with the $G$-pattern, which is the $10$-cycle, and the red graph is the exchange graph associated with the $G$-fan, which is the $5$-cycle.
\par
As this example indicates, if the $G$-fan is really a fan, an edge of each exchange graph of the $G$-fan can be characterized by the following geometric condition.
\begin{quote}
Two vertices $\mathcal{C}(G_{t})$ and $\mathcal{C}(G_{t'})$ of ${\bf EG}(\Delta_{\bf G}(B))$ are connected if and only if the codimension of its intersection $\mathcal{C}(G_{t})\cap\mathcal{C}(G_{t'})$ is $1$.
\end{quote}
In this sense, the exchange graph of the $G$-fan is the same as the {\em dual graph} of this fan.
\end{example}
Under some conditions, these exchange graphs satisfy the following fundamental properties.
\begin{lemma}\label{lem: fundamental property of exchange graphs}
Let $B \in {\bf SC}$ of rank $n \geq 2$.
\\
\textup{($a$)} The exchange graphs associated with the modified $C$-pattern ${\bf EG}(\tilde{\bf C}(B))$ and the modified $G$-pattern ${\bf EG}(\tilde{\bf G}(B))$ are $n$-regular.
\\
\textup{($b$)} Suppose that Conjecture~\ref{conj: standard and discreteness conjecture} holds for this $B$. Then, the exchange graphs associated with the $C$-pattern ${\bf EG}({\bf C}(B))$, the $G$-pattern ${\bf EG}({\bf G}(B))$, and the $G$-fan ${\bf EG}(\Delta_{\bf G}(B))$ are $n$-regular.
\end{lemma}
\begin{proof}
Firstly, we will show the case for four matrix patterns. Since the following proof works well for each pattern, we show the claim for $C$-pattern. Let $\sim$ be the equivalence relation on $\mathbb{T}_{n}$ to define the exchange graph of $C$-pattern.
Let $t \in \mathbb{T}_{n}$ be any vertex. For each $i=1,2,\dots,n$, let $t_{i}$ be the $i$-adjacent vertex to $t$. Then, by definition, for each exchange graph, we can show that $[t] \neq [t_i]$ and $[t_i] \neq [t_j]$ if $i \neq j$. (For example, if $[C_{t}] = [C_{t_i}]$, it induces a nontrivial linear relation among $\{{\bf c}_{i;t}\mid i=1,\dots,n\}$ by considering the mutation of $c$-vectors (\ref{eq: mutation for c-,g-vectors}). However, it contradicts to Proposition~\ref{prop: fundamental properties under sign-coherency}~(a). Thus, this claim holds for $C$-pattern. We can do the same argument for other patterns.) Thus, we can find distinct $n$-vertices $[t_i]$ ($i=1,2,\dots,n$) connected to $[t]$. Suppose that $[t'] \in \mathbb{T}_n/{\sim}$ is connected to $[t]$. We show that $[t']$ is the same as $[t_i]$ for some $i$. Since $[t]$ and $[t']$ are connected, there exist $s \in [t]$ and $s' \in [t']$ such that $s$ and $s'$ are adjacent in $\mathbb{T}_{n}$. Then, by Lemma~\ref{lem: unlabeled cluster equality}, there exists $\sigma \in \mathfrak{S}_{n}$ such that $C_{s}=\tilde{\sigma}C_{t}$. By \Cref{prop: CG synchronicity}, we also have $B_{s}=\sigma B_{t}$. (When we consider modified $C$-pattern or modified $G$-pattern, we obtain $\tilde{B}_s=\sigma \tilde{B}_{t}$ from \Cref{lem: equivalency of the periodicity among modified c- g-vectors}.) Suppose that $s$ and $s'$ are $k$-adjacent. Namely, we have $C_{s'}=\mu_{k}(C_s)$. Now, we already know that $C_{s}=\tilde{\sigma} C_{t}$ and $B_{s}=\sigma B_{t}$. By Proposition~\ref{prop: periodicity}, this means that $C_{s'}=\mu_{k}(C_s)=\tilde{\sigma}(\mu_{\sigma^{-1}(k)}(C_t))=\tilde{\sigma}(C_{t_{\sigma^{-1}(k)}})$. Thus, we have $[C_{s'}]=[C_{t_{\sigma^{-1}(k)}}]$, which implies $[s'] = [t_{\sigma^{-1}(k)}]$ in ${\bf EG}({\bf C}(B))$. Since $s' \in [t']$, we have $[t']=[t_{\sigma^{-1}(k)}]$ as we desired.
\par
Next, we show the claim for the $G$-fan. By \Cref{prop: matrix-cone synchronicity}, the equality  $\mathcal{C}(G_t)=\mathcal{C}(G_{t'})$ is equivalent to $\tilde{G}_{t'}=\tilde{\sigma}\tilde{G}_{t}$ for some $\sigma \in \mathfrak{S}_{n}$. Thus, we can do the same argument for the exchange graph associated with the modified $G$-pattern and show the claim.
\end{proof}
In the following, we summarize the relationship among these exchange graphs.
\begin{theorem}\label{thm: relationship among exchange graphs}
Let $B \in {\bf SC}$ of rank $n$.\\
\textup{($a$)} The following canonical graph isomorphism holds.
\begin{equation}
{\bf EG}(\tilde{\bf C}(B)) \cong {\bf EG}({\bf \tilde{G}}(B)).
\end{equation}
\textup{($b$)} Suppose that Conjecture~\ref{conj: standard and discreteness conjecture} holds for this $B$. Then, the following canonical graph isomorphisms hold.
\begin{equation}
{\bf EG}(\tilde{\bf C}(B)) \cong {\bf EG}(\tilde{\bf G}(B)) \cong {\bf EG}(\Delta_{\bf G}(B)).
\end{equation}
\textup{($c$)} Suppose that Conjecture~\ref{conj: standard and discreteness conjecture} holds for this $B$. Then, 
the following canonical graph isomorphism holds.
\begin{equation}\label{eq: synchronicity among C-, G-patterns}
{\bf EG}({\bf C}(B)) \cong {\bf EG}({\bf G}(B)).
\end{equation}
\textup{($d$)} Suppose that Conjecture~\ref{conj: standard and discreteness conjecture} holds for this $B$. We define the equivalence relation $\approx$ on the set of all bases of $\mathbb{R}^{n}$.
\begin{equation}
\{{\bf u}_1,\dots,{\bf u}_n\} \approx \{{\bf v}_1,\dots,{\bf v}_n\} \Longleftrightarrow\ \textup{$\exists \sigma \in \mathfrak{S}_n$ and $\exists \lambda_i \in \mathbb{R}_{>0}$ such that ${\bf v}_{\sigma(i)}=\lambda_i{\bf u}_i$}.
\end{equation}
Here, we identify the vertices of ${\bf EG}({\bf C}(B))$ and ${\bf EG}({\bf G}(B))$ as the clusters of $c$-, $g$-vectors. (Then, the above $\approx$ is an equivalence relation of these vertex sets.)
Then, we have the following canonical graph isomorphisms.
\begin{equation} 
{\bf EG}(\tilde{\bf C}(B)) \cong {\bf EG}(\tilde{\bf G}(B)) \cong {\bf EG}(\Delta_{\bf G}(B)) \cong {\bf EG}({\bf C}(B))/{\approx} \cong {\bf EG}({\bf G}(B))/{\approx}.
\end{equation}
\end{theorem}
\begin{proof}
The claim ($a$) follows from \Cref{lem: equivalency of the periodicity among modified c- g-vectors} and the claim ($b$) follows from \Cref{prop: matrix-cone synchronicity}. Furthermore, the claim ($c$) follows from \Cref{prop: CG synchronicity}. To prove ($d$), we need to show that ${\bf EG}(\Delta_{\bf G}(B)) \cong {\bf EG}({\bf C}(B))/{\approx}$ and ${\bf EG}(\Delta_{\bf G}(B)) \cong {\bf EG}({\bf G}(B))/{\approx}$. The latter follows from $\mathcal{C}(G_{t})=\mathcal{C}(G_{t'}) \Leftrightarrow [G_{t}] \approx [G_{t'}]$. By Lemma~\ref{lem: c-vector expression of G-cones}, $[C_{t}]\approx[C_{t'}] \Rightarrow \mathcal{C}(G_{t})=\mathcal{C}(G_{t'})$ holds. Conversely, if $\mathcal{C}(G_{t})=\mathcal{C}(G_{t'})$, since they are $n$-dimensional cones, it implies that all normal vectors of their $(n-1)$-dimensional faces have the same direction. By Lemma~\ref{lem: c-vector expression of G-cones}, their normal vectors are parallel to the $c$-vectors, which implies that $[C_{t}] \approx [C_{t'}]$. Thus, ${\bf EG}({\bf C}(B))/{\approx} \cong {\bf EG}(\Delta_{\bf G}(B))$ holds.
\end{proof}
\subsection*{Acknowledgements} 
 The authors would like to express their sincere gratitude to Tomoki Nakanishi for his thoughtful guidance. The authors also wish to thank Peigen Cao, Changjian Fu, Yasuaki Gyoda, Fang Li, Lang Mou and Salvatore Stella for their valuable discussions and insightful suggestions. We also thank Nathan Reading for explaining the reference and valuable comments. In addition, Z. Chen wants to thank Xiaowu Chen, Zhe Sun and Yu Ye for their help and support. R. Akagi is supported by JSPS KAKENHI Grant Number JP25KJ1438 and Chubei Itoh Foundation. Z. Chen is supported by National Natural Science Foundation of China (Grant No. 124B2003) and China Scholarship Council (Grant No. 202406340022).
\clearpage
\appendix
\section{Proof of Lemma~\ref{lem: type H}}\label{sec: proof of Lemma type H}
The purpose of this section is to share the program to calculate all $C$-matrices. The SageMath source code used for the computations in this paper is available in the GitHub repository \cite{Aka26}.
\subsection{Example of the program}
We will use the program for Sage Math 9.3 written in Program~code~\ref{program}. The main functions are the following.
\\
{\bf B\_pattern($B_0$, $l$)}\par
Arguments are a skew-symmetrizable matrix $B_0$ and a positive integer $l \in \mathbb{Z}_{\geq 1}$. Return is separated as the following four objects:
\begin{itemize}
\item ($B$-pattern) All distinct $B$-matrices obtained by applying mutations at most $l$ times to $B_0$ up to the action of permutations.
\item (Periodicity) All minimal periodicity up to permutationd.
\item (Finiteness) If all $B$-matrices are obtained by applying $i \leq l-1$ times, it returns ``finite, maximum depth = $i$". If not, return ``undeterminable".
\item (Size) The number of distinct $B$-matrices applying mutations at most $l$ times.
\end{itemize}
We can see one example in \Cref{fig: B-pattern of type A3}. Each index $[k_1,k_2,\dots,k_r]$ means that the corresponding matrix $B$ below is obtained by $B=\mu_{k_r}\cdots\mu_{k_2}\mu_{k_1}(B_0)$.  In Periodicity, each permutation $[p_1,p_2,\dots,p_n]$ corresponds to $\sigma=(i \mapsto p_i) \in \mathfrak{S}_n$, and the following ``same as ${\bf w}$" means that the $B$-matrix $B^{\bf w}$ located at the index ${\bf w}$ is the same as this matrix up to the difference of this permutation $B^{\bf w}=\sigma B$.
\\
{\bf C\_pattern($B_0$, $l$)}
\par
Arguments are the same as the ones of B\_pattern($B_0$, $l$). Returns are also almost the same by replacing $B$-matrices to $C$-matrices, but additionally, it returns the following data.
\begin{itemize}
\item (Sign-coherence) If all $C$-matrices obtained by applying mutations at most $l$ times are sign-coherent, it returns ``sign-coherent up to $l$". If not, it returns ``incoherent" and the list of all indices whose $C$-matrices are not sign-coherent.
\end{itemize}
We can see a sign-coherent example in \Cref{fig: C-pattern of type A2} and an incoherent example in \Cref{fig: incoherent C-pattern}.
\begin{figure}[htbp]
  \centering
  \includegraphics[width=\textwidth, keepaspectratio]{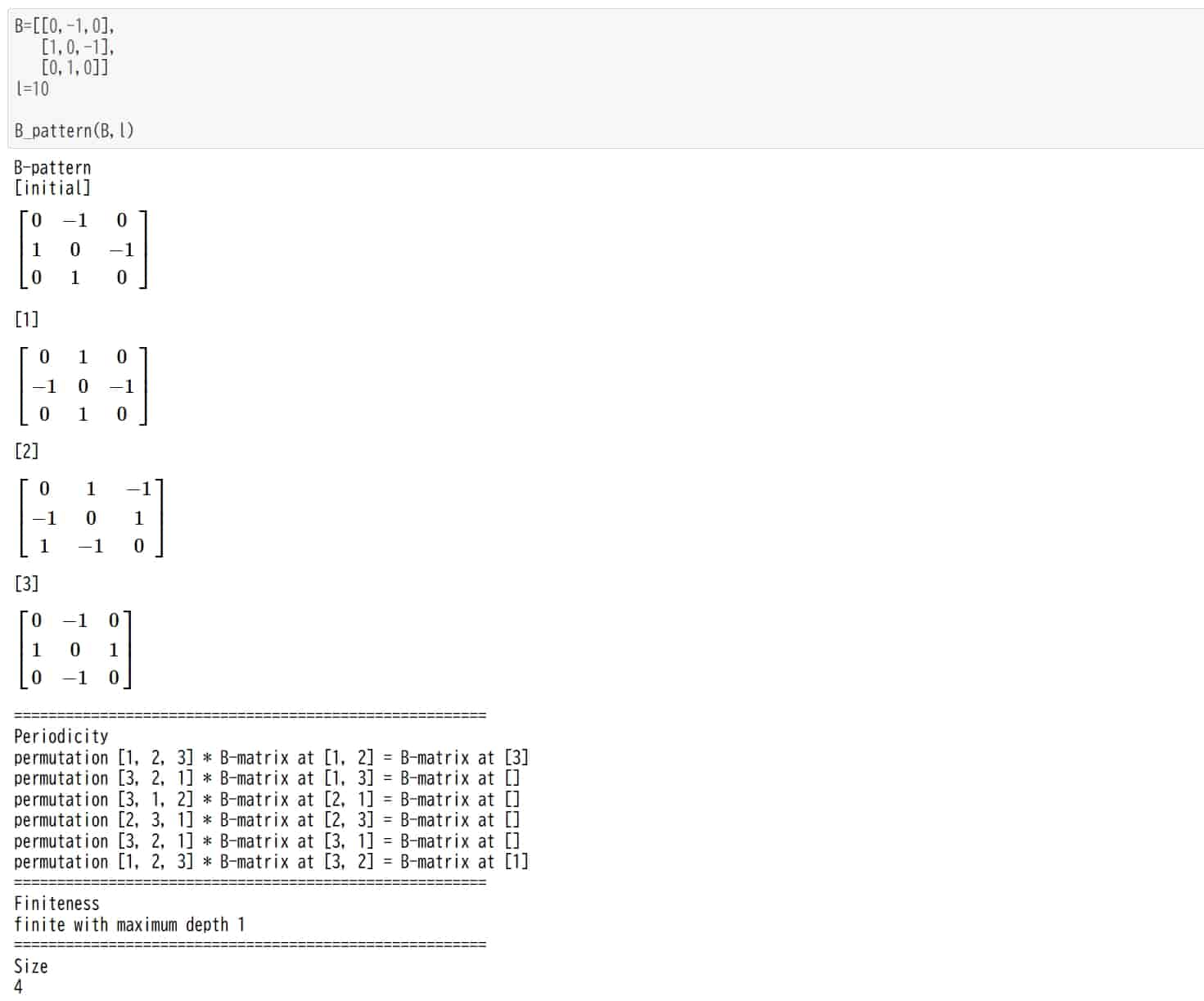}
  \caption{Example of a $B$-pattern.}
  \label{fig: B-pattern of type A3}
\end{figure}
\begin{figure}[htbp]
  \centering
  \includegraphics[width=\textwidth, keepaspectratio]{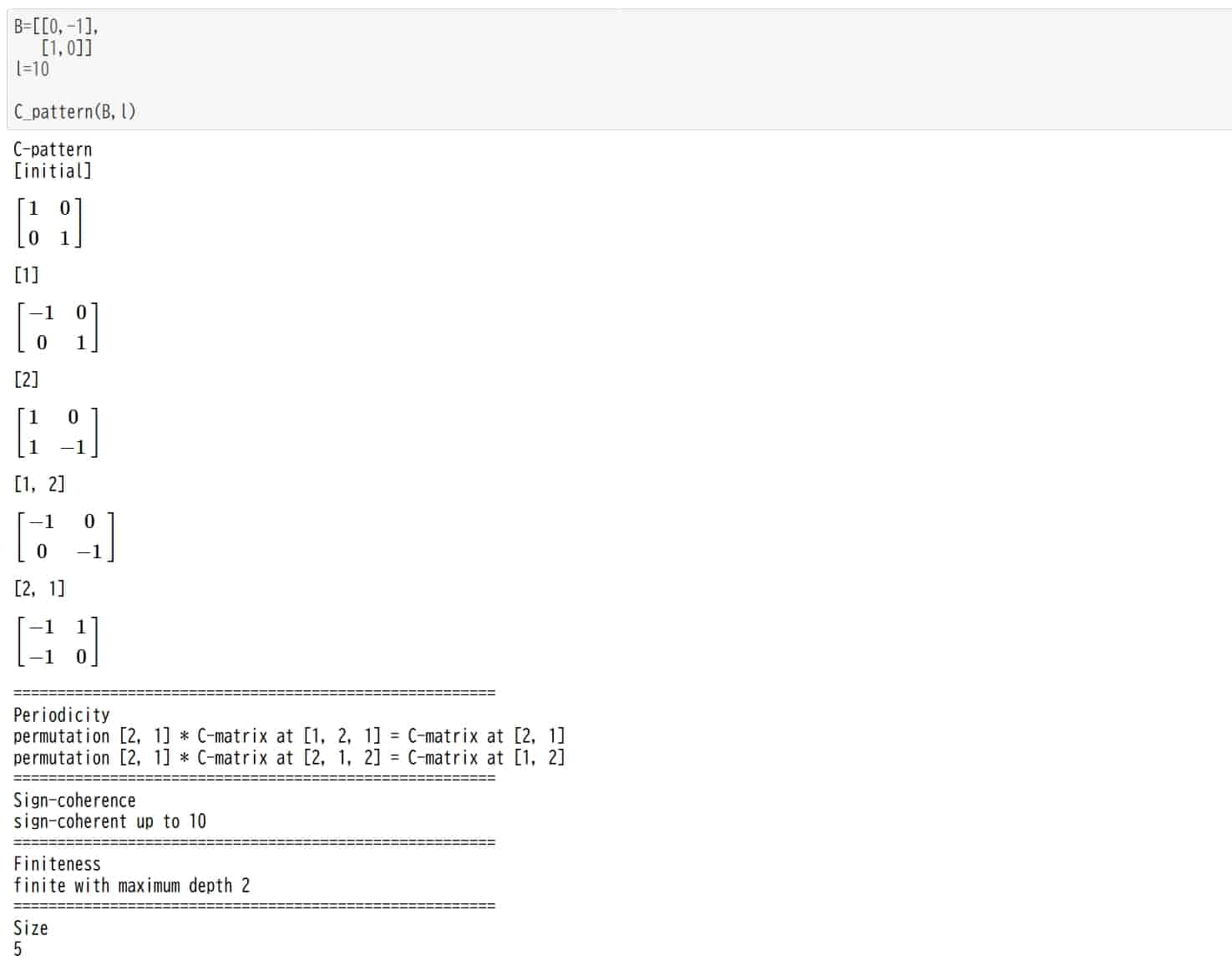}
  \caption{Example of a sign-coherent $C$-pattern.}
  \label{fig: C-pattern of type A2}
\end{figure}
\begin{figure}[htbp]
  \centering
  \includegraphics[width=\textwidth, keepaspectratio]{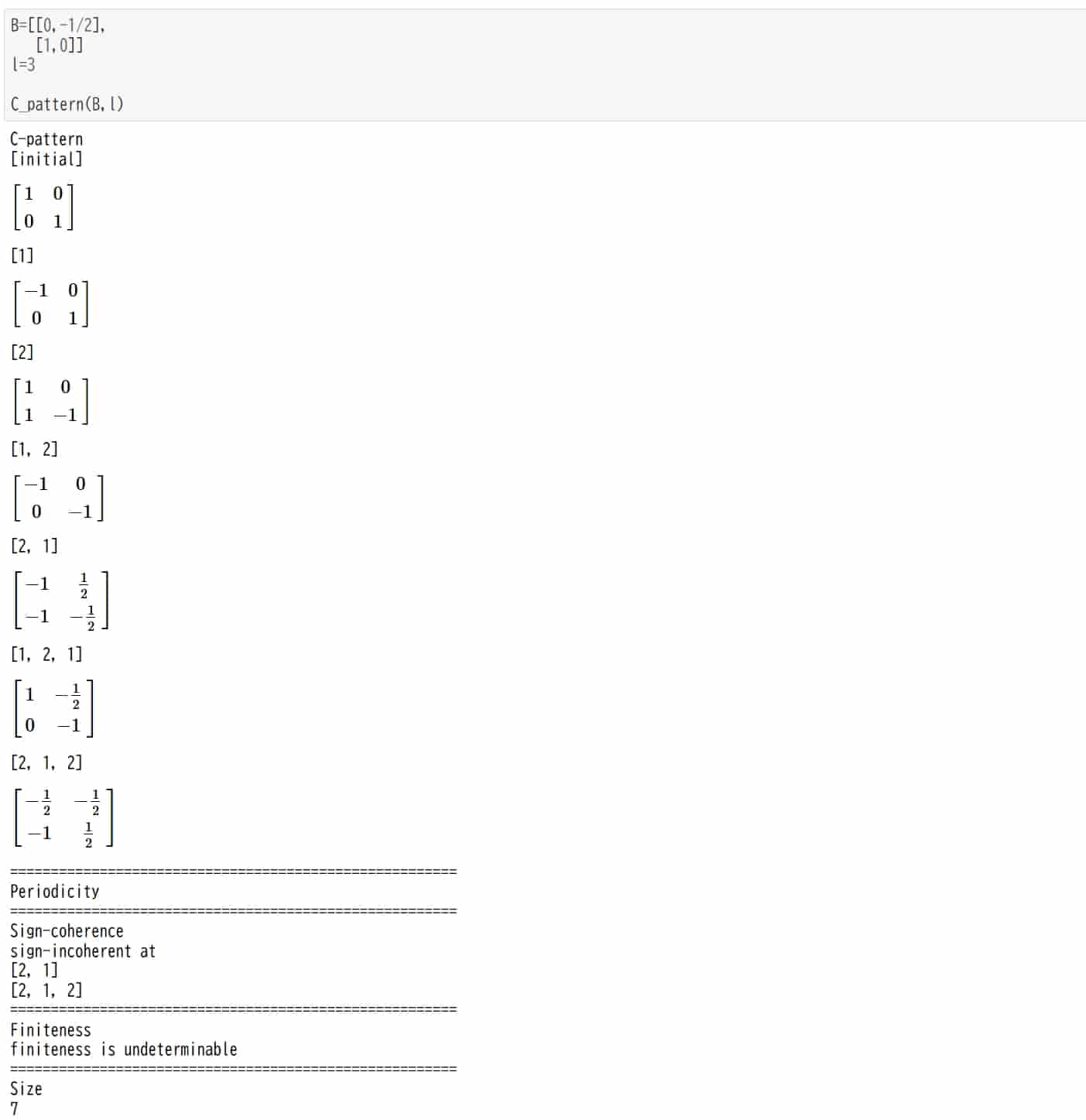}
  \caption{Example of a sign-incoherent $C$-pattern.}
  \label{fig: incoherent C-pattern}
\end{figure}
\subsection{Results for type $H_3$ and $H_4$}
For simplicity, we set $\phi=2\cos{\frac{\pi}{5}}=\frac{1+\sqrt{5}}{2}$.
By using this program, we can show Lemma~\ref{lem: type H} by only finitely many times calculation. For the reader's convenience, we give a $B$-pattern of type $H_3$ in Figure~\ref{fig: B-pattern of type H3}, and of type $H_4$ in Figure~\ref{fig: B-pattern of type H4}. (Note that, to show Lemma~\ref{lem: type H}, we also need to check their transposition.) So, we can finish the proof by only finitely many calculations. However, we need so many pages to write all $C$-patterns. Here, we write the only one case whose initial exchange matrix is
\begin{equation}
B=\left(\begin{matrix}
0 & -\phi & \phi\\
\phi & 0 & -\phi\\
-\phi & \phi & 0
\end{matrix}\right),
\end{equation}
in Figure~\ref{fig: C-pattern of type H3}. (We can find this $B$-matrix at $[2,1]$ in Figure~\ref{fig: B-pattern of type H3}.)
\par
If we run this program, we should set $l=7$ for type $H_3$ and $l=11$ for type $H_4$. (Note that we need a little long time to complete it. The authors needed to wait about 10 minutes for each initial exchange matrix of type $H_4$.)
\par
We summarize the important properties which we can easily obtain from this program.
\begin{proposition}
\textup{($a$)} Let $B$ be mutation-equivalent to any of type $H_3$. Then, the number of distinct $C$-matrices (up to the difference of permutation) is $32$. Moreover, we can obtain all $C$-matrices by applying mutations at most $6$ times.
\\
\textup{($b$)} Let $B$ be mutation-equivalent to any of type $H_4$. Then, the number of distinct $C$-matrices (up to the difference of permutation) is $280$. Moreover, we can obtain all $C$-matrices by applying mutations at most $10$ times.
\end{proposition}
By Corollary~\ref{cor: synchronicity for the skew-symmetric case}, the number of $G$-cones is the same as the number of distinct clusters of $c$-vectors. We summarize the number of $G$-cones and $g$-vectors for each finite type in (\ref{eq: Number finite type}). In the above row, that is, for the types corresponding to crystallographic root systems, these numbers have already been obtained in \cite{FZ03a}. See also \cite[Fig.~5.17]{FWZ16}. Due to \Cref{thm: ordinary synchronicity}, the number of cones and seeds is the same. Note that, by \Cref{prop: dual mutation}, these numbers only depend on the $B$-pattern, not the initial exchange matrix.
\begin{equation}\label{eq: Number finite type}
\begin{aligned}
&\begin{array}{|c||c|c|c|c|c|c|c|c|c|c|}
\hline
X_n & A_n & B_n=C_n & D_n & E_6 & E_7 & E_8 & F_4\\
\hline
\#\textup{$G$-cones} & \frac{1}{n+2}\binom{2n+2}{n+1} & \binom{2n}{n} & \frac{3n-2}{n}\binom{2n-2}{n-1} & 833 & 4160 & 25080 & 105\\
\hline
\#\textup{$g$-vectors} & \frac{n(n+3)}{2} & n(n+1) & n^2 & 42 & 70 & 128 & 28
\\
\hline
\end{array}
\\
&\begin{array}{|c||c|c|c|}
\hline
X_n & H_3 & H_4 & I_2(m)\\
\hline
\#\textup{$G$-cones} & 32 & 280 & m+2\\
\hline
\#\textup{$g$-vectors} & 18 & 64 & m+2\\
\hline
\end{array}
\end{aligned}
\end{equation}
These numbers have already appeared in the theory of Coxeter groups. It is known that the number of $g$-vectors coincides with the number of {\em almost positive roots} \cite{FZ03a}, and the number of $G$-cones coincides with the number of chambers induced by the almost positive roots in the Coxeter arrangements. The number of chambers has already been given by \cite[Prop.~3.9]{FZ03b}, and also see \cite[Fig.~5.1, Fig.~5.15]{FR07} for the non-crystallographic type.

\begin{figure}
\begin{equation*}
\begin{gathered}
\begin{gathered}
\textup{initial}\\
\left(\begin{smallmatrix}
0 & -\phi &0 \\
\phi &0 &-1 \\
0 &1 &0 \\
\end{smallmatrix}\right)
\end{gathered}\ 
\begin{gathered}
[1]\\
\left(\begin{smallmatrix}
0 & \phi &0 \\
- \phi &0 &-1 \\
0 &1 &0 \\
\end{smallmatrix}\right)
\end{gathered}\ 
\begin{gathered}
[2]\\
\left(\begin{smallmatrix}
0 &\phi &- \phi \\
- \phi &0 &1 \\
\phi &-1 &0 \\
\end{smallmatrix}\right)
\end{gathered}\ 
\begin{gathered}
[3]\\
\left(\begin{smallmatrix}
0 &- \phi &0 \\
\phi &0 &1 \\
0 &-1 &0 \\
\end{smallmatrix}\right)
\end{gathered}\ 
\begin{gathered}
[1, 3]\\
\left(\begin{smallmatrix}
0 &\phi &0 \\
- \phi &0 &1 \\
0 &-1 &0 \\
\end{smallmatrix}\right)
\end{gathered}\ 
\begin{gathered}
[2, 1]\\
\left(\begin{smallmatrix}
0 &- \phi &\phi \\
\phi &0 &- \phi \\
- \phi &\phi &0 \\
\end{smallmatrix}\right)
\end{gathered}\ 
\end{gathered}
\end{equation*}
\caption{$B$-pattern of type $H_3$}\label{fig: B-pattern of type H3}
\end{figure}
\begin{figure}
\begin{equation*}
\begin{gathered}
\begin{gathered}
\textup{initial}\\
\left(\begin{smallmatrix}
0 &- \phi &0 &0 \\
\phi &0 &-1 &0 \\
0 &1 &0 &-1 \\
0 &0 &1 &0 \\
\end{smallmatrix}\right)
\end{gathered}\ 
\begin{gathered}
[1]\\
\left(\begin{smallmatrix}
0 &\phi &0 &0 \\
- \phi &0 &-1 &0 \\
0 &1 &0 &-1 \\
0 &0 &1 &0 \\
\end{smallmatrix}\right)
\end{gathered}\ 
\begin{gathered}
[2]\\
\left(\begin{smallmatrix}
0 &\phi &- \phi &0 \\
- \phi &0 &1 &0 \\
\phi &-1 &0 &-1 \\
0 &0 &1 &0 \\
\end{smallmatrix}\right)
\end{gathered}\ 
\begin{gathered}
[3]\\
\left(\begin{smallmatrix}
0 &- \phi &0 &0 \\
\phi &0 &1 &-1 \\
0 &-1 &0 &1 \\
0 &1 &-1 &0 \\
\end{smallmatrix}\right)
\end{gathered}\ 
\begin{gathered}
[4]\\
\left(\begin{smallmatrix}
0 &- \phi &0 &0 \\
\phi &0 &-1 &0 \\
0 &1 &0 &1 \\
0 &0 &-1 &0 \\
\end{smallmatrix}\right)
\end{gathered}\ 
\\
\begin{gathered}
[1, 2]\\
\left(\begin{smallmatrix}
0 &- \phi &0 &0 \\
\phi &0 &1 &0 \\
0 &-1 &0 &-1 \\
0 &0 &1 &0 \\
\end{smallmatrix}\right)
\end{gathered}\ 
\begin{gathered}
[1, 3]\\
\left(\begin{smallmatrix}
0 &\phi &0 &0 \\
- \phi &0 &1 &-1 \\
0 &-1 &0 &1 \\
0 &1 &-1 &0 \\
\end{smallmatrix}\right)
\end{gathered}\ 
\begin{gathered}
[1, 4]\\
\left(\begin{smallmatrix}
0 &\phi &0 &0 \\
- \phi &0 &-1 &0 \\
0 &1 &0 &1 \\
0 &0 &-1 &0 \\
\end{smallmatrix}\right)
\end{gathered}\ 
\begin{gathered}
[2, 1]\\
\left(\begin{smallmatrix}
0 &- \phi &\phi &0 \\
\phi &0 &- \phi &0 \\
- \phi &\phi &0 &-1 \\
0 &0 &1 &0 \\
\end{smallmatrix}\right)
\end{gathered}\ 
\begin{gathered}
[2, 3]\\
\left(\begin{smallmatrix}
0 &0 &\phi &- \phi \\
0 &0 &-1 &0 \\
- \phi &1 &0 &1 \\
\phi &0 &-1 &0 \\
\end{smallmatrix}\right)
\end{gathered}\ 
\\
\begin{gathered}
[2, 4]\\
\left(\begin{smallmatrix}
0 &\phi &- \phi &0 \\
- \phi &0 &1 &0 \\
\phi &-1 &0 &1 \\
0 &0 &-1 &0 \\
\end{smallmatrix}\right)
\end{gathered}\ 
\begin{gathered}
[3, 2]\\
\left(\begin{smallmatrix}
0 &\phi &0 &- \phi \\
- \phi &0 &-1 &1 \\
0 &1 &0 &0 \\
\phi &-1 &0 &0 \\
\end{smallmatrix}\right)
\end{gathered}\ 
\begin{gathered}
[3, 4]\\
\left(\begin{smallmatrix}
0 &- \phi &0 &0 \\
\phi &0 &0 &1 \\
0 &0 &0 &-1 \\
0 &-1 &1 &0 \\
\end{smallmatrix}\right)
\end{gathered}\ 
\begin{gathered}
[1, 2, 1]\\
\left(\begin{smallmatrix}
0 &\phi &0 &0 \\
- \phi &0 &1 &0 \\
0 &-1 &0 &-1 \\
0 &0 &1 &0 \\
\end{smallmatrix}\right)
\end{gathered}\ 
\begin{gathered}
[1, 3, 4]\\
\left(\begin{smallmatrix}
0 &\phi &0 &0 \\
- \phi &0 &0 &1 \\
0 &0 &0 &-1 \\
0 &-1 &1 &0 \\
\end{smallmatrix}\right)
\end{gathered}\ 
\\
\begin{gathered}
[2, 1, 3]\\
\left(\begin{smallmatrix}
0 &1 &- \phi &0 \\
-1 &0 &\phi &- \phi \\
\phi &- \phi &0 &1 \\
0 &\phi &-1 &0 \\
\end{smallmatrix}\right)
\end{gathered}\ 
\begin{gathered}
[2, 1, 4]\\
\left(\begin{smallmatrix}
0 &- \phi &\phi &0 \\
\phi &0 &- \phi &0 \\
- \phi &\phi &0 &1 \\
0 &0 &-1 &0 \\
\end{smallmatrix}\right)
\end{gathered}\ 
\begin{gathered}
[2, 1, 3, 1]\\
\left(\begin{smallmatrix}
0 &-1 &\phi &0 \\
1 &0 &0 &- \phi \\
- \phi &0 &0 &1 \\
0 &\phi &-1 &0 \\
\end{smallmatrix}\right)
\end{gathered}\ 
\end{gathered}
\end{equation*}
\caption{$B$-pattern of type $H_4$}\label{fig: B-pattern of type H4}
\end{figure}
\begin{figure}
\begin{equation*}
\begin{gathered}
\begin{gathered}
{\tiny \textup{initial}}\\
\left(\begin{smallmatrix}
1 &0 &0 \\
0 &1 &0 \\
0 &0 &1 \\
\end{smallmatrix}\right)
\end{gathered}\ 
\begin{gathered}
[1]\\
\left(\begin{smallmatrix}
-1 &0 &\phi \\
0 &1 &0 \\
0 &0 &1 \\
\end{smallmatrix}\right)
\end{gathered}\ 
\begin{gathered}
[2]\\
\left(\begin{smallmatrix}
1 &0 &0 \\
\phi &-1 &0 \\
0 &0 &1 \\
\end{smallmatrix}\right)
\end{gathered}\ 
\begin{gathered}
[3]\\
\left(\begin{smallmatrix}
1 &0 &0 \\
0 &1 &0 \\
0 &\phi &-1 \\
\end{smallmatrix}\right)
\end{gathered}\ 
\begin{gathered}
[1, 2]\\
\left(\begin{smallmatrix}
-1 &0 &\phi \\
0 &-1 &1 \\
0 &0 &1 \\
\end{smallmatrix}\right)
\end{gathered}\ 
\begin{gathered}
[1, 3]\\
\left(\begin{smallmatrix}
\phi &0 &- \phi \\
0 &1 &0 \\
\phi &0 &-1 \\
\end{smallmatrix}\right)
\end{gathered}\ 
\\
\begin{gathered}
[2, 1]\\
\left(\begin{smallmatrix}
-1 &\phi &0 \\
- \phi &\phi &0 \\
0 &0 &1 \\
\end{smallmatrix}\right)
\end{gathered}\ 
\begin{gathered}
[2, 3]\\
\left(\begin{smallmatrix}
1 &0 &0 \\
\phi &-1 &0 \\
1 &0 &-1 \\
\end{smallmatrix}\right)
\end{gathered}\ 
\begin{gathered}
[3, 1]\\
\left(\begin{smallmatrix}
-1 &1 &0 \\
0 &1 &0 \\
0 &\phi &-1 \\
\end{smallmatrix}\right)
\end{gathered}\ 
\begin{gathered}
[3, 2]\\
\left(\begin{smallmatrix}
1 &0 &0 \\
0 &-1 &\phi \\
0 &- \phi &\phi \\
\end{smallmatrix}\right)
\end{gathered}\ 
\begin{gathered}
[1, 2, 1]\\
\left(\begin{smallmatrix}
1 &- \phi &\phi \\
0 &-1 &1 \\
0 &0 &1 \\
\end{smallmatrix}\right)
\end{gathered}\ 
\begin{gathered}
[1, 2, 3]\\
\left(\begin{smallmatrix}
-1 &\phi &- \phi \\
0 &0 &-1 \\
0 &1 &-1 \\
\end{smallmatrix}\right)
\end{gathered}\ 
\\
\begin{gathered}
[1, 3, 1]\\
\left(\begin{smallmatrix}
- \phi &0 &1 \\
0 &1 &0 \\
- \phi &0 &\phi \\
\end{smallmatrix}\right)
\end{gathered}\ 
\begin{gathered}
[1, 3, 2]\\
\left(\begin{smallmatrix}
\phi &0 &- \phi \\
0 &-1 &0 \\
\phi &0 &-1 \\
\end{smallmatrix}\right)
\end{gathered}\ 
\begin{gathered}
[2, 1, 2]\\
\left(\begin{smallmatrix}
\phi &- \phi &0 \\
1 &- \phi &0 \\
0 &0 &1 \\
\end{smallmatrix}\right)
\end{gathered}\ 
\begin{gathered}
[2, 1, 3]\\
\left(\begin{smallmatrix}
-1 &\phi &0 \\
- \phi &\phi &0 \\
0 &0 &-1 \\
\end{smallmatrix}\right)
\end{gathered}\ 
\begin{gathered}
[2, 3, 1]\\
\left(\begin{smallmatrix}
-1 &0 &1 \\
- \phi &-1 &\phi \\
-1 &0 &0 \\
\end{smallmatrix}\right)
\end{gathered}\ 
\begin{gathered}
[2, 3, 2]\\
\left(\begin{smallmatrix}
1 &0 &0 \\
\phi &1 &- \phi \\
1 &0 &-1 \\
\end{smallmatrix}\right)
\end{gathered}\ 
\\
\begin{gathered}
[3, 1, 2]\\
\left(\begin{smallmatrix}
0 &-1 &0 \\
1 &-1 &0 \\
\phi &- \phi &-1 \\
\end{smallmatrix}\right)
\end{gathered}\ 
\begin{gathered}
[3, 1, 3]\\
\left(\begin{smallmatrix}
-1 &1 &0 \\
0 &1 &0 \\
- \phi &\phi &1 \\
\end{smallmatrix}\right)
\end{gathered}\ 
\begin{gathered}
[3, 2, 1]\\
\left(\begin{smallmatrix}
-1 &0 &0 \\
0 &-1 &\phi \\
0 &- \phi &\phi \\
\end{smallmatrix}\right)
\end{gathered}\ 
\begin{gathered}
[3, 2, 3]\\
\left(\begin{smallmatrix}
1 &0 &0 \\
0 &\phi &- \phi \\
0 &1 &- \phi \\
\end{smallmatrix}\right)
\end{gathered}\ 
\begin{gathered}
[1, 2, 1, 3]\\
\left(\begin{smallmatrix}
1 &0 &- \phi \\
0 &0 &-1 \\
0 &1 &-1 \\
\end{smallmatrix}\right)
\end{gathered}\ 
\begin{gathered}
[1, 3, 1, 2]\\
\left(\begin{smallmatrix}
- \phi &0 &1 \\
0 &-1 &0 \\
- \phi &0 &\phi \\
\end{smallmatrix}\right)
\end{gathered}\ 
\\
\begin{gathered}
[2, 1, 2, 3]\\
\left(\begin{smallmatrix}
\phi &- \phi &0 \\
1 &- \phi &0 \\
0 &0 &-1 \\
\end{smallmatrix}\right)
\end{gathered}\ 
\begin{gathered}
[2, 3, 1, 2]\\
\left(\begin{smallmatrix}
-1 &0 &1 \\
- \phi &1 &0 \\
-1 &0 &0 \\
\end{smallmatrix}\right)
\end{gathered}\ 
\begin{gathered}
[3, 1, 2, 3]\\
\left(\begin{smallmatrix}
0 &-1 &0 \\
1 &-1 &0 \\
0 &- \phi &1 \\
\end{smallmatrix}\right)
\end{gathered}\ 
\begin{gathered}
[3, 2, 1, 3]\\
\left(\begin{smallmatrix}
-1 &0 &0 \\
0 &\phi &- \phi \\
0 &1 &- \phi \\
\end{smallmatrix}\right)
\end{gathered}\ 
\begin{gathered}
[1, 2, 1, 3, 2]\\
\left(\begin{smallmatrix}
1 &0 &- \phi \\
0 &0 &-1 \\
0 &-1 &0 \\
\end{smallmatrix}\right)
\end{gathered}\ 
\begin{gathered}
[1, 3, 1, 2, 3]\\
\left(\begin{smallmatrix}
0 &0 &-1 \\
0 &-1 &0 \\
1 &0 &- \phi \\
\end{smallmatrix}\right)
\end{gathered}\ 
\\
\begin{gathered}
[2, 3, 1, 2, 3]\\
\left(\begin{smallmatrix}
0 &0 &-1 \\
- \phi &1 &0 \\
-1 &0 &0 \\
\end{smallmatrix}\right)
\end{gathered}\ 
\begin{gathered}
[1, 2, 1, 3, 2, 1]\\
\left(\begin{smallmatrix}
-1 &0 &0 \\
0 &0 &-1 \\
0 &-1 &0 \\
\end{smallmatrix}\right)
\end{gathered}\ 
\end{gathered}
\end{equation*}
\caption{$C$-pattern with the initial exchange matrix $B_0=\left(\begin{smallmatrix}
0 & -\phi & \phi\\
\phi & 0 & -\phi\\
-\phi & \phi & 0
\end{smallmatrix}\right)$}\label{fig: C-pattern of type H3}
\end{figure}

\clearpage
\begin{lstlisting}[caption=Matrix mutations, label=program, showstringspaces=false]
import sympy as sp
from sympy.combinatorics import Permutation, SymmetricGroup

def positive(A):
    return sp.Matrix(A).applyfunc(lambda x: max(x,0))

def E(i,j,n):
    """
    the almostly zero matrix but only (i,j)th entry is 1
    """
    res=sp.zeros(n,n)
    res[i-1,j-1]=1
    return res

def J(k,n):
    """
    the almostly identity matrix but only (k,k)th entry is -1
    """
    res=sp.eye(n)
    res[k-1,k-1]=-1
    return res

def vanish_row(A,k):
    """
    replace 0 except for the kth row of A
    """
    A=sp.Matrix(A)
    n=A.shape[0]
    return E(k,k,n)*A

def vanish_col(A,k):
    """
    replace 0 except for the kth column of A
    """
    A=sp.Matrix(A)
    n=A.shape[0]
    return A*E(k,k,n)

def check_col_sign_coherence(A):
    r"""
    If A is column sign-coherent, return True.
    Else, return False
    """
    A=sp.Matrix(A)
    m=A.shape[0]
    n=A.shape[1]
    for j in range(n):
        col_vector=A.col(j)
        nonzero_elements=[x for x in col_vector if x!=0]
        if not nonzero_elements:
            pass
        else:
            first_sign=sp.sign(nonzero_elements[0])
            for ele in nonzero_elements:
                if sp.sign(ele)!=first_sign:
                    return False
    return True

def Permutation_Matrix(n,s):
    r"""
    permutation matrix defined in (2.10)
    """
    s_inv=s**(-1)
    res=sp.zeros(n)
    for i in range(n):
        res[i,s_inv(i)]=1
    return res

def Col_Arrangement(A,s):
    r"""
    move the ith column to the s(i)th column.
    """
    A=sp.Matrix(A)
    n=A.shape[0]
    P=Permutation_Matrix(n,s)
    return A*P

def Row_Col_Arrangement(A,s):
    r"""
    move the (i,j)th entry to the (s(i),s(j))th entry.
    """
    A=sp.Matrix(A)
    n=A.shape[0]
    P=Permutation_Matrix(n,s)
    P_top=P.transpose()
    return P_top*A*P

def mut_B(B,k):
    r"""
    B-mutation indexed by k=1,2,3,...,n
    """
    B=sp.Matrix(B)
    n=B.shape[0]
    J_k=J(k,n)
    B_L=vanish_col(positive(-B),k)
    B_R=vanish_row(positive(B),k)
    return ((J_k+B_L)*B*(J_k+B_R)).applyfunc(expand)

def mut_C(C,B,k):
    r"""
    C-mutation indexed by k=1,2,3,...,n
    """
    C=sp.Matrix(C)
    B=sp.Matrix(B)
    n=C.shape[0]
    J_k=J(k,n)
    M1=C*J_k
    M2=C*vanish_row(positive(B),k)
    M3=vanish_col(positive(-C),k)*B
    return (M1+M2+M3).applyfunc(expand)

def distinct_B_dict(B0,l):
    r"""
    arg:
        B0: initial exchange matrix
        l: length of the added mutation sequences
    return:
        dict:
            key="B" -> value= dictionary of B-matrices such that duplicates are eliminated. (key is tuples (k1,k2,...,kr))
            key="period" -> value=period_dict
                                key=tuple (k1,k2,...,kr) corresponds to the mutation sequence between w0 and w0+(l indices)
                                value=[p=permutation, T=another tuple in B_dict] such that p*(this B-matrix) = (B-matrix at T)
            key="finiteness" -> value="finite" if all B-matrices are obtained and "undeterminable" if not
            key="depth" -> value=the maximum depth in B_dict
    """
    B0=sp.Matrix(B0)
    n=B0.shape[0]
    symmetric_group=Permutations(range(n))
    B_dict={():B0}
    period_dict=dict()
    current_key_list=[()]
    for depth in range(l):
        next_key_list=[]
        for key_pre in current_key_list:
            B_pre=B_dict[key_pre]
            for k in range(1,n+1):
                if key_pre==() or key_pre[-1]!=k:
                    key_new=key_pre+(k,)
                    B_new=mut_B(B_pre,k)
                    exit_loop=False
                    permutated_B_dict=dict()
                    for p in symmetric_group:
                        permutated_B_dict[p]=Row_Col_Arrangement(B_new,Permutation(p))
                    for w, B_w in B_dict.items():
                        for p, B in permutated_B_dict.items():
                            if B_w==B:
                                p=[i+1 for i in p]
                                period_dict[key_new]=[p,w]
                                exit_loop=True
                                break
                            else:
                                pass
                    if exit_loop==False:
                        next_key_list+=[key_new]
                        B_dict[key_new]=B_new
                    else:
                        pass
        if next_key_list==[]:
            return {"B":B_dict, "period": period_dict, "fin": f"finite with maximum depth {depth}", "size": len(B_dict)}
        else:
            current_key_list=next_key_list
    return {"B":B_dict, "period": period_dict, "fin": "finiteness is undeterminable", "size": len(B_dict)}

def distinct_BC_dict(B0,l):
    r"""
    arg:
        B0: initial exchange matrix
        l: length of the mutation sequences
    return:
        dict:
            key="C" -> value= the dict of (B,C)-matrices.
            key="period" -> value=period_dict
                                key=tuple (k1,k2,...,kr) corresponds to the mutation sequence up to l
                                value=[p=permutation, T=another tuple in G_dict] such that p*(this C-matrix) = (C-matrix at T)
            key="SIC" -> value=the list of mutation sequences such that the corresponding C-matrix is sign-incoherent
            key="fin" -> value="finite" if all C-matrices are obtained and "undeterminable" if not
            key="depth" -> value=the maximum depth in C_dict
    """
    B0=sp.Matrix(B0)
    n=B0.shape[0]
    symmetric_group=Permutations(range(n))
    I=sp.eye(n)
    seed_dict={():(B0,I)}
    period_dict=dict()
    current_key_list=[()]
    incoherent_list=[]
    for depth in range(l):
        ######### When current_key_list becomes a empty list, break the loop. 
        next_key_list=[]
        for key_pre in current_key_list:
            seed_pre=seed_dict[key_pre]
            B_pre=seed_pre[0]
            C_pre=seed_pre[1]
            for k in range(1,n+1):
                if key_pre==() or key_pre[-1]!=k:
                    key_new=key_pre+(k,)
                    C_new=mut_C(C_pre,B_pre,k)
                    B_new=mut_B(B_pre,k)
                    exit_loop=False
                    permutated_C_dict=dict()
                    for p in symmetric_group:
                        permutated_C_dict[p]=Col_Arrangement(C_new,Permutation(p))
                    for w, seed_w in seed_dict.items():
                        C_w=seed_w[1]
                        for p, C in permutated_C_dict.items():
                            if C_w==C:
                                p=[i+1 for i in p]
                                period_dict[key_new]=[p,w]
                                exit_loop=True
                                break
                    if exit_loop==False:
                        next_key_list+=[key_new]
                        seed_new=(B_new,C_new)
                        seed_dict[key_new]=seed_new
                        if check_col_sign_coherence(C_new)==False:
                            incoherent_list+=[key_new]
        if next_key_list==[]:
            return {"C": seed_dict, "period": period_dict, "SIC_list": incoherent_list, "fin": f"finite with maximum depth {depth}", "size": len(seed_dict)}
        else:
            current_key_list=next_key_list
    return {"C": seed_dict, "period": period_dict, "SIC_list": incoherent_list, "fin": "finiteness is undeterminable", "size": len(seed_dict)}
    
        
def B_pattern(B0,l):
    D=distinct_B_dict(B0,l)
    print("B-pattern")
    for key, value in D["B"].items():
        if key==():
            print("[initial]")
        else:
            print(list(key))
        display(value)
    print("=======================================================")
    print("Periodicity")
    for key, value in D["period"].items():
        print(f"permutation {value[0]} * B-matrix at {list(key)} = B-matrix at {list(value[1])}")
    print("=======================================================")
    print("Finiteness")
    print(D["fin"])
    print("=======================================================")
    print("Size")
    print(D["size"])
    
def C_pattern(B0,l):
    D=distinct_BC_dict(B0,l)
    print("C-pattern")
    for key, value in D["C"].items():
        if key==():
            print("[initial]")
        else:
            print(list(key))
        display(value[1])
    print("=======================================================")
    print("Periodicity")
    for key, value in D["period"].items():
        print(f"permutation {value[0]} * C-matrix at {list(key)} = C-matrix at {list(value[1])}")
    print("=======================================================")
    print("Sign-coherence")
    SIC_list=D["SIC_list"]
    if SIC_list==[]:
        print(f"sign-coherent up to {l}")
    else:
        print("sign-incoherent at")
        for w in SIC_list:
            print(list(w))
    print("=======================================================")
    print("Finiteness")
    print(D["fin"])
    print("=======================================================")
    print("Size")
    print(D["size"])
\end{lstlisting}
 


\newpage
\begin{thebibliography}{99}
\newcommand{\au}[1]{\textrm{#1},}
\newcommand{\ti}[1]{\textrm{#1},}
\newcommand{\jo}[1]{\textit{#1}}
\newcommand{\vo}[1]{\textbf{#1}}
\newcommand{\yr}[1]{(#1)}
\newcommand{\pp}[2]{#1--#2.}
\newcommand{\arxiv}[1]{\href{http://arxiv.org/abs/#1}{arXiv:#1}}




\bibitem[AC25]{AC25}
\au{R. Akagi, Z. Chen}
\ti{Sign-coherence and tropical sign pattern for rank 3 real cluster-cyclic exchange matrices}
\jo{arXiv: 2509.07454} \yr{2025}.
\bibitem[Aka26]{Aka26}
\au{R. Akagi}
\ti{{C-matrix-mutations-and-sign-coherence-in-cluster-algebras}}
ver.~1.0.0,
\url{https://github.com/RyotaAkagi/C-matrix-mutations-and-sign-coherence-in-cluster-algebras},
\yr{2026}.
\bibitem[BBH11]{BBH11}
\au{A. Beineke, T. Br\"ustle, L. Hille}
\ti{Cluster-cyclic quivers with three vertices and the Markov equation}
\jo{Algebr. Represent. Theory}
{\bf 14} \yr{2011}, no.~1, \pp{97}{112} MR2763295
\bibitem[BGZ06]{BGZ06}
\au{M. Barot, C. Geiss, A. Zelevinsky}
\ti{Cluster algebras of finite type and positive symmetrizable matrices}
\jo{J. London Math. Soc. }
\vo{73}, \yr{2006}, no. 3, \pp{545}{564} MR2241966
\bibitem[BHIT17]{BHIT17}
\au{T. Br{\"u}stle, S. Hermes, K. Igusa, G. Todorov}
\ti{Semi-invariant pictures and two conjectures on maximal green sequences}
\jo{J. Algebra}
\vo{473}, \yr{2017}, \pp{80}{109} MR3591142
\bibitem[DP24]{DP24}
\au{D.~D. Duffield, P. Tumarkin}
\ti{Categorifications of non-integer quivers: types $H_4$, $H_3$ and $I_2(2n+1)$}
\jo{Represent. Theory}
\vo{28} \yr{2024},\pp{275}{327} MR4806405
\bibitem[DP25]{DP25}
\au{D.~D. Duffield, P. Tumarkin}
\ti{Categorifications of Non-Integer Quivers: Type $I_2(2n)$}
\jo{Algebr. Represent. Theory}
\vo{\bf 28} \yr{2025} no.~3, \pp{787}{840} MR4927767
\bibitem[DWZ10]{DWZ10}
\au{H. Derksen, J. Weyman, A. Zelevinsky} 
\ti{Quivers with potentials and their representations II: applications to cluster algebras}
\jo{J. Amer. Math. Soc.} \vo{23(3)} \yr{2010}, \pp{749}{790} MR2629987
\bibitem[EJLN24]{EJLN24}
\au{T. J. Ervin, B. Jackson, K. Lee, S. D. Nguyen} 
\ti{Geometry of $c$-vectors and $c$-matrices for mutation-infinite quivers}
\jo{arXiv:2410.08510} \yr{2024}.
\bibitem[GHKK18]{GHKK18}
\au{M. Gross, P. Hacking, S. Keel, M. Kontsevich}
\ti{Canonical bases for cluster algebras}
\jo{J. Amer. Math. Soc.}
\vo{31}, \yr{2018}, \pp{497}{608} MR3758151
\bibitem[FG19]{FG19}
\au{S. Fujiwara, Y. Gyoda}
\ti{Duality between final-seed and initial-seed mutations in cluster algebras}
\jo{SIGMA Symmetry Integrability Geom. Methods Appl} 
\vo{\bf 15} \yr{2019} Paper No. 040, 24pp. MR3950163
\bibitem[FL23]{FL23}
\au{A. Felikson, P. Lampe}
\ti{Exchange graphs for mutation-finite non-integer quivers} \jo{J. Geom. Phys.} 
\vo{188} \yr{2023},
104811. MR4572447 
\bibitem[FR07]{FR07}
\au{S. Fomin, N. Reading}
\ti{Root systems and generalized associahedra, in {\it Geometric combinatorics}}
\yr{2007}
\pp{63}{131} {IAS/Park City Math. Ser.},  \vo{13} \jo{Amer. Math. Soc. Providence, RI} MR2383126
\bibitem[FST12]{FST12}
\au{A. Felikson, M. Shapiro, P. Tumarkin} \ti{Cluster algebras and triangulated orbifolds} \jo{Adv. Math.} \vo{231} \yr{2012}, \pp{2953}{3002} MR2970470
\bibitem[FT19]{FT19}
\au{A. Felikson, P. Tumarkin} 
\ti{Geometry of mutation classes of rank 3 quivers} 
\jo{Arnold Math. J.}
\vo{5} \yr{2019} no.~1, \pp{37}{55} MR3981452
\bibitem[FT23]{FT23}
\au{A. Felikson, P. Tumarkin}
\ti{Mutation-finite quivers with real weights}
\jo{Forum Math. Sigma {\bf 11}}
\yr{2023}, Paper No. e9, 22 pp. MR4549712
\bibitem[FWZ16]{FWZ16}
\au{S. Fomin, L. Williams, A. Zelevinsky} \ti{Introduction to Cluster Algebras Chapters 4–5} \jo{arXiv:1707.07190} \yr{2016}.
\bibitem[FZ02]{FZ02}
\au{S. Fomin, A. Zelevinsky}
\ti{Cluster Algebra I: Foundations}
\jo{J. Amer. Math. Soc.}
\vo{15} \yr{2002}, \pp{497}{529} MR1887642 
\bibitem[FZ03a]{FZ03a}
\au{S. Fomin, A. Zelevinsky}
\ti{Cluster algebras. II. Finite type classification}
\jo{Invent. Math}
\vo{154}, \yr{2003}, no. 1, \pp{63}{121} MR2004457 
\bibitem[FZ03b]{FZ03b}
\au{S. Fomin, A. Zelevinsky}
\ti{$Y$-systems and generalized associahedra}
\jo{Ann. of Math} (2) \vo{158} \yr{2003}, no.~3, \pp{977}{1018} MR2031858
\bibitem[FZ07]{FZ07}
\au{S. Fomin, A. Zelevinsky}
\ti{Cluster Algebra IV: Coefficients}
\jo{Comp. Math.}
\vo{143} \yr{2007}, \pp{63}{121} MR2295199 
\bibitem[GN22]{GN22}
\au{M. Gekhtman, T. Nakanishi}
\ti{Asymptotic sign coherence conjecture}
\jo{Exp. Math.} \vo{31} \yr{2022} no.~2, \pp{497}{505} MR4458128
\bibitem[HM03]{HM03}
\au{D.~C. Handscomb, J.~C. Mason} 
\ti{Chebyshev polynomials}
\jo{Chapman \& Hall} CRC, Boca Raton, FL, \yr{2003} MR1937591
\bibitem[Hum90]{Hum90}
\au{J.~E. Humphreys}
\ti{Reflection groups and Coxeter groups}
\jo{Cambridge Studies in Advanced Mathematics}
29, Cambridge Univ. Press, Cambridge
\yr{1990} MR1066460
\bibitem[Lam18]{Lam18}
\au{P. Lampe}
\ti{On the approximate periodicity of sequences attached to non-crystallographic root systems}
\jo{Exp. Math.}
\vo{27} \yr{2018}, no.~3, \pp{265}{271} MR3857662
\bibitem[LL24]{LL24}
\au{J. Lee, K. Lee} 
\ti{An unexpected property of $\mathbf{g}$-vectors for rank 3 mutation-cyclic quivers}
\jo{arXiv:2409.00599} \yr{2024}.
\bibitem[LS15]{LS15}
\au{K. Lee, R. Schiffler} 
\ti{Positivity for cluster algebras}
\jo{Ann. of Math. (2)}
\vo{182} \yr{2015}, no.~1, \pp{73}{125} MR3374957
\bibitem[MS20]{MS20}
\au{J.~F. McKee, C.~J. Smyth}
\ti{Symmetrizable matrices, quotients, and the trace problem}
\jo{Linear Algebra Appl.}
\vo{600} \yr{2020} \pp{60}{81} MR4090841

\bibitem[Nag13]{Nag13}
\au{K. Nagao}
\ti{Donaldson-Thomas theory and cluster algebras}
\jo{Duke Math. J.}
\vo{7}, \yr{2013}, \pp{1313}{1367} MR3079250
\bibitem[Nak21]{Nak21}
\au{T. Nakanishi}
\ti{Synchronicity phenomenon in cluster patterns}
\jo{J. London Math. Soc.}
\vo{103}, \yr{2021}, \pp{1120}{1152} MR4245832
\bibitem[Nak23]{Nak23}
\au{T. Nakanishi}
\ti{Cluster algebras and scattering diagrams}
\jo{MSJ Mem.} \vo{41} \yr{2023}, 279 pp; ISBN: 978-4-86497-105-8. MR4563311 
\bibitem[NZ12]{NZ12}
\au{T. Nakanishi, A. Zelevinsky}
\ti{On tropical dualities in cluster algebras}
\jo{Contemp. Math.} \vo{565} \yr{2012}, \pp{217}{226} MR2932428 
\bibitem[Pla11]{Pla11}
\au{P. Plamondon} 
\ti{Cluster algebras via cluster categories with infinite-dimensional morphism spaces}
\jo{Compos. Math.} \vo{147} \yr{2011}, \pp{1921}{1954} MR2862067
\bibitem[Rea14]{Rea14}
\au{N. Reading}
\ti{Universal geometric cluster algebras}
\jo{Math. Z.}
\vo{277} \yr{2014} no.~1-2, \pp{499}{547} MR3205782
\bibitem[RS16]{RS16}
\au{N. Reading, D.~E. Speyer}
\ti{Combinatorial frameworks for cluster algebras}
\jo{Int. Math. Res. Not. IMRN}
\yr{2016}, no.~1, \pp{109}{173}; MR3514060

\bibitem[RS20]{RS20}
\au{N. Reading, S. Stella}
\ti{An affine almost positive roots model}
\jo{J. Comb. Algebra}
\vo{4} \yr{2020}, no.~1, \pp{1}{59} MR4073889
\bibitem[Sev11]{Sev11}
\au{A. Seven}
\ti{Cluster algebras and semipositive symmetrizable matrices}
\jo{Trans. Amer. Math.
Soc.} \vo{363} \yr{2011}, \pp{2733}{2762} MR2763735
\bibitem[Sev14]{Sev14}
\au{A. Seven}
\ti{Maximal green sequences of skew-symmetrizable $3\times 3$ matrices}
\jo{Linear Algebra Appl.} \vo{440} \yr{2014}, \pp{125}{130} MR3134258 
\bibitem[Sev15]{Sev15}
\au{A. Seven}
\ti{Cluster algebras and symmetric matrices}
\jo{Proc. Amer. Math.
Soc.} \vo{143}(2) \yr{2015}, \pp{469}{478} MR3973884 
\bibitem[Sev19]{Sev19}
\au{A.~I. Seven}
\ti{Cluster algebras and symmetrizable matrices}
\jo{Proc. Amer. Math. Soc.}
\vo{147} \yr{2019} no.~7, \pp{2809}{2814} MR3973884
\bibitem[Ste13]{Ste13}
\au{S. Stella}
\ti{Polyhedral models for generalized associahedra via Coxeter elements}
\jo{J. Algebraic Combin.}
\vo{38} \yr{2013}, no.~1, \pp{121}{158} MR3070123
\bibitem[War14]{War14}
\au{M. Warkentin} 
\ti{Exchange graphs via quiver mutation} Dissertation \yr{2014}, 103.

\end{thebibliography}
\end{document}